\documentclass[reqno]{amsart}
\input epsf
\address{Simons Center for Geometry and Physics,
State University of New York, Stony Brook, NY 11794-3636 U.S.A.
\& Riken Interdisciplinary Theoretical and Mathematical
Sciences Program (iTHEMS), Wako 351-0198, Japan
} \email{kfukaya@scgp.stonybrook.edu}
\address{Center for Geometry and Physics, Institute for Basic Sciences (IBS), Pohang, Korea \& Department of Mathematics,
POSTECH, Pohang, Korea} \email{yongoh1@postech.ac.kr}
\address{Graduate School of Mathematics,
Nagoya University, Nagoya, Japan} \email{ohta@math.nagoya-u.ac.jp}
\address{Research Institute for Mathematical Sciences, Kyoto University, Kyoto, Japan}
\email{ono@kurims.kyoto-u.ac.jp}

\usepackage[dvipdfmx]{graphicx}
\usepackage{amsmath}
\usepackage{amscd}
\usepackage{amssymb}
\usepackage{amstext}
\usepackage{amsmath}
\usepackage[all]{xy}
\usepackage{yhmath}
\usepackage{mathrsfs}
\usepackage{bbding}
\usepackage{color}

\makeatletter
\def\l@section{\@tocline{1}{0pt}{3mm}{8mm}{}}
\def\l@subsection{\@tocline{2}{0pt}{6mm}{10mm}{}}
\def\l@subsubsection{\@tocline{3}{0pt}{9mm}{11mm}{}}
\makeatother

\def\E{\ifmmode{\mathbb E}\else{$\mathbb E$}\fi} 
\def\N{\ifmmode{\mathbb N}\else{$\mathbb N$}\fi} 
\def\R{\ifmmode{\mathbb R}\else{$\mathbb R$}\fi} 
\def\Q{\ifmmode{\mathbb Q}\else{$\mathbb Q$}\fi} 
\def\C{\ifmmode{\mathbb C}\else{$\mathbb C$}\fi} 
\def\H{\ifmmode{\mathbb H}\else{$\mathbb H$}\fi} 
\def\Z{\ifmmode{\mathbb Z}\else{$\mathbb Z$}\fi} 
\def\P{\ifmmode{\mathbb P}\else{$\mathbb P$}\fi} 
\def\T{\ifmmode{\mathbb T}\else{$\mathbb T$}\fi} 
\def\SS{\ifmmode{\mathbb S}\else{$\mathbb S$}\fi} 
\def\DD{\ifmmode{\mathbb D}\else{$\mathbb D$}\fi} 
\def\K{\ifmmode{\mathbb K}\else{$\mathbb K$}\fi}

\theoremstyle{theorem}
\newtheorem{thm}{Theorem}[section]
\newtheorem{cor}[thm]{Corollary}
\newtheorem{lem}[thm]{Lemma}
\newtheorem{sublem}[thm]{Sublemma}

\newtheorem{prop}[thm]{Proposition}

\theoremstyle{definition}
\newtheorem{defn}[thm]{Definition}
\newtheorem{rem}[thm]{Remark}

\newtheorem{conds}[thm]{Condition}
\newtheorem{conven}[thm]{Convention}

\newtheorem{shitu}[thm]{Situation}

\newtheorem{notation}[thm]{Notation}

\numberwithin{equation}{section}
\makeindex
\begin{document}

\title[Moduli spaces of pseudoholomorphic disks]{
Construction of Kuranishi structures
on the moduli spaces of pseudo-holomorphic disks: II}
\author{Kenji Fukaya, Yong-Geun Oh, Hiroshi Ohta, Kaoru Ono}

\thanks{Kenji Fukaya is supported partially by Simons Collaboration on homological Mirror symmetry, NSF 1406423,
Yong-Geun Oh by the IBS project IBS-R003-D1, Hiroshi Ohta by JSPS Grant-in-Aid
for Scientific Research Nos. 23340015, 15H02054, 
21H00983, 21K18576, 21H00985 and Kaoru Ono by JSPS Grant-in-Aid for
Scientific Research, Nos. 26247006, 23224001, 19H00636.}

\begin{abstract}
This is the second of a series of two articles
in which we provide detailed and
self-contained account of the construction of a system
of Kuranishi structures on the moduli spaces of pseudo-holomorphic
disks.
Using the notion of obstruction bundle data introduced in \cite{diskconst1},
we give a systematic way of constructing a system of
Kuranishi structures on the moduli spaces of pseudo-holomorphic
disks which are compatible at the boundary and corners.
More specifically, it defines a tree-like K-system
in the sense of \cite[Definition 21.9]{foootech21}, 
\cite[Definition 21.9]{Springer}.
The method given in this paper does not only simplify the
description of the constructions in the earlier literature, but also
is designed to provide a systematic utility tool for the construction of
a system of Kuranishi structures in the future research.
We also establish its uniqueness.
\end{abstract}
\maketitle
\par
\date{Aug 17th, 2018}


\maketitle

\tableofcontents
\newpage

\section{Introduction}
\label{sec;intro}
This article is a sequel to \cite{diskconst1}. In the latter article, we gave a
detailed construction of a Kuranishi structure on each \emph{individual} moduli space of pseudo-holomorphic
disks. In the present article we construct a \emph{system} of Kuranishi structures
on the moduli spaces of pseudo-holomorphic disks. More specifically,
we will construct a tree-like K-system as defined in \cite[Definition 21.9]{foootech21}, 
\cite[Definition 21.9]{Springer}.
Some explanation of such a construction has been given already in
the earlier literature such as \cite{fooobook2,foootech} focusing more on its
applications to Lagrangian Floer theory. In this article and \cite{diskconst1}, we aim
at focusing more on explaining minute details of the construction of a tree-like K-system.
The present article adopts terminologies of \cite{diskconst1}.
In particular, systematically using the notion of {\it obstruction bundle data}
introduced in \cite{diskconst1},
we give a systematic way of constructing a system of Kuranishi structures on the moduli spaces of pseudo-holomorphic disks.
We also disseminate the construction into various
pieces so that the outcome of each piece can be stated as an individual theorem
which can be used by other researchers.
Therefore the method we address in this paper does not only simplify the
description of the constructions in the earlier literature, but also
is designed to provide a systematic utility tool for the construction of
a system of Kuranishi structures in the future research.
We also establish uniqueness of the resulting system of Kuranishi structures 
(Theorem \ref{thm219}).
\par
Let $(X,\omega)$ be a compact (or tame) symplectic manifold and $L$ its
compact relatively spin Lagrangian submanifold.
We consider the compactified moduli space $\mathcal M_{k+1}(X,L;\beta)$ of
pseudo-holomorphic disks in $X$ bounding $L$ with $k+1$ boundary marked points
in a given homology class $\beta \in H_2(X,L;\Z)$. Our previous article \cite{diskconst1}
concerns the construction of a Kuranishi structure on this single moduli space individually.

In the present article we consider the whole collection of $\mathcal M_{k+1}(X,L;\beta)$ over
$k, \, \beta \in H_2(X,L;\Z)$ and construct a system of Kuranishi structures thereon
so that the next equality holds as an isomorphism of spaces with Kuranishi structures, i.e.,
of K-spaces:
\begin{equation}\label{form111}
\aligned
\partial\mathcal M_{k+1}(X,L;\beta)
=
\bigcup_{k_1+k_2=k+1}&\bigcup_{i=1,\dots,k_1}\bigcup_{\beta_1+\beta_2=\beta} \\
& \mathcal M_{k_1+1}(X,L;\beta_1) {}_{{\rm ev}_i} \times_{{\rm ev}_0} \mathcal M_{k_2+1}(X,L;\beta_2).
\endaligned
\end{equation}
Here the fiber product is taken over $L$ using the evaluation maps at the $i$-th and the
$0$-th boundary marked points.
Construction of such a system of Kuranishi structures is crucial for the construction
of  an $A_{\infty}$ structure associated to a Lagrangian submanifold.
The issue of compatibility between  various Kuranishi structures on
various moduli spaces is more serious in the case of disks than in the case of closed Riemann surfaces,
for example, in the study of Gromov-Witten invariants (\cite{FO}). This is the reason why the theory of open Gromov-Witten
invariants takes a rather different shape from that of closed Gromov-Witten invariants.
In fact, each $A_{\infty}$ operation $\frak m_{k,\beta}$ itself,
which is defined by  a single
moduli space $\mathcal M_{k+1}(X,L;\beta)$, depends on various choices
involved, and so we have to construct the operators $\frak m_{k,\beta}$ simultaneously
in the way that they satisfy certain compatibility between one another. Only after that the totality of
$\{\frak m_{k,\beta}\}$ forms an $A_{\infty}$ structure that is well-defined up to certain
homotopy equivalence.
\par
In \cite{diskconst1}, we carried out the construction of a Kuranishi structure on
each moduli space $\mathcal M_{k+1}(X,L;\beta)$ in the following order:
\begin{enumerate}
\item[(i)]
We define the notion of  obstruction bundle data for each $\mathcal M_{k+1}(X,L;\beta)$.
\cite[Definition 5.1]{diskconst1}. (See also Section \ref{subsec;obstbundledata} of this paper.)
\item[(ii)]
We prove that we can extract a Kuranishi structure from the obstruction bundle data
in a canonical way at the level of germs. (\cite[Theorem 7.1]{diskconst1})
\item[(iii)]
We prove the existence of  obstruction bundle data.  (\cite[Theorem 11.2]{diskconst1})
\end{enumerate}
In the present paper we perform the construction of a compatible system of
Kuranishi structures in the following order:
\begin{enumerate}
\item
We define the notion of a {\it disk-component-wise system of obstruction bundle data}. (Definition \ref{defn5151}.)
Such a system assigns obstruction bundle data to each $\mathcal M_{k+1}(X,L;\beta)$
for which we require certain compatibility conditions.
\item
We prove that from each disk-component-wise system of obstruction bundle data
we can extract a system of Kuranishi structures that is compatible
with the decomposition of the boundary (\ref{form111}). (Theorem \ref{thm42}.)
\item
We prove the existence of a disk-component-wise system of obstruction bundle data.  (Theorem \ref{thm43}.)
\end{enumerate}
Item (1) is the content of Section \ref{subsec;obstbundledata}.
We first define stratifications
of the moduli spaces and their ambient `sets' in Section \ref{subsec;stratifi}.
The stratifications are used to define the notion of \emph{compatible system},
called {\it disk-component-wise system}, of obstruction bundle data
in Section \ref{subsec;obstbundledata}.
Item (2) is carried out in Section \ref{subsec;cornercompa}.
Item (3), the proof of existence of a disk-component-wise system,
is technically the most involved one. It is carried out in Sections
\ref{subsec;exi1} and \ref{subsec;exi2}.
In Section \ref{subsec;unique1} we prove that the system of Kuranishi structures is
independent of the choice of the system of obstruction bundle data from which it is extracted.
We previously formulated the notion of pseudo-isotopy between two systems of Kuranishi structures in \cite[Definition 21.19]{foootech21}, \cite[Definition 21.19]{Springer}.
We prove in Theorem \ref{thm96} that if we are given two disk-component-wise systems of obstruction bundle data
then the resulting systems of Kuranishi structures are pseudo-isotopic. (See Definition \ref{def218}.)\footnote{We actaully prove that they are isotopic (Definition \ref{defn913}) which is 
stronger than pseudo-isotopic.}
We also prove the resulting system of Kuranishi structures is independent of the choices of almost complex structures up to pseudo-isotopy,
but depends only on $(X,\omega)$ and $(L,\sigma)$,
where $\sigma$ is a relative spin structure on $L$.
(See Theorems \ref{pisotopyexixsts},\ref{thm219}
and also Corollary \ref{cormainA}. A similar result is proved in \cite{DF} by a slightly 
different method.)

\begin{rem}
This paper and \cite{diskconst1} describe the case of moduli spaces of
pseudo-holomorphic disks.
However many of the constructions of this paper and \cite{diskconst1}
can be {easily} adapted to the case of other moduli spaces of pseudo-holomorphic curves. For example, the construction in \cite{diskconst1} can be used
to define a Kuranishi structure on the moduli space of marked stable maps (without boundary)
of arbitrary genus.
Therefore we can use it to define Gromov-Witten invariants of arbitrary genus.
The argument of Section \ref{subsec;unique1} of this
paper together with \cite[Corollary 14.28]{foootech2}, \cite[Proposition 14.13]{Springer}
can be used to prove their independence of various choices.
\end{rem}
The way  Kuranishi structure is associated to given
obstruction bundle data provided in this paper and \cite{diskconst1}
is explicit and simple. (See \eqref{Knhd}.)
Therefore when one wants to construct a Kuranishi structure with
certain additional properties, one can do it by finding
obstruction bundle data with corresponding additional properties.
This method is useful in various applications.
The fact that disk-component-wise-ness of  the obstruction bundle
data implies compatibility of Kuranishi structures with
the isomorphism \eqref{form111} (Theorem \ref{thm42}) is just one example.
Other examples where such method can be applied are
the compatibility with the forgetful map, the
compatibility with a group action on the target space
or anti-symplectic involution, cyclic symmetry of
boundary marked points and etc.

\par
\smallskip\noindent
{\it Acknowledgments:}
KF, HO and KO thank IBS Center for Geometry and Physics for hospitality
where a part of this work was done.
The authors thank the anonymous referee for careful reading and useful comments.

\section{Statement of the results}
\label{subsec;discsstatement}

In this section, we precisely state the main result of this article. (We consider not only boundary marked points but also
interior marked points.)

\begin{shitu}\label{situ271}
$(X,\omega)$ is a symplectic manifold that is tame (at infinity).\footnote{
Namely we assume that there exists a compatible 
almost complex structure which is tame at infinity. See \cite[Definition 4.1.1]{Si}, for example, for the definition of tameness (at infinity). We also fix a connected component of compatible tame almost complex structures.}
$(L,\sigma)$ is a compact Lagrangian submanifold of $X$ equipped with a relative spin
structure $\sigma$.\footnote{
See \cite[Definition 1.6]{fooobook} and \cite[Definition 8.1.2]{fooobook2} for the definition of relative spin structure.}
$J$ is an almost complex structure on $X$, which is tamed by $\omega$.
$\beta \in H_2(X,L;\Z)$. $\diamond$\footnote{The mark $\diamond$ indicates the end of Situation.}
\end{shitu}

\begin{defn}\label{defn24222}
Let $k,\ell \in \Z_{\ge 0}$.
We denote by $\mathcal M_{k+1,\ell}(X,L,J;\beta)$ the set of all
$\sim$ equivalence classes of
$((\Sigma,\vec z,\vec{\frak z}),u)$ with the following properties.
\begin{enumerate}
\item
$\Sigma$ is a genus $0$ bordered curve with one boundary component
that has only (boundary or interior) nodal singularities.
\item
$\vec z = (z_0,z_1,\dots,z_k)$ is a $(k+1)$-tuple of boundary marked points.
We assume that they are distinct and are not nodal.
Moreover we assume that they are numbered so that it respects the counter-clockwise cyclic
ordering of the boundary.
\item
$\vec{\frak z} = ({\frak z}_1,\dots,{\frak z}_{\ell})$ are $\ell$ interior
marked points. We assume that they are distinct and are not nodal.
\item
$u : (\Sigma,\partial\Sigma) \to (X,L)$ is a continuous map
which is pseudo-holomorphic on each irreducible component.
The homology class  $u_*([\Sigma,\partial\Sigma])$  is
$\beta$.
\item
$((\Sigma,\vec z,\vec{\frak z}),u)$ is stable in the sense of
Definition \ref{stabilitydefn27} below.
\end{enumerate}
We define the equivalence relation $\sim$ as follows.
$((\Sigma,\vec z,\vec{\frak z}),u) \sim ((\Sigma',\vec z^{\,\prime},\vec{\frak z'}),u')$
if there exists a homeomorphism $v : \Sigma \to \Sigma'$ such that
\begin{enumerate}
\item[(i)]
$v$ is biholomorphic on each irreducible component.
\item[(ii)]
$u' \circ v = u$.
\item[(iii)]
$v(\vec z) = \vec z^{\,\prime}$. $v(\vec{\frak z}) = \vec{\frak z}'$.
\end{enumerate}
Here and hereafter $v(\vec{\frak z}) = \vec{\frak z}$
(resp. $v(\vec z) = \vec z$)
means
$v({\frak z}_i) = {\frak z}_i$ (resp. $v(z_i) = z_i$).
\par
In case $\ell = 0$, we write $\mathcal M_{k+1}(X,L,J;\beta)$
in place of $\mathcal M_{k+1,0}(X,L,J;\beta)$.
\end{defn}
\begin{defn}\label{stabilitydefn27}
Suppose $((\Sigma,\vec z,\vec{\frak z}),u)$ satisfies
Definition \ref{defn24222} (1)(2)(3)(4).
The group ${\rm Aut}^+((\Sigma,\vec z,\vec{\frak z}),u)$
of its {\it extended automorphisms} consists
of homeomorphisms $v : \Sigma \to \Sigma$ such that:
\begin{enumerate}
\item[(i)]
$v$ is biholomorphic on each of the irreducible components.
\item[(ii)]
$u \circ v = u$.
\item[(iii)]
$v(\vec z) = \vec z$ and there exists $\sigma \in {\rm Perm}(\ell)$
such that $(v({\frak z}_1),\dots,v({\frak z}_{\ell}))$
coincides with $({\frak z}_{\sigma(1)},\dots,{\frak z}_{\sigma(\ell)})$.
Here ${\rm Perm}(\ell)$ is the group of permutations 
of the set $\{1,\dots,\ell\}$.
\end{enumerate}
(iii) defines a group homomorphism
${\rm Aut}^+((\Sigma,\vec z,\vec{\frak z}),u)
\to {\rm Perm}(\ell)$.
The group of {\it automorphisms}
${\rm Aut}((\Sigma,\vec z,\vec{\frak z}),u) $ is its kernel.
\par
The object $((\Sigma,\vec z,\vec{\frak z}),u)$ is said
to be {\it stable} if ${\rm Aut}^+((\Sigma,\vec z,\vec{\frak z}),u)$ is a finite group.
\end{defn}
\begin{notation}
For ${\bf p} \in \mathcal M_{k+1,\ell}(X,L,J;\beta)$,
we denote its representative by
$$
((\Sigma_{\bf p},\vec z_{\bf p},\vec{\frak z}_{\bf p}),u_{\bf p}).
$$
\end{notation}

In \cite[Definition 7.1.42 and Theorem 7.1.43]{fooobook2},
we defined a topology on $\mathcal M_{k+1,\ell}(X,L,J;\beta)$
and proved that it is compact and Hausdorff with respect to this
topology.
We call this topology the {\it stable map topology}.
The main result of this paper is as follows.

\begin{thm}\label{therem273}
In Situation \ref{situ271},
there exists a  tree-like K-system\footnote{
In the previous literature we used 
the terminology ``$A_{\infty}$ correspondence'', which is the same as ``tree-like K-system''.} whose moduli spaces of operations are $\mathcal M_{k+1}(X,L,J;\beta)$.
\end{thm}

A tree-like K-system is  defined in Definition \ref{defn2167}.
More precisely it is the $\ell =0$ case of Definition \ref{defn2167}.
See \cite[Definition 21.9]{foootech21}, \cite[Definition 21.9]{Springer}.
The notion of  the moduli spaces of operations of a tree-like K-system 
is defined in \cite[Condition 21.6 (III)]{Springer}.

We also remark that Theorems \ref{therem273} and \cite[Theorem 21.35 (1)]{foootech21},
\cite[Theorem 21.35 (1)]{Springer}
(using an algebraic  lemma \cite[Theorem 5.4.2]{fooobook}) imply
the following:
\begin{cor}{\rm(\cite[Theorem A]{fooobook})}\label{ThemAinfooobook}
To each $(X,\omega,J)$ and $(L,\sigma)$ as in Situation \ref{situ271},
we can associate a filtered $A_{\infty}$ structure
on $H(L;\Lambda^{\R}_{\rm 0,nov})$.
\end{cor}
We can prove well-defined-ness of the  tree-like K-system 
in Theorem \ref{therem273} and also its independence of the
choice of  almost complex structure $J$ as follows.
\begin{shitu}\label{familyamcoxstru}
Let $(X,\omega), (L,\sigma)$ be as in Situation \ref{situ271}.
Suppose that $J_1$ and $J_2$ are  compatible and  tame almost complex structures on
$X$.
We take a  smooth family of  compatible and  tame almost complex structures $\mathcal J
= \{J_s\}$, $s\in[1,2]$
joining $J_1$ and $J_2$. $\diamond$
\end{shitu}
\begin{thm}\label{pisotopyexixsts}
Suppose we are in Situation \ref{familyamcoxstru}.
By Theorem \ref{therem273} we have  a
tree-like K-system  whose
moduli spaces of operations are
$\mathcal M_{k+1}(X,L,J_j;\beta)$,
for $j=1,2$.
\par
Then there exists a pseudo-isotopy of
tree-like K-systems
between them in the sense of
\cite[Definition 21.15]{foootech21}
\cite[Definition 21.15]{Springer}.
The moduli space of a $[1,2]$-parameterized family of
$A_{\infty}$ operations are given by
$$
\mathcal M_{k+1}(X,L,\mathcal J;\beta;[1,2])
=
\bigcup_{s \in [1,2]} \{s\} \times  \mathcal M_{k+1}(X,L,J_s;\beta)
.
$$
\end{thm}
Theorems \ref{pisotopyexixsts} and \cite[Theorem 21.35 (3)]{foootech21}, 
\cite[Theorem 21.35 (3)]{Springer}
(using the algebraic lemma \cite[Theorem 5.4.2]{fooobook} again) imply the following.
\begin{cor}{\rm(\cite[Theorem A]{fooobook})}\label{cormainA}
The filtered $A_{\infty}$ structure
on $H(L;\Lambda^{\R}_{\rm 0,nov})$
given in Corollary \ref{ThemAinfooobook}
depends only on $(X,\omega)$ and $(L,\sigma)$ up to isomorphism.
In particular, it is independent of the
choices of almost complex structures and of obstruction bundle data up to isomorphism.
\end{cor}
\begin{rem}
$ $
\begin{enumerate}
\item
The isomorphism of filtered $A_{\infty}$ structure
in Corollary \ref{cormainA} means a
filtered $A_{\infty}$ homomorphism that has an inverse.
Note it may not be linear.
(In that sense it is called sometimes a quasi-isomorphism.)
\item
\cite[Theorem A]{fooobook} is mostly the same result
as Corollaries \ref{ThemAinfooobook} and \ref{cormainA},
but the ground ring in \cite[Theorem A]{fooobook} is $\Q$.
We used singular homology to construct a filtered $A_{\infty}$
structure over $\Q$ coefficients
in place of $\R$ coefficients.
\end{enumerate}
\end{rem}
\par
We can also prove the version with interior marked points.
To state this version we need to prepare some notations.
We first recall the following:
\begin{defn}\label{defnGEM}
We put $\frak G = H_2(X,L;\Z)$,
$E(\beta) = \omega(\beta)$ and denote $\mu(\beta)$ by the Maslov index
associated to $\beta \in \frak G$.
\end{defn}
\begin{defn}\label{defn192}
A {\it decorated rooted ribbon tree}
\index{decorated rooted ribbon tree}
\index{ribbon tree ! {\it see: decorated rooted ribbon tree}}
is a pair $(\mathcal T,\beta(\cdot))$
such that:
\begin{enumerate}
\item $\mathcal T$ is a connected tree.
Let $C_0(\mathcal T)$, $C_1(\mathcal T)$ be the sets of
all vertices and edges of $\mathcal T$, respectively.
\item
For each ${\rm v} \in C_0(\mathcal T)$
we fix a cyclic order
of the set of edges containing ${\rm v}$.
This is equivalent to fixing an isotopy type of an embedding
of $\mathcal T$ to the plane $\R^2$.
(Namely, the cyclic order of the edges is given by the orientation
of the plane so that the edges are enumerated according to the counter clockwise orientation. We call it a {\it ribbon structure} at the vertex ${\rm v}$.)
\item
$C_0(\mathcal T)$ is divided into the set of {\it exterior vertices}
$C_{0}^{{\rm ext}}(\mathcal T)$ and
the set of {\it interior vertices} $C_{0}^{{\rm int}}(\mathcal T)$.
\item
We fix one element of $C_{0}^{{\rm ext}}(\mathcal T)$, which we call the
{\it root}.
\item
The valency of all the exterior vertices are $1$.\footnote{
A vertex of valency $1$ may not be exterior.}
\item
$\beta(\cdot) : C_{0}^{{\rm int}}(\mathcal T) \to \frak G$ is a map.
We require $E(\beta({\rm v})) \ge 0$.
Moreover if $E(\beta({\rm v})) = 0$ then $\beta({\rm v})$ is
required to be $0 \in H_2(X,L;\Z)$.
\item (Stability)
For each ${\rm v} \in C_{0}^{{\rm int}}(\mathcal T)$
we assume that one of the following holds.
\begin{enumerate}
\item
$E(\beta({\rm v})) > 0$.
\item
The valency of ${\rm v}$ is not smaller than $3$.
\end{enumerate}
\end{enumerate}
We denote by $\mathcal G(k+1,\beta)$ the set of all decorated rooted
ribbon trees $(\mathcal T,\beta(\cdot))$ such that:
\begin{enumerate}
\item[(I)]
$\# C_{0}^{{\rm ext}}(\mathcal T) = k+1$.
\item[(II)]
$\sum_{{\rm v} \in C_{0}^{{\rm int}}(\mathcal T)}(\beta({\rm v})) = \beta.$
\end{enumerate}
\end{defn}

We decompose the set of edges $C_{1}(\mathcal T)$ into two types.
If an edge ${\rm e}$ contains an exterior vertex,
we call ${\rm e}$ an {\it exterior edge}.
Otherwise we call
${\rm e}$ an {\it interior edge}.
We denote by $C_{1}^{{\rm int}}(\mathcal T)$,
(resp. $C_{1}^{{\rm ext}}(\mathcal T)$) the set of all interior (resp.
exterior) edges.

Now we add interior marked points to $\mathcal G(k+1,\beta)$
and define the following set.

\begin{defn}\label{defn274}
The set $\mathcal G(k+1,\ell,\beta)$
consists of objects
$$
\frak T = (\mathcal T,\beta(\cdot),l(\cdot))
$$
with the following properties.
We call $\frak T$ a {\it marked decorated rooted ribbon tree}.
\begin{enumerate}
\item
$(\mathcal T,\beta(\cdot))$ satisfies
Definition \ref{defn192} (1)-(6),
which is a  part of the definition of the set
$\mathcal G(k+1,\beta)$.
\item $l({\rm v}) \subset \{1,\dots,\ell\}$
such that $\{1,\dots,\ell\}$ is a disjoint union of
$\{l({\rm v}) \mid {\rm v} \in C_{0}^{{\rm int}}(\mathcal T)\}$.
\item
Instead of Definition \ref{defn192} (7)
we assume the following.
For each ${\rm v} \in C_{0}^{{\rm int}}(\mathcal T)$
we assume that one of the following holds.
\begin{enumerate}
\item
$E(\beta({\rm v})) > 0$.
\item
The valency of ${\rm v}$ is not smaller than $3$.
\item
$l({\rm v}) \ne \emptyset$.
\end{enumerate}

\end{enumerate}
\end{defn}
Using this, we incorporate the data of interior marked points into the definition of tree-like K-system (\cite[Definition 21.9]{foootech21}, \cite[Definition 21.9]{Springer}) as in
Theorem \ref{therem274}.
We also study isotopy etc. between them. For that purpose we include the
parametrized version in the next definition.
\begin{shitu}\label{situ213}
Let $(X,\omega), L$ be as in Situation \ref{situ271}
and $P$ a smooth compact oriented manifold with corners.
We consider $\mathcal J = \{J_t \mid t \in P\}$,
the smooth family of tame almost complex structures on $X$
parametrized by $P$.
$\diamond$
\end{shitu}
\begin{defn}
We put
$$
\mathcal M_{k+1,\ell}(X,L,\mathcal J;\beta)
=
\bigcup_{t \in P} \{t\}  \times \mathcal M_{k+1,\ell}(X,L,J_t;\beta).
$$
We can define a stable map topology on this space in the same way as the
case when $P$ is a point.
There exists a map
$$
{\rm ev}_P : \mathcal M_{k+1,\ell}(X,L,\mathcal J;\beta) \to P
$$
which sends
$\{t\} \times \mathcal M_{k+1,\ell}(X,L,J_t;\beta)$ to $t$.
\end{defn}

\begin{thm}\label{therem274}
In Situation \ref{situ213}, there exists a system of Kuranishi structures on
$\{\mathcal M_{k+1,\ell}(X,L,\mathcal J;\beta) \mid k, \ell, \beta\}$
with the following properties:
\par\smallskip
\noindent {\rm (I)}
$\frak G = H_2(X,L;\Z)$,
$E : \frak G \to \R$ and $\mu : \frak G \to \Z$ are
as in Definition \ref{defnGEM}.
\par\smallskip
\noindent {\rm (II)}
Nothing to add.\footnote{Each item (I)-(X) corresponds to
the item of \cite[Condition 21.11]{foootech21}, 
\cite[Condition 21.11]{Springer} with the same number.
We leave item (II) void for this consistency.}
\par\smallskip
\noindent {\rm (III)}
$$
\aligned
({\rm ev}_P,{\rm ev},{\rm ev}^{\rm int}) =
&({\rm ev}_P,({\rm ev}_0,\dots,{\rm ev}_k),({\rm ev}^{\rm int}_1,\dots,{\rm ev}^{\rm int}_{\ell})) \\
&: {\mathcal M}_{k+1,\ell}(X,L,\mathcal J;\beta) \to P \times L^{k+1} \times X^{\ell}
\endaligned
$$
is a strongly smooth map such that
$({\rm ev}_P,{\rm ev}_0)$ is weakly submersive.

\par\smallskip
\noindent {\rm (IV)} {\bf (Positivity of energy)}
We assume  $
{\mathcal M}_{k+1,\ell}(X,L,\mathcal J;\beta) = \emptyset$
if $E(\beta) < 0$.
\par\smallskip
\noindent {\rm (V)} {\bf (Energy zero part)}
In case $E(\beta) = 0$, we have $
{\mathcal M}_{k+1,\ell}(X,L,\mathcal J;\beta) = \emptyset$ unless $\beta = 0$.
In case
$\beta = \beta_0 = 0$, we assume
${\mathcal M}_{k+1,\ell}(X,L,\mathcal J;\beta_0) = \emptyset$ if
$k+1+2\ell < 3$ and
${\mathcal M}_{k+1,\ell}(X,L,\mathcal J;\beta_0) = P \times L \times
{\mathcal M}_{k+1,\ell}$ otherwise.
Here ${\mathcal M}_{k+1,\ell}$ is the compactified moduli space
of stable marked bordered curve of genus $0$ with one boundary component,
$\ell$ interior marked points and $k+1$ boundary marked points.
\par\smallskip
\noindent {\rm (VI)} {\bf (Dimension)}
The dimension of the moduli space $\mathcal M_{k+1,\ell}(X,L,\mathcal J;\beta)$ is given by
\begin{equation}
\dim
\mathcal M_{k+1,\ell}(X,L,\mathcal J;\beta)
=
\mu(\beta) + \dim L + k -2 + 2\ell + \dim P.
\end{equation}
\par\smallskip
\noindent {\rm (VII)} {\bf (Orientation)}
${\mathcal M}_{k+1,\ell}(X,L,\mathcal J;\beta)$ is oriented.
\par\smallskip
\noindent {\rm (VIII)} {\bf (Gromov compactness)}
For any $E_0$ the set
\begin{equation}
\{\beta \in \frak B \mid \exists k \,\, \exists \ell\,\,\,
{\mathcal M}_{k+1,\ell}(X,L,\mathcal J;\beta)
\ne \emptyset,
\,\, E(\beta) \le E_0\}
\end{equation}
is a finite set.
\par\smallskip
\noindent {\rm (IX)} {\bf (Compatibility at the boundary)}
The normalized boundary\footnote{
See Remark \ref{rem217} and reference therein for the notion of 
a normalized boundary and normalized corners.}
of
${\mathcal M}_{k+1,\ell}(X,L,\mathcal J;\beta)$
is decomposed to the disjoint union of fiber products as follows.
\begin{equation}\label{formula1660rev}
\aligned
&\widehat\partial {\mathcal M}_{k+1,\ell}(X,L,\mathcal J;\beta) \\
&\cong
\coprod_{\beta_1,\beta_2,k_1,k_2,i,l_1,l_2} (-1)^{\epsilon}
{\mathcal M}_{k_1+1,\#l_1}(X,L,\mathcal J;\beta_1) \\
&\qquad\qquad\qquad\qquad\qquad\qquad\,\, {}_{({\rm ev}_P,{\rm ev}_i)}\times_{({\rm ev}_P,{\rm ev}_0)}
{\mathcal M}_{k_2+1,\#l_2}(X,L,\mathcal J;\beta_2)
\\
&\qquad\qquad\cup {\mathcal M}_{k+1,\ell}(X,L,\mathcal J\vert_{\widehat\partial P};\beta)
\endaligned
\end{equation}
where the union of the first term of the right hand side is taken over
$\beta_1,\beta_2,k_1,k_2$ such that $\beta_1 + \beta_2 = \beta$,
$k_1 + k_2 = k+1$, $i=1,\dots,k_1$
and $l_1,l_2$ that are subsets of $\{1,\dots,\ell\}$
such that $l_1 \cap l_2 = \emptyset$,
$l_1 \cup l_2 =\{1,\dots,\ell\}$.
$\mathcal J\vert_{\widehat\partial P}$ is the restriction of the family
$\mathcal J$ to the normalized boundary of $P$.
\par
We refer \cite[(21.15),(21.16)]{foootech21},
\cite[(21.15),(21.16)]{Springer} for the description of the sign ${\epsilon}$. 
(c.f. \cite[Section 8.5]{fooobook2}) See also \cite[Remark 16,2]{foootech21}, \cite[Remark 16,2]{Springer}
for the order of fiber product and orientation etc.
\par
This isomorphism is compatible with orientation.
It is compatible also with the evaluation maps. (Compatibility with the
evaluation maps at the boundary marked points is the same
as that of \cite[Formula (21.8)]{foootech21}, \cite[Formula (21.8)]{Springer}.
Compatibility at the interior marked points can be
formulated in the similar way by using the indexing set
$l_1 \sqcup l_2 =\{1,\dots,\ell\}$.)
\par\smallskip
\noindent {\rm (X)} {\bf (Corner compatibility isomorphism)}
Let $\widehat{S}_m( {\mathcal M}_{k+1,\ell}(X,L,\mathcal J;\beta))$ be the
normalized corner of the K-space\footnote{Following
\cite[Definition 3.11]{foootech2}, \cite[Definition 3.11]{Springer}, 
we use the terminology `K-space' as a
paracompact metrizable space equipped with a Kuranishi structure.}
${\mathcal M}_{k+1,\ell}(X,L,\mathcal J;\beta)$
in the sense of \cite[Definition 24.17]{foootech21}, \cite[Definition 24.18]{Springer}.
Then it is isomorphic to the
disjoint union of
\begin{equation}\label{cornecomAinf1rev}
\prod_{(\mathcal T,\beta(\cdot),l(\cdot))} {\mathcal M}_{k_{\rm v}+1,
\#l({\rm v})}(X,L,\mathcal J\vert_{\widehat S_{m'}P};\beta({\rm v})).
\end{equation}
Here the union is taken over all
$(\mathcal T,\beta(\cdot),l(\cdot))
\in \mathcal G(k+1,\ell,\beta)$, $m'$
such that $\#  C_{1}^{{\rm int}}(\mathcal T) + m'= m$,
and
$\prod_{(\mathcal T,\beta(\cdot),l(\cdot))}$ means a fiber product
defined in Definition \ref{defn3131}.
$\mathcal J\vert_{\widehat S_{m'}P}$ is the restriction
of the family $\mathcal J$ to the codimension $m'$
normalized corner $\widehat S_{m'}P$ of $P$.
We call the isomorphism
\begin{equation}\label{cornercompiso}
\aligned
& \widehat{S}_m( {\mathcal M}_{k+1,\ell}(X,L,\mathcal J;\beta)) \\
& \xrightarrow{\sim}
\coprod_{\substack{(\mathcal T,\beta(\cdot),l(\cdot))
\in \mathcal G(k+1,\ell,\beta),\\m'}}
\prod_{(\mathcal T,\beta(\cdot),l(\cdot))} {\mathcal M}_{k_{\rm v}+1,
\#l({\rm v})}(X,L,\mathcal J\vert_{\widehat S_{m'}P};\beta({\rm v}))
\endaligned
\end{equation}
{\rm the corner compatibility isomorphism}.
It is compatible with the evaluation maps.
\par\smallskip
\noindent {\rm (XI)} {\bf (Consistency of corner compatibility isomorphisms)}
We iterate the construction of normalized corner and
obtain a space $\widehat{S}_{m_2}(\widehat{S}_{m_1}( {\mathcal M}_{k+1,\ell}(X,L,\mathcal J;\beta)))$.
Condition (X) implies that
$\widehat{S}_{m_2}(\widehat{S}_{m_1}( {\mathcal M}_{k+1,\ell}(X,L,\mathcal J;\beta)))$
is a disjoint union of $(m'+m_1+m_2)!/m'!m_1!m_2!$ copies
of (\ref{cornecomAinf1rev}),
where the union is taken over all
$(\mathcal T,\beta(\cdot),l(\cdot))
\in \mathcal G(k+1,\ell,\beta)$
such that $\#  C_{1}^{{\rm int}}(\mathcal T) = m_1+m_2$,
$m = m' + m_1+m_2$.
\par
The map
$$
\pi_{m_2,m_1} :
\widehat{S}_{m_2}(\widehat{S}_{m_1}( {\mathcal M}_{k+1,\ell}(X,L,\mathcal J;\beta)))
\to \widehat{S}_{m_1+m_2}( {\mathcal M}_{k+1,\ell}(X,L,\mathcal J;\beta))
$$
in \cite[Proposition 24.16]{foootech21}, \cite[Proposition 24.17]{Springer}
is identified with the identity map
on each component of (\ref{cornecomAinf1rev}).
\par\smallskip
\noindent {\rm (XII)}  {\bf (Exchange symmetry of the interior marked points)}
There exists an action\footnote{See \cite[Subsection 24.4]{foootech21},
\cite[Chapter 24.4]{Springer}
for the definition of finite group action on K-spaces.} of symmetric group ${\rm Perm}(\ell)$
on the K-space ${\mathcal M}_{k+1,\ell}(X,L,\mathcal J;\beta)$, whose underlying action
on the topological space ${\mathcal M}_{k+1,\ell}(X,L,\mathcal J;\beta)$
is by exchanging the interior marked points.
This action is compatible with the evaluation map and the
corner compatibility isomorphisms given in (X)(XI).
It is orientation preserving.
\end{thm}
\begin{rem}\label{rem217}
The notion of a normalized boundary and normalized corners 
of a manifold (or an orbifold) with corners 
are defined in \cite[Definitions 8.4 and 24.18]{Springer}
\cite[Lemma-Definition 8.8]{Springer}.
For example the normalized boundary of the subspace 
$\{(x,y) \mid x,y \in \R, x,y\ge 0\}$ of $\R^2$ is 
a {\it disjoint} union of two half lines 
$\{ x \mid x \ge 0\} \sqcup \{ y \mid y \ge 0\}$.
Two points of the normalized boundary, $x=0$ in the first summand and $y=0$ in the second summand,  
become the same point in the usual boundary.
A normalized boundary of a manifold with corners 
becomes a manifold with corners. (This is not the 
case for the usual boundary.)
\end{rem}
\begin{defn}\label{defn2167}
We call the system satisfying (I)-(XII) of Theorem \ref{therem274} with $P=$ point,
a {\it tree-like K-system with interior marked points}.
\end{defn}
Theorem \ref{therem274} implies the existence of
bulk deformations of Lagrangian Floer cohomology.
(See \cite[Subsection 3.8.5]{fooobook}.) Namely it induces
\begin{equation}\label{operatorq}
\frak q :
H(X;\Lambda^{\R}_{\rm 0,nov})
\to HH(H(L;\Lambda^{\R}_{\rm 0,nov}),H(L;\Lambda^{\R}_{\rm 0,nov}))
\end{equation}
from the cohomology of the ambient space $X$ to the
Hochschild cohomology of the filtered $A_{\infty}$ algebra
in Corollary \ref{ThemAinfooobook}.
More precisely (\ref{operatorq}) is a filtered $L_{\infty}$
homomorphism, where the $L_{\infty}$ structure of the domain is
trivial and the $L_{\infty}$ structure of the target is induced by the
Gerstenhaber bracket.
See \cite[Theorem Y and Theorem 3.8.32]{fooobook},
\cite[Corollary 7.4.40]{fooobook2}
for the precise
statement.
There is a similar well-definedness statement as
Theorem \ref{pisotopyexixsts} in the situation of
Theorem \ref{therem274}.
It implies that the map (\ref{operatorq}) is independent of the choices
up to homotopy of filtered $L_{\infty}$ homomorphism.
\begin{rem}
The map $\frak q$ is called the {\it closed-open  map}\footnote{
In the physics literature, it is called the {\it bulk-boundary map}.}.
It is a ring homomorphism when we use quantum cup product on
$H(X;\Lambda^{\R}_{\rm 0,nov})$.
See \cite[Subsection 4.7]{foootoric32} for various works related to this map.
\end{rem}

\begin{defn}\label{def218}
(\cite[Definition 21.15]{foootech21}, \cite[Definition 21.15]{Springer}) 
A {\it pseudo isotopy} between two tree-like K-systems with interior marked points is
a system satisfying (I)-(XII) of Theorem \ref{therem274} with $P=[1,2]$
such that its restriction to $1$ and $2$ becomes the given
two tree-like K-systems with interior marked points.
\end{defn}

\begin{thm}\label{thm219}
Any two tree-like K-systems with interior marked points
obtained by Theorem \ref{therem274} with $P=$ point are pseudo-isotopic.
\end{thm}

The proofs of Theorems \ref{therem273},  \ref{pisotopyexixsts} and \ref{therem274}, \ref{thm219}
occupy the rest of this article.

\section{Obstruction bundle data: Review}
\label{subsec;obsbudat}

We first review the definition of obstruction bundle data
in \cite{diskconst1}.

We considered the set ${\mathcal X}_{k+1,\ell}(X,L;\beta)$
of all isomorphism classes of $((\Sigma,\vec z,\vec {\frak z}),u)$
which satisfy the same condition as in Definition \ref{defn24222}
except we do not require $u$ to be pseudo-holomorphic.
We require $u$ to be continuous and of $C^2$ class on each irreducible component.
(See \cite[Definition 4.2]{diskconst1}.)
This is a set which contains  ${\mathcal M}_{k+1,\ell}(X,L,J;\beta)$
set-theoretically. We {\it emphasize that we do not put structures
on ${\mathcal X}_{k+1,\ell}(X,L;\beta)$ such as topology}.
We use {\it partial topology} of the pair
$({\mathcal X}_{k+1,\ell}(X,L;\beta),{\mathcal M}_{k+1,\ell}(X,L,J;\beta))$:
It assigns $B_{\epsilon}(\mathcal X,{\bf p}) \subset {\mathcal X}_{k+1,\ell}(X,L;\beta)$
to each ${\bf p} \in {\mathcal M}_{k+1,\ell}(X,L,J;\beta)$ and $\epsilon >0$.
Here $B_{\epsilon}(\mathcal X,{\bf p})$ is
the set of
${\bf x} \in {\mathcal X}_{k+1,\ell}(X,L;\beta)$
such that
$\bf x$ is {\it $\epsilon$-close to} $\bf p$.
See \cite[Definition 4.12]{diskconst1} for the definition that
`$\bf x$ is $\epsilon$-close to $\bf p$,'\footnote{One can easily see that this notion is independent of choices of representatives of $\bf x$ and
$\bf p$.} and
\cite[Definition 4.1]{diskconst1} for the definition of the notion of
partial topology and \cite[Subsection 4.3]{diskconst1} for the definition
of the partial topology in the case of the pair
$({\mathcal X}_{k+1,\ell}(X,L;\beta),{\mathcal M}_{k+1,\ell}(X,L,J;\beta))$.
A subset
$$\mathscr U_{\bf p} \subset {\mathcal X}_{k+1,\ell}(X,L;\beta)$$
is called a {\it neighborhood} of ${\bf p}$ in ${\mathcal X}_{k+1,\ell}(X,L;\beta)$ if it contains
$B_{\epsilon}(\mathcal X,{\bf p})$ for sufficiently small $\epsilon > 0$.
Now we recall:

\begin{defn}\label{defn51}{\rm (\cite[Definition 5.1]{diskconst1})}
{\it Obstruction bundle data of (or for) the \index{obstruction bundle data}
moduli space} ${\mathcal M}_{k+1,\ell}(X,L,J;\beta)$
assign to each
${\bf p} \in {\mathcal M}_{k+1,\ell}(X,L,J;\beta)$
a neighborhood $\mathscr U_{\bf p}$ of ${\bf p}$ in ${\mathcal X}_{k+1,\ell}(X,L;\beta)$
and an object
$E_{\bf p}(\bf x)$ to each ${\bf x} \in \mathscr U_{\bf p}$ . We require that they have the following properties.
\begin{enumerate}
\item
We put ${\bf x} = ((\Sigma_{\bf x},\vec z_{\bf x},\vec {\frak z}_{\bf x}),u_{\bf x})$.
Then $E_{\bf p}({\bf x})$
\index{00E1_{p}({\bf x})@$E_{\bf p}({\bf x})$} is a finite dimensional linear subspace of
the set of $C^2$ sections
$$
E_{\bf p}({\bf x}) \subset C^2(\Sigma_{\bf x};u_{\bf x}^*TX \otimes \Lambda^{0,1}),
$$
whose supports are away from nodal or marked points and the boundary.
\item {\bf (Smoothness)}
$E_{\bf p}({\bf x})$ depends smoothly on ${\bf x}$ as defined in
\cite[Definition 8.7]{diskconst1}.
\item {\bf (Transversality)}
$E_{\bf p}({\bf x})$ satisfies the transversality condition as in
\cite[Definition 5.5]{diskconst1}.
\item {\bf (Semi-continuity)}
$E_{\bf p}({\bf x})$ is semi-continuous on  ${\bf p}$ as defined in \cite[Definition 5.2]{diskconst1}.
\item {\bf (Invariance under extended automorphisms)}
$E_{\bf p}({\bf x})$ is invariant under the extended automorphism group of
${\bf x}$ as in \cite[Condition 5.6]{diskconst1}.
\item[(6)]{\bf (Effectivity)}
The action of ${\rm Aut}({\bf p})$ on $(D_{u_{\bf p}}\overline{\partial})^{-1}(E_{\bf p})/ {\frak {aut}} (\Sigma_{\bf p}, \vec z_{\bf p}, \vec {\frak z}_{\bf p})$ is effective.
\end{enumerate}
\end{defn}

In \cite[Theorem 7.1]{diskconst1} we associated a Kuranishi structure
of ${\mathcal M}_{k+1,\ell}(X,L,J;\beta)$ to the 
obstruction bundle data, which is determined canonically in the sense of
germ of Kuranishi structures.

\section{Stratification of the moduli space}
\label{subsec;stratifi}

In the next section, Section \ref{subsec;obstbundledata}, we spell out the condition
of the obstruction bundle data which enables us to relate
those Kuranishi structures on ${\mathcal M}_{k+1,\ell}(X,L,J;\beta)$
one another
on the boundaries and the corners by appropriate fiber products.
In this section,
we describe  stratifications of ${\mathcal M}_{k+1,\ell}(X,L,J;\beta)$,
${\mathcal X}_{k+1,\ell}(X,L;\beta)$,\footnote{This stratification is written in
\cite[Subsection 7.1]{fooobook2}.}
which we use for this purpose.
\par
Let
$(\mathcal T,\beta(\cdot),\ell(\cdot)) \in \mathcal G(k+1,\ell,\beta)$.
We define the fiber product (\ref{cornecomAinf1rev}) as follows.
We first consider the direct product
\begin{equation}\label{form31111}
\prod_{{\rm v} \in C_{0}^{{\rm int}}(\mathcal T)} {\mathcal M}_{k_{\rm v}+1,
\#l({\rm v})}(X,L,J;\beta({\rm v})).
\end{equation}
We will define a map
$$
\mathcal{EV} : \prod_{{\rm v} \in C_{0}^{{\rm int}}(\mathcal T)} {\mathcal M}_{k_{\rm v}+1,
\#l({\rm v})}(X,L,J;\beta({\rm v}))
\to (L \times L)^{\# C_{1}^{{\rm int}}(\mathcal T)}
$$
below. Here the target is a direct product of
$\# C_{1}^{{\rm int}}(\mathcal T)$  copies of $L\times L$.
\par
Let ${\rm e} \in C_{1}^{{\rm int}}(\mathcal T)$
and $\{t({\rm e}),s({\rm e})\} = \partial {\rm e}$.
Here we require the vertex $t({\rm e})$ to be in the same
connected component of $\mathcal T \setminus \{s({\rm e})\}$
as the root $\frak v_0$.
(See Figure \ref{zu1}.)
For each ${\rm v} \in C_{0}^{{\rm int}}(\mathcal T)$ there is a unique edge
${\rm e}_{0}(\rm v) \in C_{0}(\mathcal T)$ adjacent to ${\rm v}$
such that ${\rm e}_{0}(\rm v)$ is contained in the same connected component of
$\mathcal T \setminus {\rm v}$ as the root.
Let ${\rm e}_{1}({\rm v}),\dots, {\rm e}_{k_{\rm v}}({\rm v})$ be the
edges containing ${\rm v}$ such that
$({\rm e}_{0}({\rm v}),{\rm e}_{1}({\rm v}),\dots, {\rm e}_{k_{\rm v}}({\rm v}))$
respects the counterclockwise order induced by the ribbon structure.
By definition
$$
{\rm e} = {\rm e}_0(s({\rm e})).
$$
Let $k_{{\rm e}}$ be a positive integer such that
$$
{\rm e} = {\rm e}_{k_{{\rm e}}}(t({\rm e})).
$$
\begin{figure}[h]
\centering
\includegraphics[scale=0.7]{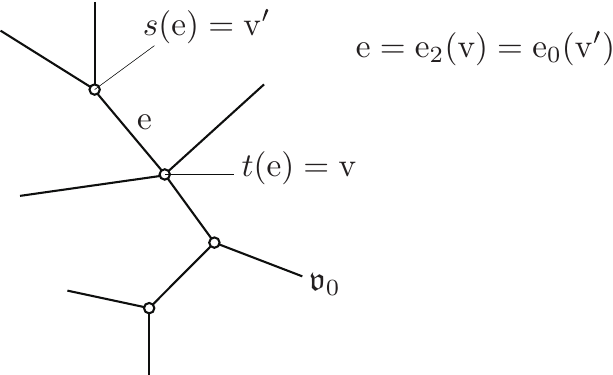}
\caption{$s({\rm e})$ and $t({\rm e})$}
\label{zu1}
\end{figure}
Let $\vec{\bf p} = ({\bf p}_{{\rm v}})_{{\rm v} \in C_{0}^{{\rm int}}}$ be an element of (\ref{form31111}).
We put
$$
\mathcal{EV}_{{\rm e}}(\vec{\bf p})
= ({\rm ev}_{0}({\bf p}_{s({\rm e})}),{\rm ev}_{k_{\rm e}}({\bf p}_{t({\rm e})})) \in L \times L
$$
and
$$
\mathcal{EV} = (\mathcal{EV}_{{\rm e}}(\vec{\bf p}))_{{\rm e} \in
C_{1}^{{\rm int}}(\mathcal T)}.
$$
\begin{defn}\label{defn3131}
We put
$$
\aligned
&\prod_{(\mathcal T,\beta(\cdot),l(\cdot))} {\mathcal M}_{k_{\rm v}+1,
\#l({\rm v})}(X,L,J;\beta({\rm v})) \\
&=
\prod_{{\rm v} \in C_{0}^{{\rm int}}(\mathcal T)} {\mathcal M}_{k_{\rm v}+1,
\#l({\rm v})}(X,L,J;\beta({\rm v}))
{}_{\mathcal{EV}}\times \Delta^{\# C_{1}^{{\rm int}}(\mathcal T)}.
\endaligned
$$
Here $\Delta \subset L \times L$ is the diagonal.
In other words, it is the set of all elements
$\vec{\bf p} = ({\bf p}_{{\rm v}})_{{\rm v} \in C_{0}^{{\rm int}}}({\mathcal T})$ of (\ref{form31111})
such that
\begin{equation}\label{equation3230}
{\rm ev}_{0}({\bf p}_{s({\rm e})}) = {\rm ev}_{k_{\rm e}}({\bf p}_{t({\rm e})})
\end{equation}
for all ${\rm e} \in C_{1}^{{\rm int}}(\mathcal T)$.
To simplify the notation we write
\begin{equation}\label{eq43}
{\mathcal M}_{k+1,\ell}(X,L,J;\beta)(\frak T) :=
\prod_{(\mathcal T,\beta(\cdot),l(\cdot))} {\mathcal M}_{k_{\rm v}+1,
\#l({\rm v})}(X,L,J;\beta({\rm v}))
\end{equation}
with $\frak T = (\mathcal T,\beta(\cdot),l(\cdot))$.
\end{defn}
\begin{lem}\label{lem3232}
There exists a homeomorphism onto its image
\begin{equation}\label{stratumemb}
{\mathcal M}_{k+1,\ell}(X,L,J;\beta)(\frak T) \to {\mathcal M}_{k+1,\ell}(X,L,J;\beta).
\end{equation}
\end{lem}
\begin{proof}
Let $\vec{\bf p} = ({\bf p}_{{\rm v}})_{{\rm v} \in C_{0}^{{\rm int}}}(\mathcal T)$ be an element of (\ref{form31111})
satisfying (\ref{equation3230}).
We put ${\bf p}_{{\rm v}} = ((\Sigma_{{\rm v}},\vec z_{{\rm v}},\vec{\frak z}_{{\rm v}}),u_{\rm v})$.
We glue $(\Sigma_{{\rm v}},\vec z_{{\rm v}},\vec{\frak z}_{{\rm v}})$ to obtain
$(\Sigma,\vec z,\vec{\frak z})$ as follows.
Consider the disjoint union $\bigcup \Sigma_{{\rm v}}$. For each ${\rm e} \in C_{1}^{{\rm int}}(\mathcal T)$
we identify the $0$-th (boundary) marked point of $\Sigma_{s({\rm e})}$
with $k_{\rm e}$-th (boundary) marked point of $\Sigma_{t({\rm e})}$.
We thus obtain $\Sigma$.
We define  $k+1$-boundary marked points $\vec z$ of $\Sigma$ as follows.
Let $i \in \{0,\dots,k\}$. We consider the $i$-th exterior vertex ${\rm v}_i$ of $\mathcal T$.
Let ${\rm e}_i$ be the edge containing it and ${\rm v}'_i$ be the other vertex
contained in ${\rm e}_i$. Suppose ${\rm e}_i = {\rm e}_j({\rm v}'_i)$.
Then $z_i$ is the $j$-th boundary marked point of $\Sigma_{{\rm v}'_i}$.
We define interior marked points $\vec{\frak z}_{{\rm v}}$ using $l(\cdot)$ in a similar way.
\par
We define $u : \Sigma \to X$ so that $u\vert_{\Sigma_{\rm v}} = u_{\rm v}$.
By (\ref{equation3230}) $u$ is continuous.
It is easy to see that $((\Sigma,\vec z,\vec{\frak z}),u)$ represents an element of
${\mathcal M}_{k+1,\ell}(X,L,J;\beta)$.
We have thus defined a map (\ref{stratumemb}).
By the definition of stable map topology (see \cite[Definition 7.1.39]{fooobook2}) this is a homeomorphism
onto its image.
\end{proof}
Note ${\mathcal M}_{k+1,\ell}(X,L,J;\beta)(\frak T)$ is compact. So its image in
${\mathcal M}_{k+1,\ell}(X,L,J;\beta)$ is a closed subset.
\par
Let
$$
{\mathcal M}^{\circ}_{k+1,\ell}(X,L,J;\beta)
$$
be a subset of
${\mathcal M}_{k+1,\ell}(X,L,J;\beta)$
consisting of those elements $((\Sigma,\vec z,\vec{\frak z}),u)$ whose domain $\Sigma$
contains only one irreducible disk component.
We put
\begin{equation}\label{stratumemb2}
\aligned
&{\mathcal M}^{\circ}_{k+1,\ell}(X,L,J;\beta)(\frak T) \\
&:=
\prod_{{\rm v} \in C_{0}^{{\rm int}}(\mathcal T)} {\mathcal M}^{\circ}_{k_{\rm v}+1,
\#l({\rm v})}(X,L,J;\beta({\rm v}))
{}_{\mathcal{EV}}\times \Delta^{\# C_{1}^{{\rm int}}(\mathcal T)},
\endaligned
\end{equation}
which is a subset of ${\mathcal M}_{k+1,\ell}(X,L,J;\beta)(\frak T)$.
We regard it as a subset of ${\mathcal M}_{k+1,\ell}(X,L,J;\beta)$
by Lemma \ref{lem3232}.
The next lemma is immediate from the construction.
\begin{lem}\label{lem333}
${\mathcal M}_{k+1,\ell}(X,L,J;\beta)$
is a {\bf disjoint} union of
${\mathcal M}^{\circ}_{k+1,\ell}(X,L,J;\beta)(\frak T)$ over various
$\frak T \in \mathcal G(k+1,\ell,\beta)$.
\end{lem}
Let $(\mathcal T,\beta(\cdot),l(\cdot)) \in \mathcal G(k+1,\ell,\beta)$
and $\rm e$  its interior edge.
We define $(\mathcal T',\beta'(\cdot),l'(\cdot)) \in \mathcal G(k+1,\ell,\beta)$
as follows:
\begin{enumerate}
\item[(a)]
$\mathcal T'$ is obtained by contracting ${\rm e}$ to a point in $\mathcal T$.
\item[(b)]
We have $\beta'({\rm v}) = \beta({\rm v})$ if ${\rm v}$ is not a vertex
corresponding to ${\rm e}$, and $\beta'({\rm v}) = \beta({\rm v}_1) + \beta({\rm v}_2)$
if ${\rm v}$ is the vertex corresponding to the contracted edge ${\rm e}$
with $\partial{\rm e} = \{{\rm v}_1, {\rm v}_2\}$.
\item [(c)] We set $l'({\rm v}) = l({\rm v})$ if ${\rm v} \in C_0^{\rm int}(\mathcal T)$ is neither ${\rm v}_1$ nor ${\rm v}_2$
(vertices of ${\rm e}$).
(Note such ${\rm v}$ can be regarded as an  interior vertex of $\mathcal T'$.)
We also put  $l'({\rm v}) = l({\rm v}_1) \cup l({\rm v}_2)$
if ${\rm v}$ is the new vertex obtained by collapsing ${\rm e}$.
\end{enumerate}
We say $(\mathcal T',\beta'(\cdot),l'(\cdot))$
is obtained from $(\mathcal T,\beta(\cdot),l(\cdot))$ by an {\it edge
contraction}.
We say $\frak T' \ge \frak T$ if  $\frak T'$ is obtained from
$\frak T$ by finitely many times of edge contractions.
(The case $\frak T = \frak T'$ is included.)
\par
The next lemma is obvious from definition.

\begin{lem}\label{lema3434}
Suppose ${\mathcal M}_{k+1,\ell}(X,L,J;\beta)(\frak T)$ is nonempty. Then
$\frak T \le \frak T'$ if and only if
${\mathcal M}_{k+1,\ell}(X,L,J;\beta)(\frak T) \subseteq {\mathcal M}_{k+1,\ell}(X,L,J;\beta)(\frak T')$.
\end{lem}
We also remark the following.
\begin{lem}
Let ${\bf p} \in {\mathcal M}^{\circ}_{k+1,\ell}(X,L,J;\beta)(\frak T)
\subset {\mathcal M}_{k+1,\ell}(X,L,J;\beta)$.
Then ${\bf p}$ is a point in the codimension $m$ corner
with respect to the Kuranishi structure of ${\mathcal M}_{k+1,\ell}(X,L,J;\beta)$
if and only if $\frak T$ has at least $m$ interior edges.
\end{lem}
\begin{proof}
Let ${\bf p} = ((\Sigma,\vec z,\vec{\frak z}),u)
= ({\bf p}_{{\rm v}})_{{\rm v} \in C_{0}^{{\rm int}}(\mathcal T)}$.
By construction, the Kuranishi neighborhoods of ${\bf p}$ is diffeomorphic to the fiber product of
the Kuranishi neighborhoods of  various ${\bf p}_{{\rm v}}$ in
${\mathcal M}^{\circ}_{k_{\rm v}+1,
\#l({\rm v})}(X,L,J;\beta({\rm v}))$ times
$[0,1)^{\# C_{1}^{{\rm int}}(\mathcal T)}$.
The Kuranishi neighborhood of ${\bf p}_{{\rm v}}$ has no boundary.
Therefore ${\bf p}$ is in the codimension $\# C_{1}^{{\rm int}}(\mathcal T)$ corner.
\end{proof}
We next discuss the `ambient set'
${\mathcal X}_{k+1,\ell}(X,L;\beta)$.
We first remark that the evaluation map ${\rm ev}_i$ ($i=0,\dots,k$)
on ${\mathcal M}_{k+1,\ell}(X,L,J;\beta)$
at the $i$-th boundary point extends to a map
$$
{\rm ev}_i : {\mathcal X}_{k+1,\ell}(X,L;\beta)
\to L
$$
in an obvious way.
\begin{defn}\label{defn3131rev}
We put
$$
\aligned
&\prod_{(\mathcal T,\beta(\cdot),l(\cdot))} {\mathcal X}_{k_{\rm v}+1,
\#l({\rm v})}(X,L;\beta({\rm v})) \\
&=
\prod_{{\rm v} \in C_{0}^{{\rm int}}(\mathcal T)} {\mathcal X}_{k_{\rm v}+1,
\#l({\rm v})}(X,L;\beta({\rm v}))
{}_{\mathcal{EV}}\times \Delta^{\# C_{1}^{{\rm int}}(\mathcal T)}.
\endaligned
$$
In other words, it is the set of all elements
$\vec{\bf x} = ({\bf x}_{{\rm v}})_{{\rm v} \in C_{0}^{{\rm int}}}$ of
the direct product
\begin{equation}
\prod_{{\rm v} \in C_{0}^{{\rm int}}(\mathcal T)} {\mathcal X}_{k_{\rm v}+1,
\#l({\rm v})}(X,L;\beta({\rm v}))
\end{equation}
satisfying
\begin{equation}\label{equation323}
{\rm ev}_{0}({\bf x}_{s({\rm e})}) = {\rm ev}_{k_{\rm e}}({\bf x}_{t({\rm e})})
\end{equation}
for all ${\rm e} \in C_{1}^{{\rm int}}(\mathcal T)$.
To simplify the notation, we write
\begin{equation}
{\mathcal X}_{k+1,\ell}(X,L;\beta)(\frak T) :=
\prod_{(\mathcal T,\beta(\cdot),l(\cdot))} {\mathcal X}_{k_{\rm v}+1,
\#l({\rm v})}(X,L;\beta({\rm v}))
\end{equation}
with $\frak T = (\mathcal T,\beta(\cdot),l(\cdot))$.
\end{defn}
We again emphasize that the fiber product and etc. in the above definition are
taken in the category of sets.
\begin{lem}
There exists a (set theoretical) map
\begin{equation}\label{stratumembset}
{\mathcal X}_{k+1,\ell}(X,L;\beta)(\frak T) \to {\mathcal X}_{k+1,\ell}(X,L;\beta).
\end{equation}
which extends the map (\ref{stratumemb}).
The map (\ref{stratumembset}) is injective.
\end{lem}
The proof is the same as the proof of Lemma \ref{lem3232}.
\begin{rem}
The equality
\begin{equation}
{\mathcal X}_{k+1,\ell}(X,L;\beta)
=
\bigcup_{\frak T \in \mathcal G(k+1,\ell,\beta)}
{\mathcal X}_{k+1,\ell}(X,L;\beta)(\frak T)
\end{equation}
does {\it not} hold. This is because we do not assume that
the restriction of the map to an unstable component has
positive energy
for an element of ${\mathcal X}_{k+1,\ell}(X,L;\beta)$.
In the case of pseudo-holomorphic map, the stability in the sense of
Definition \ref{stabilitydefn27} implies that the restriction of the map to an unstable component has
positive energy.
\end{rem}
\begin{rem}
In our situation of the disk, the group of automorphisms of an element $\frak T$
of $\mathcal T(k+1,\ell,\beta)$ is trivial.
This is the main reason why (\ref{stratumembset}) is injective.
In various other situations, for example, when we consider the
bordered curve of higher genus, the group of automorphisms
of  the graph preserving the additional data (describing the combinatorial
structure of the  object) can be nontrivial.
\end{rem}
The next lemma describes the relationship between the stratification, fiber product
and the partial topology.
\begin{lem}\label{lem393939}
Let ${\bf p} = ({\bf p}_{{\rm v}})_{{\rm v} \in C_{0}^{{\rm int}}(\mathcal T)}
\in {\mathcal M}_{k+1,\ell}(X,L,J;\beta)(\frak T)$.
Then for any $\epsilon > 0$ there exists $\epsilon' > 0$
with the following properties:
\begin{enumerate}
\item
We have an inclusion
$$
\prod_{(\mathcal T,\beta(\cdot),l(\cdot))} B_{\epsilon'}({\mathcal X},{\bf p}_{{\rm v}})
\subset
B_{\epsilon}({\mathcal X},{\bf p}).
$$
Here $B_{\epsilon'}({\mathcal X},{\bf p}_{{\rm v}})
\subset {\mathcal X}_{k_{\rm v}+1,
\#l({\rm v})}(X,L,J;\beta({\rm v}))$
is the $\epsilon'$-neighborhood of  ${\bf p}_{{\rm v}}$
and
$B_{\epsilon}({\mathcal X},{\bf p}) \subset {\mathcal X}_{k+1,\ell}(X,L;\beta)$
is the  $\epsilon$-neighborhood of  ${\bf p}$.
\item
We have an inclusion
$$
B_{\epsilon'}({\mathcal X},{\bf p})
\cap {\mathcal X}_{k+1,\ell}(X,L;\beta)(\frak T)
\subset
\prod_{(\mathcal T,\beta(\cdot),l(\cdot))} B_{\epsilon}({\mathcal X},{\bf p}_{{\rm v}}).
$$
\end{enumerate}
\end{lem}
\begin{proof}
We recall that when we define the $\epsilon$-neighborhood
$B_{\epsilon}({\mathcal X},{\bf p})$ we choose and fix a
stabilization and trivialization data $\frak W_{\bf p}$
defined as in \cite[Definition 4.9]{diskconst1}.
(See \cite[Definition 4.12]{diskconst1}.)
However the partial topology is independent of such choices
up to equivalence by \cite[Lemma 4.14]{diskconst1}.
(See \cite[Definition 4.1]{diskconst1} for the definition of
equivalence of partial topology.)
In particular, it implies that the
validity of Lemma \ref{lem393939} is independent of the choice
of the stabilization and trivialization data $\frak W_{\bf p}$
of ${\bf p}$ and of $\frak W_{\bf p_{\rm v}}$
of ${\bf p_{\rm v}}$.
\par
We take the following choice. Recall that $\frak W_{\bf p}$ consists of the following data:
\begin{enumerate}
\item
The additional (interior) marked points $\vec{\frak w}_{\bf p}$.
\item
An analytic family of coordinates at each node of ${\bf p} \cup \vec{\frak w}_{\bf p}$.
\item
A $C^{\infty}$ trivialization of the universal family of the deformation
of the source curve of (each irreducible component) of ${\bf p} \cup \vec{\frak w}_{\bf p}$.
It is assumed to be compatible with the analytic family of coordinates at each node
in Item (2).
\item
A Riemannian metric on each irreducible component of $\Sigma_{\bf p}$.
\end{enumerate}
See \cite[Definition 4.9]{diskconst1}.
Suppose the choices of (1)-(4) are given for each ${\bf p}_{\rm v}$.
Then
we can define such a choice for  ${\bf p}$ as follows:
\begin{enumerate}
\item
$\vec{\frak w}_{\bf p} = \cup_{{\rm v} \in C_{0}^{{\rm int}}(\mathcal T)}\vec{\frak w}_{{\bf p}_{\rm v}}$,
where $\vec{\frak w}_{{\bf p}_{\rm v}}$ is the choice for ${\bf p}_{\rm v}$ we have taken.
\item
If a node of $\Sigma_{\bf p}$ is a node of $\Sigma_{\bf p_{\rm v}}$ then we take
the analytic family of coordinates at that node of $\Sigma_{\bf p_{\rm v}}$
which we fixed as a part of stabilization and trivialization data $\frak W_{\bf p_{\rm v}}$. If a node of $\Sigma_{\bf p}$ is not a node of any $\Sigma_{\bf p_{\rm v}}$
we take any analytic family of coordinates at that node.
\item
Any irreducible component of $\Sigma_{\bf p}$ is an irreducible component of
$\Sigma_{\bf p_{\rm v}}$ for some ${\rm v}$. The $C^{\infty}$ trivialization of the
universal family of the deformation of the source curve of this irreducible component
is the one we have taken as a part of stabilization and trivialization data $\frak W_{\bf p_{\rm v}}$.
\item
We can fix a Riemannian metric of each irreducible component of $\Sigma_{\bf p}$
using the Riemannian metric of irreducible components of $\Sigma_{\bf p_{\rm v}}$.
\end{enumerate}
When we take this choices of $\frak W_{\bf p}$
of ${\bf p}$ and of $\frak W_{\bf p_{\rm v}}$,
the conclusions (1)(2) of Lemma \ref{lem393939} are obvious from
the definition. (\cite[Definition 4.12]{diskconst1}.)
\end{proof}

\section{Disk-component-wise-ness of obstruction bundle data}
\label{subsec;obstbundledata}

In this section we use the discussion in the previous two sections to spell out the
condition we require for the obstruction bundle data so that the induced
Kuranishi structures satisfy the conclusions of Theorem \ref{therem273}
and of the $P=$ point case of Theorem \ref{therem274}.
\par
\begin{defn}\label{defn5151}
Suppose we are given obstruction bundle data $\{E_{\bf p}({\bf x})\}$
of the moduli space ${\mathcal M}_{k+1,\ell}(X,L,J;\beta)$ for each $\beta$.
We say that they consist of a {\it disk-component-wise system of obstruction bundle data}
if the following holds:
Let ${\bf p} \in {\mathcal M}_{k+1,\ell}(X,L,J;\beta)$ and
${\bf p} = ({\bf p}_{{\rm v}})_{{\rm v} \in C_{0}^{{\rm int}}(\mathcal T)} \in \mathcal M_{k+1,\ell}(X,L,J;\beta)(\frak T)$
with $\frak T =({\mathcal T},\beta(\cdot), l(\cdot))$.
Then for sufficiently small neighborhoods
$$
\mathscr U_{\bf p}
\subset {\mathcal X}_{k+1,\ell}(X,L;\beta), \quad
\mathscr U_{\bf p_{\rm v}}
\subset {\mathcal X}_{k_{\rm v}+1,l({\rm v})}(X,L;\beta_{\rm v})
$$
with
$$
\prod_{(\mathcal T,\beta(\cdot),l(\cdot))} \mathscr U_{{\bf p}_{\rm v}}
\subseteq
\mathscr U_{{\bf p}},
$$
(see Lemma \ref{lem393939})
the equality
\begin{equation}\label{form410}
E_{\bf p}({\bf x}) = \bigoplus_{{\rm v} \in C_{0}^{{\rm int}}(\mathcal T)}
E_{{\bf p}_{\rm v}}({\bf x}_{\rm v})
\end{equation}
holds, where ${\bf x} = ({\bf x}_{{\rm v}})_{{\rm v} \in C_{0}^{{\rm int}}(\mathcal T)}$
is an arbitrary element of $\prod_{(\mathcal T,\beta(\cdot),l(\cdot))} \mathscr U_{{\bf p}_{\rm v}}$.
\end{defn}
We elaborate on the equality (\ref{form410}) below.
Let
${\bf x}_{{\rm v}} = ((\Sigma_{{\bf x}_{{\rm v}}},\vec z_{{\bf x}_{{\rm v}}},\vec{\frak z}_{{\bf x}_{{\rm v}}}),
u_{{\bf x}_{{\rm v}}})$,
${\bf x} = ((\Sigma_{{\bf x}},\vec z_{{\bf x}},\vec{\frak z}_{{\bf x}}),
u_{{\bf x}})$.
Then
$$
E_{\bf p}({\bf x}) \subset C^{2}(\Sigma_{{\bf x}},u_{{\bf x}}^*TX \otimes \Lambda^{0,1}),
\quad
E_{\bf p_{\rm v}}({\bf x}_{{\rm v}}) \subset C^{2}(\Sigma_{{\bf x}_{{\rm v}}},u_{{\bf x}_{{\rm v}}}^*TX \otimes \Lambda^{0,1}).
$$
More precisely,
$C^{2}(\Sigma_{{\bf x}},u_{{\bf x}}^*TX \otimes \Lambda^{0,1})$ and
etc, is the direct sum of the spaces of $C^2$-sections of irreducible components.
Since the normalization of $\Sigma_{{\bf x}}$ is the disjoint union of the
normalizations of $\Sigma_{{\bf x}_{\rm v}}$ and the restriction of $u_{{\bf x}}$
to the irreducible components coincides with the restriction of some
$u_{{\bf x}_{{\rm v}}}$, we have
$$
C^{2}(\Sigma_{{\bf x}},u_{{\bf x}}^*TX \otimes \Lambda^{0,1})
=
\bigoplus_{{\rm v} \in C_{0}^{{\rm int}}(\mathcal T)}
C^{2}(\Sigma_{{\bf x}_{{\rm v}}},u_{{\bf x}_{{\rm v}}}^*TX \otimes \Lambda^{0,1}).
$$
Thus (\ref{form410}) makes sense.
\begin{rem}
The notion of disk-component-wiseness appeared
in \cite[Definition 4.2.2]{foootoric32}.
\end{rem}
Now the proofs of Theorem \ref{therem273}
and the $P=$ point case of Theorem \ref{therem274} are divided into the proofs of the next two theorems.
Theorem \ref{thm42} is proved in Section \ref{subsec;cornercompa},
and Theorem \ref{thm43} is proved in
Sections \ref{subsec;exi1} and \ref{subsec;exi2}.

\begin{thm}\label{thm42}
Let $\{E_{\bf p}({\bf x})\}$ be a disk-component-wise system of obstruction bundle data of $\{{\mathcal M}_{k+1,\ell}(X,L,J;\beta) \mid k,\ell,\beta\}$.
It induces a Kuranishi structure on ${\mathcal M}_{k+1,\ell}(X,L,J;\beta)$
for each $k,\ell,\beta$ by \cite[Theorem 7.1]{diskconst1}.
Then the system of obtained Kuranishi structures satisfies
Theorem \ref{therem274} (IX)(Compatibility at the boundary),
(X)(Corner compatibility isomorphism),
(XI)(Consistency of corner compatibility isomorphisms).
In other words,
the Kuranishi structures are compatible with boundary and corners.
\end{thm}
\begin{thm}\label{thm43}
There exists a disk-component-wise system of obstruction bundle data
of $\{{\mathcal M}_{k+1,\ell}(X,L,J;\beta) \mid k,\ell,\beta\}$.
\end{thm}
 $P=$ point case of Theorem \ref{therem274} follows from Theorems \ref{thm42} and \ref{thm43}.
Theorem  \ref{therem273} is its special case where $\ell = 0$.

\section{Disk-component-wise-ness implies corner compatibility condition}
\label{subsec;cornercompa}

In this section we prove Theorem \ref{thm42}.
We first fix a representative of the Kuranishi structure of
${\mathcal M}_{k+1,\ell}(X,L,J;\beta)$  for each
$(k,\ell,\beta)$.
Recall from \cite[Theorem 7.1 (2)]{diskconst1} that the germs of the
Kuranishi structures of $\{{\mathcal M}_{k+1,\ell}(X,L,J;\beta)\}$ are canonically
determined by the system of obstruction bundle data we start with.
A germ of Kuranishi structure is an equivalence class of the
set of Kuranishi structures. We fix its representative so that
the associated obstruction bundle data $(\{\mathscr U_{{\bf p}}\}, \{E_{{\bf p}}({\bf x})\})$
satisfy
\begin{equation}\label{form51}
\mathscr U_{{\bf p}}\cap {\mathcal X}_{k+1,\ell}(X,L;\beta)(\frak T)=
\prod_{(\mathcal T,\beta(\cdot),l(\cdot))} \mathscr U_{{\bf p}_{\rm v}}
\end{equation}
for ${\bf p}  = ({\bf p}_{{\rm v}})_{{\rm v} \in C_{0}^{{\rm int}}(\mathcal T)} \in \mathcal M_{k+1,\ell}(X,L,J;\beta)(\frak T)$.

\begin{rem}
Here the isomorphism in Theorem \ref{therem274} (IX)(Compatibility at the boundary), (X)(Corner compatibility isomorphism) and the coincidence of the maps
in Theorem \ref{therem274} (XI)(Consistency of corner compatibility isomorphisms)
are taken in the sense of germs of Kuranishi structures and maps between them.
So it suffices to choose representatives satisfying (\ref{form51}) for a fixed choice of $k,\ell,\beta$.
We can easily choose the representatives so that (\ref{form51})
holds for any finitely many choices of $k,\ell,\beta$.
Then the isomorphism in Theorem \ref{therem274} (IX)(Compatibility at the boundary), (X)(Corner compatibility isomorphism) and the coincidence of maps in Theorem \ref{therem274} (XI)(Consistency of corner compatibility isomorphisms)
hold exactly (not as germs).
It seems difficult to choose the representatives so that (\ref{form51}) holds for all
(infinitely many) $k,\ell,\beta$ simultaneously.
This does not matter when applying Theorem \ref{therem274} and
\cite[Theorem 21.35 (1)]{foootech21}, \cite[Theorem 21.35 (1)]{Springer} to prove Theorem \ref{therem274}.
This is because we use the `homotopy inductive limit' in the proof of
\cite[Theorem 21.35 (1)]{foootech21}, \cite[Theorem 21.35 (1)]{Springer}.
\end{rem}
\begin{proof}[Proof of Theorem \ref{thm42}]
We first show Theorem \ref{therem274} (X)(Corner compatibility isomorphism).
We recall that in \cite[Definition 7.2]{diskconst1} the Kuranishi neighborhood
$U_{\bf p}$ of ${\bf p} \in \mathcal M_{k+1,\ell}(X,L,J;\beta)$ is set-theoretically defined by
\begin{equation}\label{Knhd}
U_{\bf p} = \{ {\bf x} = [(\Sigma_{{\bf x}},\vec z_{{\bf x}},\vec{\frak z}_{{\bf x}}),
u_{{\bf x}}]  \in \mathscr U_{{\bf p}} \mid
\overline\partial u_{{\bf x}} \in E_{\bf p}({\bf x})\}.
\end{equation}
Let
${\bf p} = ({\bf p}_{{\rm v}})_{{\rm v} \in C_{0}^{{\rm int}}(\mathcal T)} \in \mathcal M_{k+1,\ell}(X,L,J;\beta)(\frak T)$
with $\frak T = ({\mathcal T},\beta(\cdot), l(\cdot))$ as in
Definition \ref{defn5151}.
Then we also have
$$
U_{{\bf p}_{\rm v}} = \{ {\bf x}_{\rm v} = [(\Sigma_{{\bf x}_{\rm v}},\vec z_{{\bf x}_{\rm v}},\vec{\frak z}_{{\bf x}_{\rm v}}),
u_{{\bf x}_{\rm v}}]  \in \mathscr U_{{\bf p}_{\rm v}} \mid
\overline\partial u_{{\bf x}_{\rm v}} \in E_{{\bf p}_{\rm v}}({\bf x}_{\rm v})\}.
$$
By the equality (\ref{form410}) the identification (\ref{form51}) induces a
set-theoretical bijection:
\begin{equation}\label{form52}
\prod_{(\mathcal T,\beta(\cdot),l(\cdot))} U_{{\bf p}_{\rm v}}
=
U_{{\bf p}}\cap {\mathcal X}_{k+1,\ell}(X,L;\beta)(\frak T).
\end{equation}
\par
We next discuss the relationship between the right hand side of (\ref{form52}) and the
{\it normalized} corner of $U_{{\bf p}}$.
By Lemma \ref{lem333} there exists a unique $\frak T'$ such that ${\bf p} \in {\mathcal M}^{\circ}_{k+1,\ell}(X,L,J;\beta)(\frak T')$.
Let $m'$ be the number of boundary nodes of ${\bf p}$. Then $m'$ coincides with the number of interior edges
of $\frak T'$.
Since ${\bf p} \in {\mathcal M}_{k+1,\ell}(X,L,J;\beta)(\frak T)$, Lemma \ref{lema3434} implies
$\frak T \ge \frak T'$.
\begin{lem}\label{lem52}
The codimension $m''$ normalized corner $\widehat S_{m''}U_{\bf p}$ of $U_{\bf p}$ is the disjoint union
$U_{{\bf p}}\cap {\mathcal X}_{k+1,
\ell}(X,L,J;\beta)(\frak T'')$ over $\frak T''$ such that
$\frak T'' \le \frak T'$ and $\frak T''$ has exactly $m''$ interior edges.
\end{lem}
\begin{proof}
By construction
$U_{\bf p}$ is diffeomorphic to the product of $[0,1)^{m'}$ and an orbifold $\overline U_{\bf p}$ without boundary
or corner. The $[0,1)^{m'}$ factor parametrizes the way to smooth $m'$ boundary nodes.
So each of the $m'$ factors of $[0,1)^{m'}$ canonically corresponds to a boundary node.
By the definition of a normalized corner (see \cite[Definition 24.17]{foootech21}, 
\cite[Definition 24.18]{Springer}) $\widehat S_{m''}U_{\bf p}$
is identified with the {\it disjoint} union
$$
\bigcup_{I}
\left(\{(t_1,\dots,t_{m'}) \in [0,1)^{m'} \mid t_i = 0,\,\,\,\, \text{if $i \in I$}\} \times \overline U_{\bf p}
\right)
$$
where $I \subset \{1,\dots,m'\}$ runs over the set of subsets of order $m''$.\footnote
{We remark that the action of the group of (extended) isomorphisms of ${\bf p}$
on the $[0,1)^{m'}$ factor is trivial, because an extended automorphism is
the identity map on disk components.}
Therefore the connected components of $\widehat S_{m''}U_{\bf p}$ correspond one to one to an
order $m''$ subsets of the set of interior edges of $\frak T'$.
Such subsets correspond one to one to
those $\frak T''$ with $\frak T'' \le \frak T'$ that has $m''$ interior edges.
The lemma follows.
\end{proof}
By Lemma \ref{lem52} the identification (\ref{form52}) can be regarded as a
map
\begin{equation}\label{map5353}
\prod_{(\mathcal T'',\beta''(\cdot),l''(\cdot))} U_{{\bf p}_{\rm v}}
\to
\widehat S_{m''}U_{{\bf p}}
\end{equation}
which is a bijection to a connected component of the normalized corner.
Using the characterization of the smooth structures of $U_{{\bf p}_{\rm v}}$ and of
$U_{{\bf p}}$ (see \cite[Subsection 12.1]{diskconst1}),
the map (\ref{map5353}) is a diffeomorphism onto the
connected component.
Here we use \cite[Theorem 6.4]{foooexp}.
By (\ref{form410}) the map (\ref{map5353}) is covered by an isomorphism of
obstruction bundles.
The compatibility  of the diffeomorphism (\ref{map5353})
with the Kuranishi map ${\bf x} \mapsto s_{\bf p}({\bf x}) = \overline\partial u_{\bf x}
\in E_{{\bf p}}({\bf x})$, the parametrization map $s_{\bf p}^{-1}(0) \to {\mathcal M}_{k+1,\ell}(X,L,J;\beta)$ and the coordinate change
is obvious from the construction.
The proof of Theorem \ref{therem274} (X)(Corner compatibility isomorphism) is complete.
\par
Theorem \ref{therem274} (IX)(Compatibility of the boundary) is a special case of Theorem \ref{therem274} (X)(Corner compatibility isomorphism).
Theorem \ref{therem274} (XI)(Consistency of corner compatibility isomorphisms) is immediate from the description of
normalized corner we gave in the proof of Lemma \ref{lem52}.
The proof of Theorem \ref{thm42} is complete.
\end{proof}

\section{Existence of a disk-component-wise system of obstruction bundle data 1}
\label{subsec;exi1}

\subsection{The idea of the construction}
\label{subsec:idea}
In Sections \ref{subsec;exi1} and \ref{subsec;exi2} we prove Theorem \ref{thm43}.
We first explain the idea of the construction.
Let ${\bf q} \in \mathcal M^{\circ}_{k+1,\ell}(X,L,J;\beta)(\frak T')$
and ${\rm v}$ an interior vertex of $\frak T'$.
We denote by ${\bf q}_{\rm v} \in \mathcal M^{\circ}_{k_{\rm v}+1,\ell_{\rm v}}(X,L,J;\beta({\rm v}))$
the irreducible component of ${\bf q}$ corresponding to ${\rm v}$.
Let ${\bf x} \in {\mathcal X}_{k+1,\ell}(X,L;\beta)$ be an element 
close to ${\bf q}$.
The disk-component-wise-ness of the 
obstruction bundle data implies that $E_{{\bf q}}({\bf x})$
is a direct sum of the subspaces $E_{\rm v}({\bf x}) = E_{{\bf q}_{\rm v}}({\bf x})$, assigned to each 
${\rm v}$.  It is important that 
$E_{{\rm v}}({\bf x})$ is independent of ${\bf q}_{{\rm v}'}$ for ${\rm v}'\ne {\rm v}$.
To define 
$E_{{\rm v}}({\bf x})$ we take several elements 
$\frak p$ in $\mathcal M^{\circ}_{k_{\rm v}+1,\ell_{\rm v}}(X,L,J;\beta({\rm v}))$ close to ${\bf q}_{\rm v}$
and  $E_{{\rm v}}({\bf x})$ is a direct sum of 
$E_{\frak p}({\bf x})$ for those $\frak p$'s.
\par
Semi-continuity of the obstruction bundle data
implies the following.
If ${\bf p} \in \mathcal M_{k+1,\ell}(X,L,J;\beta)$ 
is sufficiently close to ${\bf q}$ then 
$E_{{\bf p}}({\bf x})$ is also 
decomposed into the sum of $E_{{\rm v}}({\bf x})$ with various ${\rm v}$.
Note that there may not exist an irreducible 
component of ${\bf p}$ corresponding to ${\rm v}$.
(See Definition \ref{defn3131} and Lemma \ref{lem3232}.)
For example in ${\bf p}$ the irreducible component
corresponding to ${\rm v}$ and that of ${\rm v}'$ may already 
be glued. Neverthless 
$E_{{\rm v}}({\bf x})$ should be independent of ${\bf q}_{{\rm v}'}$.
\par
We introduce the notion of `quasi-component' to handle such a 
situation.
{\it Roughly speaking} a quasi-component of ${\bf p}$ 
is an element $\frak p$ which is close to an irreducible component of ${\bf q}$ such that 
${\bf p}$ is sufficiently close to ${\bf q}$.
We then define $E_{{\bf p}}({\bf x})$ as a 
direct sum of the subspaces 
associated to quasi-components.
\par
We formulate this situation below.
Let $d$ be a metric on $\mathcal M_{k+1,\ell}(X,L,J;\beta)$.
\begin{shitu}\label{situ61}
Let ${\bf p} \in \mathcal M_{k+1,\ell}(X,L,J;\beta)$,
${\bf q} \in \mathcal M^{\circ}_{k+1,\ell}(X,L,J;\beta)$. 
We take
$\frak T,\frak T' = (\mathcal T,\beta(\cdot),l(\cdot)) \in \mathcal G(k+1,\ell,\beta)$
such that ${\bf p} \in \mathcal M^{\circ}_{k+1,\ell}(X,L,J;\beta)(\frak T)$
and ${\bf q} \in \mathcal M^{\circ}_{k+1,\ell}(X,L,J;\beta)(\frak T')$.
We remark that there exists $\epsilon({\bf q})>0$ depending on ${\bf q}$ 
but independent of ${\bf p}$ such that the following holds.
If
\begin{equation}\label{formnew71}
d({\bf p},{\bf q}) < \epsilon_1
\end{equation}
with $\epsilon_1 < \epsilon({\bf q})$
then $\frak T' \le \frak T$.
Let ${\rm v}$ be an interior vertex of $\frak T'$.
We denote by ${\bf q}_{\rm v} \in \mathcal M^{\circ}_{k_{\rm v}+1,\ell_{\rm v}}(X,L,J;\beta({\rm v}))$
 the irreducible component of ${\bf q}$ corresponding to ${\rm v}$.
Suppose
\begin{equation}\label{formnew72}
d({\bf q}_{\rm v},{\frak p}) < \epsilon_2
\end{equation}
for ${\frak p} \in
\mathcal M_{k_{\rm v}+1,\ell_{\rm v}}^{\circ}(X,L,J;\beta({\rm v}))$.
See Figure \ref{zu2}.
Here $\epsilon_2 > 0$ is a sufficiently small number depending on $\frak p$.
\par
Recall from the definition of
$\mathcal M^{\circ}_{k_{\rm v}+1,\ell_{\rm v}}(X,L,J;\beta({\rm v}))$ that
 $\frak p$ has only one disk component but may have sphere components.
We decompose
\begin{equation}\label{form61}
\Sigma_{\frak p} = \Sigma_{\frak p}^{{\rm d}} \cup \bigcup_{\frak v}
\Sigma_{\frak p,\frak v}^{\rm s}
\end{equation}
where $\Sigma_{\frak p}^{{\rm d}}$ is the disk component and
$\Sigma_{\frak p,\frak v}^{{\rm s}}$
is a sphere component.
$\diamond$
\begin{figure}[h]
\centering
\includegraphics[scale=0.5]{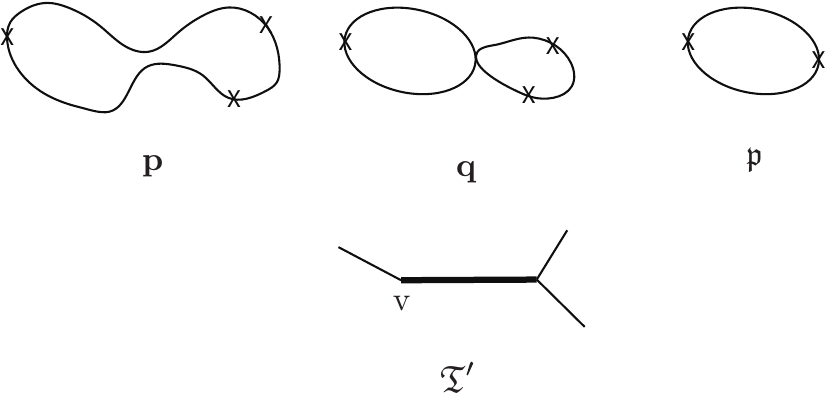}
\caption{${\bf p},{\bf q},{\frak p}$}
\label{zu2}
\end{figure}
\end{shitu}

\begin{conven}
We use specific letters/font in the notations of this paper.
\begin{enumerate}
\item[$\bullet$] ${\bf p, q} \in {\mathcal M}_{k+1,\ell}(X,L,J;\beta)$
for suitable $k,\ell,\beta$ depending on ${\bf p, q}$.
\item[$\bullet$] ${\bf x} \in {\mathcal X}_{k+1,\ell}(X,L;\beta)$, not necessarily
pseudo-holomorphic,
for suitable $k,\ell,\beta$ depending on ${\bf x}$.
\item[$\bullet$] $\frak p \in {\frak P}(k,\ell,\beta)$
for suitable $k,\ell,\beta$ depending on ${\frak p}$.
Here ${\frak P}(k,\ell,\beta)$ is a finite subset of ${\mathcal M}_{k+1,\ell}(X,L,J;\beta)$. See, for example, (ob1) in Section \ref{subsec;exi2}.
\end{enumerate}
\end{conven}
The study of Situation \ref{situ61} starts with
Situation \ref{situ64} in the next subsection after preparing and reviewing
several notions.
\par
For the actual construction, we need to specify 
how close ${\bf p}$ and ${\bf q}$ should be 
for an irreducible component of ${\bf q}$ 
to be a quasi-component of ${\bf p}$.
We need to make such a choice inductively 
on $k+1,\ell$ and $\beta \cap[\omega]$.

To work out this induction process we need to consider 
also the following situation:
``${\bf p}$ is close to ${\bf q}_1$ and 
${\bf q}'_1$ is an irreducible component 
of ${\bf q}_1$. 
${\bf q}'_1$ is close to ${\bf q}_2$
and ${\bf q}'_2$ is an irreducible component 
of ${\bf q}_2$.
${\bf q}_3$ etc may appear in a similar way.''
This iterated construction is carried out 
in Subsection \ref{subsec:itera2}.

\subsection{Stabilization by interior marked points}
\label{subsec:staint}

\par
\smallskip
The next definition is a variant of \cite[Definition 9.7]{diskconst1}.
\begin{defn}\label{defn62}
Suppose we are in Situation \ref{situ61}.
{\it Type I stabilization data}\footnote{Here $I$ stands for `interior marked points'.}
$(\vec{\frak w}_{\frak p},\vec{\mathcal N}_{\frak p})$ at ${\frak p}$ is the
following data.
\begin{enumerate}
\item
$\vec{\frak w}_{\frak p} = ({\frak w}_{{\frak p},1},\dots,{\frak w}_{{\frak p},\ell'})$
are distinct points in $\rm{Int}(\Sigma_{\frak p})$ away from $\vec{\frak z}_{\frak p}$,
the set of interior marked points of $\frak p$.
It is also away from 
nodal points.
\item
$(\Sigma_{\frak p},\vec{\frak w}_{\frak p})$ is stable, that is,
the group
$$
{\rm Aut}(\Sigma_{\frak p},\vec{\frak w}_{\frak p})
= \{ v : \Sigma_{\frak p} \to \Sigma_{\frak p} : \text{biholomorphic},\,\,\, v({\frak w}_{{\frak p},i})
={\frak w}_{{\frak p},i}\}
$$
is finite,
except the case when the unique disk component of $\Sigma_{\frak p}$ is unstable
and the map $u_{\frak p}$ is constant on it.
In such a case the connected
component of ${\rm Aut}(\Sigma_{\frak p},\vec{\frak w}_{\frak p})$ is $S^1$ and
$\pi_0{\rm Aut}(\Sigma_{\frak p},\vec{\frak w}_{\frak p})$ is a finite group.
See Figure \ref{Figure2-5}.
We explain this exceptional case more in Remark \ref{rem73}.
\begin{figure}[h]
\centering
\includegraphics[scale=0.5]{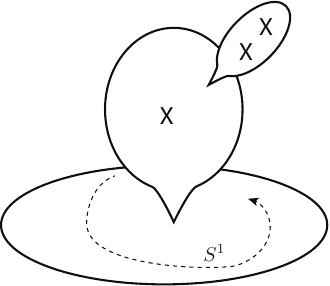}
\caption{Unstable element in $\mathcal M_{0,3}^{\rm d}$}
\label{Figure2-5}
\end{figure}
\item
$\vec{\mathcal N}_{\frak p} = (\mathcal N_{{\frak p},1},\dots,\mathcal N_{{\frak p},\ell'})$.
Here $\mathcal N_{{\frak p},i}$ is a codimension 2 submanifold of $X$.
\item There exists a neighborhood $U_i$ of $\frak w_{{\frak p},i}$ in $\Sigma_{\frak p}$
such that $u_{\frak p}(U_i)$ intersects transversally with $\mathcal N_{{\frak p},i}$ at
the unique point $u_{\frak p}(\frak w_{{\frak p},i})$.
Moreover, the restriction of $u_{\frak p}$ to $U_i$ is a smooth embedding.
We require $\{U_i\}$ are disjoint.
\item
Suppose that $v : \Sigma_{\frak p} \to \Sigma_{\frak p}$ is an extended automorphism
of $\frak p$. Then there exists a permutation $\sigma : \{1,\dots,\ell'\} \to \{1,\dots,\ell'\}$
such that
$v({\frak w}_{{\frak p},i}) = {\frak w}_{{\frak p},\sigma(i)}$,
$\mathcal N_{{\frak p},i} = \mathcal N_{{\frak p},\sigma(i)}$, and
$v(U_i) = U_{\sigma(i)}$.
\end{enumerate}
\end{defn}
We decompose $\Sigma_{\frak p}$ as in (\ref{form61}).
For each irreducible component $\Sigma_{\frak p}^{\rm d}$, (resp.
$\Sigma_{{\frak p},\frak v}^{\bf s}$)
we put $\vec{\frak w}_{\frak p} \cap \Sigma_{{\frak p}}^{{\rm d}}$,
(resp. $\vec{\frak w}_{\frak p} \cap \Sigma_{{\frak p},\frak v}^{\bf s}$) together with (necessarily interior)
nodes on it. We denote it by $(\Sigma_{{\frak p}}^{{\rm d}},
\vec{\frak w}_{{\frak p}}^{\rm d})$
(resp. $(\Sigma_{{\frak p},\frak v}^{\rm s},
\vec{\frak w}_{{\frak p},\frak v}^{\rm s})$).
They are stable\footnote{except the case explained in Remark \ref{rem73}.} and so determine an element of $\mathcal M_{0,\ell_{\rm d}}^{\rm d}$
(resp. elements of $\mathcal M^{\rm s}_{\ell_{\frak v}}$).
Here $\mathcal M_{0,\ell_{\rm d}}^{\rm d}$ is the compactified
moduli space of disks with $\ell_{\rm d}$ interior marked points
and $\mathcal M^{\rm s}_{\ell_{\frak v}}$ is the compactified
moduli space of spheres with $\ell_{\frak v}$ marked points.
We denote by
$\pi : \mathcal C^{\rm s}_{\ell} \to \mathcal M_{\ell}^{\rm s}$
the universal families of deformation.
\begin{rem}\label{rem73}
Here we consider the moduli space  $\mathcal M_{0,\ell_{\rm d}}^{\rm d}$
of genus zero curves with one boundary component and {\it without boundary marked points}.
If we define it in the same way as in Definition \ref{defn24222} assuming the stability
in the sense of Definition \ref{stabilitydefn27}, then such space is {\it not} compact.
To compactify it we need to add isomorphism classes of
elements $(\Sigma,\vec{\frak z})$ such that
$\Sigma$ has a disk component that has only one double point and that
none of the marked points $\vec{\frak z}$ are on this disk components.
See Figure \ref{Figure2-5}.
The group of extended automorphisms of such an element $(\Sigma,\vec{\frak z})$ contains
$S^1$ consisting of rotations of the disk component.
Therefore it is not stable in the sense
of Definition \ref{stabilitydefn27}.
The moduli space of such objects is diffeomorphic to
$\mathcal M_{\ell_{\rm d}+1}^{\rm s}$,
the compactified moduli space of spheres with $\ell_{\rm d}+1$ marked points.
We {\it add} the isomorphism classes of such objects
to  $\mathcal M_{0,\ell_{\rm d}}^{\rm d}$ and call the resulting moduli space
a compactified moduli space of stable disks with $\ell_{\rm d}$ interior marked points
by an abuse of notation. (Since the isotropy group $S^1$ is compact, this space is still Hausdorff.)
This is an orbifold with boundary and corners. The object described in Figure \ref{Figure2-5},
determines a boundary component of it. See \cite[Subsection 7.4.1]{fooobook2} for the discussion of this
extra boundary component.
\par
The universal family $\pi : \mathcal C_{0,\ell}^{\rm d} \to \mathcal M_{0,\ell}^{\rm d}$
is not well-defined on this boundary component for the part of disk component, because of the
automorphism $S^1$.  The universal family 
is defined outside of this disk component, which is in the fiber of the above mentioned 
component of the boudary.
\par
The condition we required in Definition \ref{defn62} (2) implies that
$(\Sigma_{{\frak p}}^{\rm d},\vec{\frak w}_{{\frak p}}^{\rm d})$ is in this extra boundary
component only when $u_{\frak p}$ is constant on the disk component.
\end{rem}
The next definition is a variant of \cite[Situation 9.8]{diskconst1}.
\begin{defn}\label{defn63}
{\it Type I strong stabilization data} at ${\frak p}$
are type I stabilization data $(\vec{\frak w}_{\frak p},\vec{\mathcal N}_{\frak p})$
together with the following data
$((\vec{\mathcal V}_{\frak p}, \vec{\phi_{{\frak p}}}),\vec{\varphi}_{\frak p})$.
\begin{enumerate}
\item[(6)]
A neighborhood $\mathcal V_{{\frak p}}^{{\rm d}}$ of
$[(\Sigma_{{\frak p}}^{{\rm d}},\vec{\frak w}_{{\frak p}}^{{\rm d}})]$
in $\mathcal M_{0,\ell_{\rm d}}^{\rm d}$ and
neighborhoods $\mathcal V_{{\frak p},\frak v}^{\rm s}$ of
$[(\Sigma_{{\frak p},\frak v}^{\rm s},\vec{\frak w}_{{\frak p},{\frak v}}^{\rm s})]$
in $\mathcal M^{\rm s}_{\ell_{\frak v}}$.
\item[(7)]
Diffeomorphisms $\phi_{{\frak p}}^{{\rm d}} : \mathcal V_{{\frak p}}^{{\rm d}} \times \Sigma_{{\frak p}}^{{\rm d}}
\to \pi^{-1}(\mathcal V_{{\frak p}}^{{\rm d}})$,
$\phi_{{\frak p},\frak v}^{\rm s} : \mathcal V_{{\frak p},\frak v}^{\rm s} \times \Sigma_{{\frak p},\frak v}^{\rm s}
\to \pi^{-1}(\mathcal V_{{\frak p},\frak v}^{\rm s})$ which commute with projections.
\item[(8)]
Analytic families of coordinates $\vec{\varphi}_{\frak p}$ at (interior) nodes which are compatible with
the trivialization in (7) in the sense of \cite[Definition 3.7]{diskconst1}.
\item[(9)]
Let $\frak s_{i,{\frak p}}^{{\rm d}} : \mathcal V_{{\frak p}}^{{\rm d}} \to \mathcal C_{0,\ell_{\rm d}}^{\rm d}$,
$\frak s_{i,{\frak p},{\frak v}}^{\rm s} : \mathcal V_{{\frak p},{\frak v}}^{\rm s} \to
\mathcal C_{\ell_{\frak v}}^{\rm s}$
be the sections assigning the $i$-th marked point. (See \cite[Section 2]{diskconst1}.)
Then
$$\aligned
\frak s_{i,{\frak p}}^{{\rm d}}({\frak p}')
&= \phi_{{\frak p}}^{{\rm d}}({\frak p}',\frak s_{i,{\frak p}}^{{\rm d}}({\frak p})),\\
\frak s_{i,{\frak p},{\frak v}}^{\rm s}({\frak p}')
&= \phi_{{\frak p},{\frak v}}^{\rm s}({\frak p}',\frak s_{i,{\frak p},{\frak v}}^{\rm s}({\frak p})).
\endaligned$$
In other words, the trivialization (7) respects the marked points.
\item[(10)]
The trivialization (7) is compatible with the action of extended automorphism
group of $\frak p$, which is induced by (5).
\end{enumerate}
\end{defn}

\begin{defn}\label{defn631}
Let type I strong stabilization data
$((\vec{\frak w}_{\frak p},\vec{\mathcal N}_{\frak p}), (\vec{\mathcal V_{\frak p}},\vec{\phi_{\frak p}}),\vec{\varphi}_{\frak p})$
be chosen. Then
an {\it obstruction space} $E_{\frak p}$ at ${\frak p}$ is defined to be
a finite dimensional subspace of
$C^{\infty}(\Sigma_{\frak p},u_{\frak p}^*TX\otimes \Lambda^{0,1})$ satisfying the following properties:
\begin{enumerate}
\item[(1)]
The union of the supports of the elements of $E_{\frak p}$, which we denote by
${\rm Supp}(E_{\frak p})$, is disjoint from the boundary, nodes, and marked points,
but contains $\vec{\frak w}_{\frak p}$.
\item[(2)]
We consider the operator
$$
D_{u_{\frak p}}\overline{\partial} :
W^2_{m+1}((\Sigma_{\frak p},\partial\Sigma_{\frak p});u_{\frak p}^*TX,u_{\frak p}^*TL)
\to L^2_m(\Sigma_{\frak p},u_{\frak p}^*TX\otimes \Lambda^{0,1})
$$
as in \cite[(5.1)]{diskconst1}.
Then
$$
{\rm Im} D_{u_{\frak p}}\overline{\partial} + E_{\frak p}
= L^2_m(\Sigma_{\frak p},u_{\frak p}^*TX\otimes \Lambda^{0,1}).
$$
\item[(3)]
Let
$$
{\mathcal{EV}} : W^2_{m+1}((\Sigma_{\frak p},\partial\Sigma_{\frak p});u_{\frak p}^*TX,u_{\frak p}^*TL)
\to T_{u_{\frak p}(z_0)}L
$$
be the linearized evaluation map at $z_0$. Then its restriction
$$
{\mathcal{EV}} : (D_{u_{\frak p}}\overline{\partial})^{-1}(E_{\frak p}) \to T_{u_{\frak p}(z_0)}L
$$
is surjective.
\item[(4)]
$E_{\frak p}$ is invariant under the action of the extended automorphism group
${\rm Aut}^+({\frak p})$.
\item[(5)]
The action of ${\rm Aut}({\frak p})$ on $(D_{u_{\frak p}}\overline{\partial})^{-1}(E_{\frak p})/ {\frak {aut}} (\Sigma_{\frak p}, \vec z_{\frak p}, \vec {\frak z}_{\frak p})$ is effective.
\item[(6)]
If $u_{\frak p}$ is constant on an irreducible component of
$\Sigma_{\frak p}$,
then the support ${\rm Supp}(E_{\frak p})$ is disjoint from this
irreducible component.
\end{enumerate}
We consider type I strong stabilization data
$((\vec{\frak w}_{\frak p},\vec{\mathcal N}_{\frak p}),(\vec{\mathcal V_{\frak p}},\vec{\phi_{\frak p}}),\vec{\varphi}_{\frak p})$
together
with an obstruction space $E_{\frak p}$ and write
$$
\Xi_{\frak p}
= ((\vec{\frak w}_{\frak p},\vec{\mathcal N}_{\frak p}),(\vec{\mathcal V_{\frak p}},\vec{\phi_{\frak p}}),
\vec{\varphi}_{\frak p},E_{\frak p}).
$$
\end{defn}
\begin{shitu}\label{situ64}
Suppose we are in Situation \ref{situ61} and $\Xi_{\frak p}$ is given.
Let
\begin{equation}\label{formnew74}
{\bf x} \in B_{\epsilon_0}({\mathcal X},{\bf p}).
\end{equation}
\par
Let $\epsilon_2$, $\epsilon_1$ be the constants given as in (\ref{formnew71}), (\ref{formnew72})
and $\epsilon_0$ be as in (\ref{formnew74}).
For a positive number $\delta$ we will take
positive constants $\epsilon_2(\delta,{\frak p},\Xi_{\frak p})$,
$\epsilon_1(\delta,{\bf q},{\frak p},\Xi_{\frak p})$,
$\epsilon_0(\delta,{\bf p},{\bf q},{\frak p},\Xi_{\frak p})$
which depend on the data in the parenthesis.
We will assume
\begin{equation}\label{formula62}
\aligned
\epsilon_2 &< \epsilon_2(\delta,{\frak p},\Xi_{\frak p}), \\
 \epsilon_1 &< \epsilon_1(\delta,{\bf q},{\frak p},\Xi_{\frak p}), \\
\epsilon_0 &< \epsilon_0(\delta,{\bf p},{\bf q},{\frak p},\Xi_{\frak p}).
\endaligned
\end{equation}
We denote by $\Sigma_{{\bf q},{\rm v}}$ the source curve 
of the irreducible 
component ${\bf q}_{\rm v}$.  Here ${\rm v}$ and ${\bf q}_{\rm v}$ 
are as in Situation \ref{situ61}.
$\diamond$
\end{shitu}
\begin{lem}\label{lem65}
For any sufficiently small $\delta > 0$ there exist
positive numbers $\epsilon_2(\delta,{\frak p},\Xi_{\frak p})$,
$\epsilon_1(\delta,{\bf q},{\frak p},\Xi_{\frak p})$,
$\epsilon_0(\delta,{\bf p},{\bf q},{\frak p},\Xi_{\frak p})$
with the following properties:
\par
Suppose we are in Situation \ref{situ64}, especially  we assume (\ref{formula62}). Then
there exists a unique collection of marked points
$$
\vec{\frak w}_{{\bf x};{\frak p}},
\quad
\vec{\frak w}_{{\bf p};{\frak p}},
\quad
\vec{\frak w}_{{\bf q};{\frak p}}
$$
such that
\begin{enumerate}
\item
${\frak w}_{{\bf x};{\frak p},i} \in \Sigma_{\bf x}$,
${\frak w}_{{\bf p};{\frak p},i} \in \Sigma_{\bf p}$,
${\frak w}_{{\bf q};{\frak p},i} \in \Sigma_{{\bf q},{\rm v}}$ for $i=1, \ldots, \ell'$.
\item
$u_{\bf x}({\frak w}_{{\bf x};{\frak p},i}) \in \mathcal N_{{\frak p},i}$,
$u_{\bf p}({\frak w}_{{\bf p};{\frak p},i}) \in \mathcal N_{{\frak p},i}$,
$u_{{\bf q}}({\frak w}_{{\bf q};{\frak p},i}) \in \mathcal N_{{\frak p},i}$.
\item
$$
\aligned
d((\Sigma_{\frak p},z_{\frak p},\frak z_{\frak p}\cup \vec {\frak w}_{\frak p}),(\Sigma_{{\bf q},{\rm v}},
z_{{\bf q},{\rm v}},\frak z_{{\bf q},{\rm v}}\cup \vec {\frak w}_{{\bf q};{\frak p}})) &< \delta, \\
d((\Sigma_{{\bf q}},z_{{\bf q}},\frak z_{{\bf q}}\cup\vec {\frak w}_{{\bf q};{\frak p}}),(\Sigma_{\bf p},
z_{{\bf p}},\frak z_{{\bf p}}\cup\vec {\frak w}_{{\bf p};{\frak p}})) &< \delta, \\
d((\Sigma_{{\bf p}},z_{{\bf p}},\frak z_{{\bf p}}\cup\vec {\frak w}_{{\bf p},{\frak p}}),(\Sigma_{\bf x},
z_{{\bf x}},\frak z_{{\bf x}}\cup\vec {\frak w}_{{\bf x},{\frak p}})) &< \delta.
\endaligned
$$
Here $d$'s are the metrics on various moduli spaces of marked stable curves. We fix such metrics.
$z_{{\bf q},{\rm v}}$ (resp. $\frak z_{{\bf q},{\rm v}}$) are boundary  (resp. interior)
marked or nodal  points of ${\bf q}$ contained in $\Sigma_{{\bf q},{\rm v}}$.
\end{enumerate}
\end{lem}
\begin{proof}
We first find $\vec {\frak w}_{{\bf q};{\frak p}}$ by the Implicit Function Theorem
and the property ${\bf q}_{\rm v} \in B_{\epsilon_2}({\mathcal X},{\frak p})$.
(See \cite[Lemma 9.9]{diskconst1}).
Then using  $\vec {\frak w}_{{\bf q};{\frak p}}$ we can find
$\vec {\frak w}_{{\bf x};{\frak p}}$,
$\vec {\frak w}_{{\bf p};{\frak p}}$ by
the Implicit Function Theorem
and (\ref{formnew71}), (\ref{formnew72}) and (\ref{formnew74}).
\end{proof}
We consider the forgetful map
$$
\frak{forget}_{k+1,\ell+\ell';0,\ell'} : \mathcal M_{k+1,\ell+\ell'}^{\rm d} \to \mathcal M_{0,\ell'}^{\rm d}
$$
which forgets all the boundary marked points and the {\it first} $\ell$ interior
marked points. (Here $\ell'$ is the cardinality of $\vec{\frak w}_{\frak p}$.)
\begin{lem}\label{lem666}
Under the assumption of Lemma \ref{lem65}\footnote{
We may replace positive numbers $\epsilon_2(\delta,{\frak p},\Xi_{\frak p})$,
$\epsilon_1(\delta,{\bf q},{\frak p},\Xi_{\frak p})$,
$\epsilon_0(\delta,{\bf p},{\bf q},{\frak p},\Xi_{\frak p})$ by 
smaller numbers if necessary. The same remark applies to 
Lemma \ref{lem6767}.}
we have
$$
\aligned
d((\Sigma_{\frak p},\vec {\frak w}_{\frak p}),\frak{forget}_{k_{\rm d}+1,\ell_{\rm d};0,\ell'}(\Sigma_{{\bf q},{\rm v}},
z_{{\bf q},{\rm v}},\vec {\frak w}_{{\bf q};{\frak p}})) &< \delta, \\
d((\Sigma_{\frak p},\vec {\frak w}_{\frak p}),
\frak{forget}_{k+1,\ell+\ell';0,\ell'}(\Sigma_{{\bf q}},z_{{\bf q}},\frak z_{{\bf q}}
\cup\vec {\frak w}_{{\bf q},{\frak p}})) &< \delta, \\
d((\Sigma_{\frak p},\vec {\frak w}_{\frak p}),\frak{forget}_{k+1,\ell+\ell';0,\ell'}(\Sigma_{{\bf p}},z_{{\bf p}},\frak z_{{\bf p}}\cup\vec {\frak w}_{{\bf p},{\frak p}})) &< \delta, \\
d((\Sigma_{\frak p},\vec {\frak w}_{\frak p}),\frak{forget}_{k+1,\ell+\ell';0,\ell'}(\Sigma_{{\bf x}},z_{{\bf x}},\frak z_{{\bf x}}\cup\vec {\frak w}_{{\bf x},{\frak p}})) &< \delta.
\endaligned
$$
\end{lem}
\begin{figure}[h]
\centering
\includegraphics[scale=0.8]{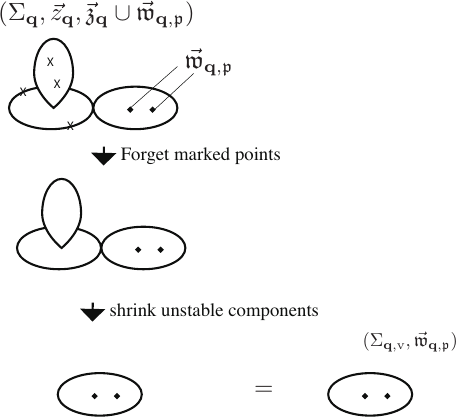}
\caption{Forgetting marked points $\vec z$, $\vec{\frak z}$}
\label{zu3}
\end{figure}
\begin{proof}
The first inequality is a consequence of Lemma \ref{lem65} (3) and the continuity of
the forgetful map.
The second inequality is a consequence of the first inequality  and the equality
$$
\frak{forget}_{k+1,\ell+\ell';0,\ell'}(\Sigma_{{\bf q}},z_{{\bf q}},\frak z_{{\bf q}}\cup\vec {\frak w}_{{\bf q},{{\frak p}}})
= (\Sigma_{{\bf q},{\rm v}},
\vec {\frak w}_{{\bf q};{\frak p}}).
$$
See Figure \ref{zu3}.
We remark that this is the place where we use the fact we are studying bordered Riemann surfaces of genus 0.
\par
The third and fourth inequalities then follow from the continuity of
$\frak{forget}_{k+1,\ell+\ell';0,\ell'}$.
\end{proof}
 
We use the data given in Definition \ref{defn63} (7), (8) to define a smooth open embedding $\Phi$ below:
\begin{equation}\label{fprmmapF}
\Phi : \mathcal V_{{\frak p}}^{{\rm d}} \times
\prod_{\frak v}\mathcal V_{{\frak p},{\frak v}}^{\rm s}
\times (D^2(c))^{m_{\rm s}} \to \mathcal M_{0,\ell'}^{\rm d}
\end{equation}
as in \cite[(3.5)]{diskconst1}.
The first and second factors parametrize the 
deformation of each irreducible component 
and the third factor is the gluing parameter, where
$m_{\rm s}$ is the number of interior nodes of $\Sigma_{\frak p}$.

By Lemma \ref{lem666}  there exist elements
$\frak x \in \mathcal V_{{\frak p}}^{{\rm d}} \times
\prod_{\frak v}\mathcal V_{{\frak p},{\frak v}}^{\rm s}$ and
$\vec{\epsilon} \in (D^2(c))^{m_{\rm s}}$
such that
$$
\frak{forget}_{k+1,\ell+\ell';0,\ell'}(\Sigma_{{\bf x}},z_{{\bf x}},\frak z_{{\bf x}}
\cup\vec {\frak w}_{{\bf x},{\frak p}}) = \Phi(\frak x;\vec{\epsilon}).
$$

Let $\Sigma_{{\frak p}}(\epsilon)$ be the $\epsilon$-thick part of $\Sigma_{{\frak p}}$.
Using the data given in Definition \ref{defn63} (7), (8),
we can apply \cite[Lemma 3.9]{diskconst1} to obtain 
$$
\widehat \Phi_{\frak r,\vec{\epsilon}}: \mathcal V \times \Sigma_{{\frak p}}(\epsilon)
\to \mathcal C^{\rm d}_{0,\ell}
$$
and smooth embeddings,
\begin{equation}\label{form6363}
\aligned
\widehat{\Phi}_{{\bf q},{\frak p};\Xi_{\frak p}} &:
{\rm Supp}(E_{\frak p}) \to \Sigma_{{\bf q}}, \\
\widehat{\Phi}_{{\bf p},{\frak p};\Xi_{\frak p}} &:
{\rm Supp}(E_{\frak p}) \to \Sigma_{{\bf p}}, \\
\widehat{\Phi}_{{\bf x},{\frak p};\Xi_{\frak p}} &:
{\rm Supp}(E_{\frak p}) \to \Sigma_{{\bf x}}.
\endaligned
\end{equation}
For example, the third map is a restriction of the map $\widehat{\Phi}_{\frak x;\vec{\epsilon}}$
appearing in \cite[Lemma 3.9]{diskconst1}.

\begin{rem}\label{rem78}
In the case where $(\Sigma_{\frak p},\vec{\frak w}_{\frak p})$ is the point appearing in
Remark \ref{rem73} and Figure \ref{Figure2-5},  the map
(\ref{fprmmapF}) becomes
\begin{equation}\label{fprmmapF2}
\Phi : \prod_{\frak v}\mathcal V_{{\frak p},{\frak v}}^{\rm s}
\times [0,c) \times (D^2(c))^{m_{\rm s}-1} \to \mathcal M_{0,\ell'}^{\rm d}.
\end{equation}
Here $[0,c)$ is the space of smoothing parameters of the node which lies on the disk component.
(See \cite[pages 590-591]{fooobook2}.)
By Definition \ref{defn631} (6),  ${\rm Supp}(E_{\frak p})$, the support of the
obstruction bundle, is disjoint from the disk component.
(This is because $u_{\frak p}$ is constant on the disk component as we explained in
Remark \ref{rem73}.)
We use this fact to show that (\ref{form6363}) is also defined in this case.
\end{rem}
\begin{lem}\label{lem6767}
We can choose $\epsilon_2(\delta,{\frak p},\Xi_{\frak p})$,
$\epsilon_1(\delta,{\bf q},{\frak p},\Xi_{\frak p})$,
$\epsilon_0(\delta,{\bf p},{\bf q},{\frak p},\Xi_{\frak p})$ so that,
under the assumption of Lemma \ref{lem65} the maps  (\ref{form6363})
have the following properties:
\begin{enumerate}
\item
Denote by $d_{C^2,{\rm Supp}(E_{\frak p})}$ the $C^2$ distance
between two maps: ${\rm Supp}(E_{\frak p}) \to X$.
Then the following inequalities hold:
\begin{equation}\label{form64640}
\aligned
&d_{C^2,{\rm Supp}(E_{\frak p})}
\left(
u_{{\bf q}} \circ \widehat{\Phi}_{{\bf q},{\frak p};\Xi_{\frak p}},u_{{\frak p}}
\right)
< \delta,
\\
&d_{C^2,{\rm Supp}(E_{\frak p})}
\left(
u_{{\bf p}} \circ \widehat{\Phi}_{{\bf p},{\frak p};\Xi_{\frak p}},u_{\frak p}
\right)
< \delta,
\\
&d_{C^2,{\rm Supp}(E_{\frak p})}
\left(
u_{{\bf x}} \circ \widehat{\Phi}_{{\bf x},{\frak p};\Xi_{\frak p}},u_{\frak p}
\right)
< \delta.
\endaligned
\end{equation}
\item
The next inequalities hold for each $x \in {\rm Supp}(E_{\frak p})$:
\begin{equation}\label{form65650}
\aligned
\vert \partial \, \widehat{\Phi}_{{\bf q},{\frak p};\Xi_{\frak p}} (x)\vert & >
10\vert \overline\partial \, \widehat{\Phi}_{{\bf q},{\frak p};\Xi_{\frak p}}(x)\vert, \\
\vert \partial \, \widehat{\Phi}_{{\bf p},{\frak p};\Xi_{\frak p}} (x)\vert & >
10\vert \overline\partial \, \widehat{\Phi}_{{\bf p},{\frak p};\Xi_{\frak p}} (x)\vert,
\\
\vert \partial \, \widehat{\Phi}_{{\bf x},{\frak p};\Xi_{\frak p}}(x)\vert
& >
10\vert \overline\partial \, \widehat{\Phi}_{{\bf x},{\frak p};\Xi_{\frak p}}(x)\vert.
\endaligned
\end{equation}
\end{enumerate}
\end{lem}
\begin{proof}
The inequalities
(\ref{form65650}) are easy consequences of the definition of  $\widehat{\Phi}_{*,\frak p,\Xi_{\frak p}}$
in (\ref{form6363})
and Lemma \ref{lem666}. The inequalities
(\ref{form64640}) are consequences of Lemma \ref{lem666}, and the fact that
${\rm Supp}(E_{\frak p})$ is away from the node. The proof of (\ref{form64640}) is the same as 
that of \cite[Lemma 4.3.75]{foootoric32}.
We repeat the proof for completeness' sake.
We prove the second inequality only. The proofs of other two inequalities
are similar.
Lemma \ref{lem666} implies that
$(\Sigma_{\bf p},\vec{\frak w}_{{\bf p},\frak p}) \in B_{\epsilon}((\Sigma_{\bf q},\vec{\frak w}_{{\bf q},\frak p}))$
for some $\epsilon$ going to zero as $\epsilon_1,\epsilon_2$ go to zero,
where we fix and use a metric on the compactified moduli space
${\mathcal M}_{0,\ell_{\frak p}}^{\rm d}$ of disks
with $\ell_{\frak p}$ interior marked points
to define $B_{\epsilon}((\Sigma_{\bf q},\vec{\frak w}_{{\bf q},\frak p}))$.
(Here $\ell_{\frak p}$ is the cardinality of the set $\vec{\frak w}_{\frak p}$.)
We take an analytic family of coordinates at nodal points which has the following properties.
\begin{defn}\label{defn122812}
$ $
\begin{enumerate}
\item
A holomorphic embedding
$D^2 \to S^2$ is said to be {\it extendable}
\index{extendable} if it is a restriction
of a biholomorphic map $S^2 \to S^2$.
\item
A holomorphic embedding
$D^2 \to D^2$ is said to be {\it extendable}
\index{extendable} if it is a restriction
of biholomorphic map $D^2(R) \to D^2$
for some $R > 1$.
\item
A holomorphic embedding
$(D_{\ge 0}^2,\partial D_{\ge 0}^2) \to (D^2,\partial D^2)$
is said to be {\it extendable} if its double is extendable
in the sense of $(1)$.
Here $D_{\ge 0}^2=\{ z \in D^2 ~\vert~ {\rm Im}\, z \ge 0 \}$.
\item
An analytic family of coordinates is said to be {\it extendable}
if its members are extendable in the sense of (1),(2) or (3).
\end{enumerate}
\end{defn}
Note that $d({\bf q}_{\rm v},\frak p)$ is assumed to be small.
So there exists an embedding
$$
\widehat{\Phi}_{{\bf q}_{\rm v},{\frak p};\Xi_{\frak p}} :
{\rm Supp}(E_{\frak p}) \to \Sigma_{{\bf q}_{\rm v}}
$$
such that
\begin{equation}\label{form711}
d_{C^2,{\rm Supp}(E_{\frak p})}
\left(
u_{{\bf q}} \circ \widehat{\Phi}_{{\bf q}_{\rm v},{\frak p};\Xi_{\frak p}},u_{{\frak p}}
\right)
< \delta/10.
\end{equation}
(Actually
$\widehat{\Phi}_{{\bf q}_{\rm v},{\frak p};\Xi_{\frak p}} =
\widehat{\Phi}_{{\bf q},{\frak p};\Xi_{\frak p}}$.)
Using $(\Sigma_{\bf p},\vec{\frak w}_{{\bf p},\frak p}) \in B_{\epsilon}((\Sigma_{\bf q},\vec{\frak w}_{{\bf q},\frak p}))$ and
data (7)(8) of Definition \ref{defn63}, $\Xi_{(\Sigma_{\bf q},\vec{\frak w}_{{\bf q},\frak p})}$, at $(\Sigma_{\bf q},\vec{\frak w}_{{\bf q},\frak p})$,
we obtain an embedding
$$
\widehat{\Phi}_{{\bf p},{\bf q};\Xi_{(\Sigma_{\bf q},\vec{\frak w}_{{\bf q},\frak p})}} :
\Sigma_{\bf q} \setminus
\text{(neck region)} \to \Sigma_{\bf p}
$$
such that
\begin{equation}\label{form7122}
d_{C^2,\widehat{\Phi}_{{\bf q}_{\rm v},{\frak p};\Xi_{\frak p}}({\rm Supp}(E_{\frak p}))}
\left(
u_{{\bf p}} \circ \widehat{\Phi}_{{\bf p},{\bf q};\Xi_{(\Sigma_{\bf q},\vec{\frak w}_{{\bf q},\frak p})}},u_{{\bf q}}
\right)
< \delta/10,
\end{equation}
where $\widehat{\Phi}_{{\bf p},{\bf q};\Xi_{(\Sigma_{\bf q},\vec{\frak w}_{{\bf q},\frak p})}}$
is a map defined in the way similar to \eqref{form6363}
using the data
$\Xi_{(\Sigma_{\bf q},\vec{\frak w}_{{\bf q},\frak p})}$.
We put
$$
(\Sigma_{\bf p},\vec{\frak w}_{{\bf p},\frak p}) = \Phi(\vec{\frak x},\vec{\rho}),
$$
where $\vec{\frak x}$ is the parameter to 
deform the complex structure of the irreducible components of
${\bf q}$ and $\vec{\rho}$ is a smoothing parameter of the nodes of ${\bf q}$.
We use $\Xi_{(\Sigma_{\bf q},\vec{\frak w}_{{\bf q},\frak p})}$ to define the map
$\Phi$ in a way similar to (\ref{fprmmapF}).
\par
We take $\vec{\frak x}_0$ such that
$\Phi(\vec{\frak x}_0,\vec{0}) = (\Sigma_{\bf q},\vec{\frak w}_{{\bf q},{\frak p}})$.
(Namely $\vec{\frak x}_0$ corresponds to the complex structure of $\Sigma_{\bf q}$ itself.)
\begin{sublem}\label{sublem777}
If $(\Sigma_{\bf p},\vec{\frak w}_{{\bf p},{\frak p}}) = \Phi(\vec{\frak x}_0,\vec{
\rho})$ for some $\vec{
\rho}$ and if the analytic families of coordinates, which are part of the data
$\Xi_{(\Sigma_{\bf q},\vec{\frak w}_{{\bf q},\frak p})}$, are extendable,
then
$$
\widehat{\Phi}_{{\bf p},{\bf q};\Xi_{(\Sigma_{\bf q},\vec{\frak w}_{{\bf q},\frak p})}}\circ \widehat{\Phi}_{{\bf q}_{\rm v},{\frak p};\Xi_{\frak p}}
=
\widehat{\Phi}_{{\bf p},{\frak p};\Xi_{\frak p}}
$$
on ${\rm Supp}(E_{\frak p})$.
\end{sublem}
\begin{proof}
There exists a biholomorphic map ${\frak I} : (\Sigma_{\bf p},\vec{\frak w}_{{\bf p},{\frak p}})
\to (\Sigma_{{\bf q}_{\rm v}},\vec{\frak w}_{{\bf q},{\frak p}})$
such that ${\frak I} \circ \widehat{\Phi}_{{\bf p},{\bf q};\Xi_{(\Sigma_{\bf q},\vec{\frak w}_{{\bf q},\frak p})}}$
is the identity map on $\Sigma_{{\bf q}_{\rm v}}$ minus a small neighborhood of
boundary nodes, 
if all the components of $\vec{\rho}$ are nonzero.
(If some components are zero the domain is smaller but still 
contains the image of $\widehat{\Phi}_{{\bf p},{\frak p};\Xi_{\frak p}}$.) This is a consequence of the extendability of the analytic family of
coordinates there.
(See Figure \ref{lastfigure}
and also \cite[Lemma 12.33]{FuFu6}.)
The sublemma is an easy consequence of this fact.
\end{proof}
\begin{figure}[h]
\centering
\includegraphics[scale=0.8]{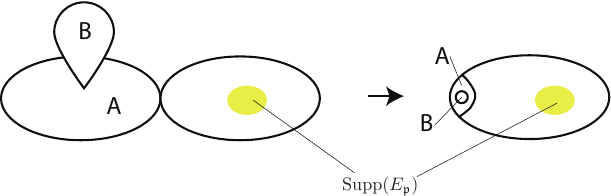}
\caption{Sublemma \ref{sublem777}}
\label{lastfigure}
\end{figure}
In the case of $(\Sigma_{\bf p},\vec{\frak w}_{{\bf p},{\frak p}}) = \Phi(\vec{\frak x}_0,\vec{
\rho})$ the second inequality of (\ref{form64640}) follows from
(\ref{form711}), (\ref{form7122}) and Sublemma \ref{sublem777}.
The general case follows by taking $\epsilon_1$
(which estimates $d({\bf p},{\bf q})$ and hence $d(\vec{\frak x},\vec{\frak x}_0)$)
sufficiently small compared to the distance between ${\rm Supp}(E_{\frak p})$ and nodes.
\end{proof}
Using Lemma \ref{lem6767} we construct a map
\begin{equation}\label{form66666}
\aligned
\mathcal P_{{\bf x},{\bf p},{\bf q},{\frak p};\Xi_{\frak p}} : &
C^2(\widehat{\Phi}_{{\bf x},{\frak p};\Xi_{\frak p}}(U({\rm Supp}(E_{\frak p})));u_{\bf x}^*TX \otimes \Lambda^{0,1}) \\
&\to C^2(U({\rm Supp}(E_{\frak p}));u_{\frak p}^*TX \otimes \Lambda^{0,1})
\endaligned
\end{equation}
as follows.
(Here $U({\rm Supp}(E_{\frak p}))$ is a sufficiently small open neighborhood of
${\rm Supp}(E_{\frak p})$.)
The map (\ref{form66666}) is similar to the map \cite[(8.1)]{diskconst1}.
\par
Let $z \in U({\rm Supp}(E_{\frak p}))$.
By (\ref{form64640}) $d(u_{\frak p}(z),u_{\bf x}(\hat\Phi_{{\bf x},{\frak p};\Xi_{\frak p}}(z)))$
is smaller than the injectivity radius of $X$.\footnote{
Here $\hat\Phi_{{\bf x},{\frak p};\Xi_{\frak p}}$ is as in (\ref{form6363}).  It actually depends not only 
on ${\bf x}$ and $\frak p$ but also on 
$\bf p$ and $\bf q$. Such an example is given in Remark \ref{rem714}.}
So there exists a
unique minimal geodesic joining them.
We fix a unitary connection of $TX$. Using the parallel transport along the
minimal geodesic we obtain
$T_{\hat\Phi_{{\bf x},{\frak p};\Xi_{\frak p}}(z)}X \to T_{u_{\frak p}(z)}X$.
We thus obtain a bundle map
\begin{equation}\label{form67}
u_{\bf x}^*TX \to u_{\frak p}^*TX
\end{equation}
over $\hat\Phi_{{\bf x},{\frak p};\Xi_{\frak p}}^{-1}$.
\par
On the other hand, the complex linear part of the differential of $\hat\Phi_{{\bf x},{\frak p}}^{-1}$
induces a bundle map
\begin{equation}\label{form68}
\Lambda^{0,1}(\Sigma_{\bf x}) \to \Lambda^{0,1}(\Sigma_{\frak p})
\end{equation}
over $\hat\Phi_{{\bf x},{\frak p};\Xi_{\frak p}}^{-1}$.
((\ref{form65650}) implies that this map is a bundle isomorphism on $U({\rm Supp}(E_{\frak p}))$.)
The bundle map which is a tensor product (over $\C$) of (\ref{form67})
and (\ref{form68}) induces the map (\ref{form66666}),
which is an isomorphism.

We also consider ${\mathcal P}_{\mathbf p,\mathbf p,\mathbf q,\mathfrak p;\Xi_{\mathfrak p}}$
and ${\mathcal P}_{\mathbf q,\mathbf q,\mathbf q,\mathfrak p;\Xi_{\mathfrak p}}$.
\par
We define:
\begin{equation}\label{form6969}
E_{{\bf p},{\bf q},{\frak p};\Xi_{\frak p}}({\bf x})
=
(\mathcal P_{{\bf x},{\bf p},{\bf q},{\frak p};\Xi_{\frak p}})^{-1}(E_{\frak p}).
\end{equation}
\begin{lem}\label{lem618}
${\bf x} \mapsto E_{{\bf p},{\bf q},{\frak p};\Xi_{\frak p}}({\bf x})$
is smooth in the sense of \cite[Definition 8.7]{diskconst1}.
\end{lem}
\begin{proof}
The proof is the same as the one given in  \cite[Subsection 11.2]{diskconst1} and so is omitted.
\end{proof}

\subsection{Iteration of the construction of Subsection \ref{subsec:staint}}
\label{subsec:itera2}

The obstruction bundle data, which we will construct for the proof of
Theorem \ref{thm43}, is obtained by taking an appropriate
direct sum of the ones defined as in (\ref{form6969}).
More precisely we also need its variant which includes the iteration of the
process appearing in Situation \ref{situ61}.
In this subsection we will explain this variant.
In the case $n=1$, Situation \ref{situ69} becomes Situation \ref{situ61}.
\par
Let $\frak T \in \mathcal G(k+1,\ell,\beta)$.
Let $\frak T' > \frak T$, ${\rm v}$ an interior vertex of $\frak T'$, and
let $\pi : \mathcal T \to \mathcal T'$ be the projection
canonically defined by a sequence of edge contraction.
We say $\frak S$ is a {\it subtree} of $\frak T$ if it is obtained
from $\pi^{-1}({\rm v})$ with data induced by $\frak T$.
A subtree $\frak S$ is an element of some  $\mathcal G(k'+1,\ell',\beta')$.
If ${\bf p} \in \mathcal M^{\circ}_{k+1,\ell}(X,L,J;\beta)(\frak T)$
and $\frak S$ is a subtree of $\frak T$, we obtain ${\bf p}_{\frak S}
\in \mathcal M_{k'+1,\ell'}(X,L,J;\beta')$
by gluing ${\bf p}_{\rm v}$ for vertices ${\rm v}$ of $\frak S$ in the way
described by $\frak S$. See Lemma \ref{lem3232} and its proof.

\begin{shitu}\label{situ69}
Consider the data ${\bf p}$, ${\bf q}_j$, ${\frak p}$,
$\frak T_j$, $\frak S_j$, $\epsilon_{1,j}>0$ ($j=1,\dots,n$ for some $n=0,1,2,\dots$\footnote{When $n=0$, $j$ is absent.}) and
$\epsilon_{2}>0$ with the following properties.
(See Figure \ref{zu4}.)
\begin{enumerate}
\item
${\bf p} \in \mathcal M_{k+1,\ell}(X,L,J;\beta)$.
\item
${\bf q}_j \in \mathcal M^{\circ}_{k_j+1,\ell_j}(X,L,J;\beta_j)(\frak T_j)$.
Here
$\frak T_j \in \mathcal G(k_j+1,\ell_j,\beta_j)$.
\item
$d({\bf p},{\bf q}_1) < \epsilon_{1,1}$.
($(k,\ell,\beta) = (k_1,\ell_1,\beta_1)$ in particular.)
\item
$\frak S_j$ is a subtree of $\frak T_j$. We define ${\bf q}_{j,\frak S_j}$
as above. We assume ${\bf q}_{j,\frak S_j} \ne {\bf q}_{j}$ for $j=1,\dots,n-1$.
\item
For $j \le n-1$
we require ${\bf q}_{j,\frak S_j} \in \mathcal M_{k_{j+1}+1,\ell_{j+1}}(X,L,J;\beta_{j+1})$
and
\begin{equation}
d({\bf q}_{j,\frak S_j}, {\bf q}_{j+1}) < \epsilon_{1,j+1}.
\end{equation}
\item
We require ${\frak p}, {\bf q}_{n,\frak S_{n}} \in \mathcal M^{\circ}_{k_{n+1}+1,\ell_{n+1}}(X,L,J;\beta_{n+1})$ and
\begin{equation}
d({\bf q}_{n,\frak S_{n}},{\frak p}) < \epsilon_2.
\end{equation}
(Note that $\frak S_{n}$ has only one interior vertex by this condition.)
\item
We require $\frak S_j > \frak T_{j+1}$. $\diamond$
\end{enumerate}
\end{shitu}
\begin{figure}[h]
\centering
\includegraphics[scale=0.8]{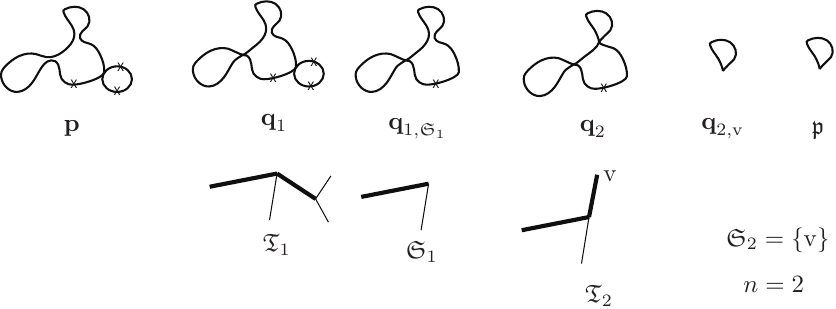}
\caption{Situation \ref{situ69}}
\label{zu4}
\end{figure}
We take type I strong stabilization data and an obstruction space at $\frak p$,
which we denote by $\Xi_{\frak p}$. (See Definition \ref{defn63}.)
Let
\begin{equation}\label{form610}
{\bf x} \in B_{\epsilon_0}({\mathcal X},{\bf p}).
\end{equation}
Let $\delta_j$, $j=0,\dots,n+1$ be 
positive numbers.
In the sequel we will find positive numbers:
\begin{equation}\label{form611}
\aligned
&\epsilon_{0}(\delta_{0};{\bf p},{\bf q}_{1},\dots,{\bf q}_n,{\frak p},\Xi_{\frak p}), \\
&\epsilon_{1,j}(\delta_{j};{\bf q}_{j},\dots,{\bf q}_n,{\frak p},\Xi_{\frak p}), (j=1,2,\dots,n)\\
&\epsilon_{2}(\delta_{n+1};{\frak p},\Xi_{\frak p}),
\endaligned
\end{equation}
which depend on the data in the parenthesis
and assume
\begin{equation}\label{form612}
\aligned
\epsilon_0 &< \epsilon_{0}(\delta_{0};{\bf p},{\bf q}_{1},\dots,{\bf q}_n,{\frak p},\Xi_{\frak p}), \\
\epsilon_{1,j} &< \epsilon_{1,j}(\delta_{j};{\bf q}_{j},\dots,{\bf q}_n,{\frak p},\Xi_{\frak p}), \\
\epsilon_{2} &< \epsilon_{2}(\delta_{n+1};{\frak p},\Xi_{\frak p}).
\endaligned
\end{equation}
\begin{lem}\label{lem61060}
There exist positive numbers as in (\ref{form611}) with the following properties:
Suppose we are in Situation \ref{situ69} (1)-(7), (\ref{form610}) and
(\ref{form612}).
Then there exists
a collection of marked points
$$\vec{\frak w}_{{\bf x};\frak p}, \quad \vec{\frak w}_{{\bf p};{\frak p }},
\quad
\vec{\frak w}_{{\bf q}_j,{\frak p}} \enskip (j=1,\dots,n),
$$
such that
\begin{enumerate}
\item
${\frak w}_{{\bf x};{\frak p},i} \in \Sigma_{\bf x}$,
${\frak w}_{{\bf p};{\frak p},i} \in \Sigma_{\bf p}$ and
${\frak w}_{{\bf q}_j;{\frak p},i} \in \Sigma_{{\bf q}_{j,\frak S_j}}$.
\item
$u_{\bf x}({\frak w}_{{\bf x};{\frak p},i}) \in \mathcal N_{{\frak p},i}$,
$u_{\bf p}({\frak w}_{{\bf p};{\frak p},i}) \in \mathcal N_{{\frak p},i}$ and
$u_{{\bf q}_j}({\frak w}_{{\bf q}_j;{\frak p},i}) \in \mathcal N_{{\frak p},i}$.
\item
$$
\aligned
d((\Sigma_{{\bf q}_{j+1}},\vec z_{{\bf q}_{j+1}},\vec{\frak z}_{{\bf q}_{j+1}}\cup \vec{\frak w}_{{\bf q}_{j+1},{\frak p}}),
(\Sigma_{{\bf q}_{j,\frak S_j}},\vec z_{{\bf q}_{j,\frak S_j}},\vec{\frak z}_{{\bf q}_{j,\frak S_j}}\cup
\vec{\frak w}_{{\bf q}_j,{\frak p}})) &< \delta_{j+1}, \\
d((\Sigma_{{\bf q}_1},\vec z_{{\bf q}_1},\vec{\frak z}_{{\bf q}_1}\cup\vec {\frak w}_{{\bf q}_1;{\frak p}}),(\Sigma_{\bf p},
\vec z_{{\bf p}},\vec{\frak z}_{{\bf p}}\cup\vec{\frak w}_{{\bf p};{\frak p}}) &< \delta_{1}, \\
d((\Sigma_{{\bf p}},\vec z_{{\bf p}},\vec{\frak z}_{{\bf p}}\cup\vec{\frak w}_{{\bf p},{\frak p}}),(\Sigma_{\bf x},
\vec{z}_{{\bf x}},\vec{\frak z}_{{\bf x}}\cup\vec{\frak w}_{{\bf x},{\frak p}})) &< \delta_{0}, \\
d((\Sigma_{{\bf q}_n,{\frak S_n}},\vec z_{{\bf q}_n,{\frak S_n}},
\vec{\frak z}_{{\bf q}_n,{\frak S_n}}\cup\vec{\frak w}_{{\bf q}_n,{\frak p}}),(\Sigma_{\frak p},
\vec{z}_{\frak p},\vec{\frak z}_{\frak p}\cup\vec{\frak w}_{{\frak p}})) &< \delta_{n+1}.
\endaligned
$$
Here $d$'s are the metrics on various moduli spaces of marked stable curves which we fix.
$\vec z_{{\bf q}_{j,\frak S_j}}$ (resp. $\vec{\frak z}_{{\bf q}_{j,\frak S_j}}$) are boundary  (resp. interior)
marked or nodal points of ${\bf q}_j$ contained in $\Sigma_{{\bf q}_{j,\frak S_j}}$.
\par
When $n=0$ we have
$$
d((\Sigma_{\bf p},\vec z_{\bf p},
\vec{\frak z}_{\bf p}\cup\vec{\frak w}_{{\bf p};\frak p}),(\Sigma_{\frak p},
\vec{z}_{\frak p},\vec{\frak z}_{\frak p}\cup\vec{\frak w}_{\frak p})) < \delta_1,
$$
instead.
\end{enumerate}
\end{lem}
\begin{proof}
Using the Implicit Function Theorem and the assumptions,
we can find $\vec{\frak w}_{{\bf q}_j,{\frak p}}$ ($j=1,\dots,n$)
which have the required properties, by a downward induction on $j$.
Then we can find $\vec{\frak w}_{{\bf x};{\frak p}}$, $\vec{\frak w}_{{\bf p};{\frak p}}$
in the same way as the proof of Lemma \ref{lem65}.
\end{proof}
Let $\delta'_j$ be a sequence of positive numbers for
$j=0,\dots,n+1$.
\begin{lem}\label{lem611}
We can take $\delta_j = \delta(\delta'_j,{\bf q}_j)$
(for $j\ne 0,n+1$), $\delta_0 = \delta(\delta'_0,\bf p)$,
$\delta_{n+1} = \delta(\delta'_{n+1},\frak p)$ 
with the following properties.
Under the assumption of Lemma \ref{lem61060}\footnote{
We may replace positive numbers as in (\ref{form611}) by 
smaller numbers if necessary, in this subsection. 
The same remark applies to 
Lemma \ref{lem6767rev}.}
we have:
\begin{equation}\label{newform22}
\aligned
d((\Sigma_{\frak p},\vec{\frak w}_{\frak p}),\frak{forget}_{k^1_{j+1}+1,\ell^1_{j+1}+\ell';0,\ell'}
(\Sigma_{{\bf q}_{j,\frak S_j}},
\vec z_{{\bf q}_{j,\frak S_j}},\vec{\frak z}_{{\bf q}_{j,\frak S_j}} \cup \vec {\frak w}_{{\bf q};{\frak p}})) &< \delta'_{j+1}+\dots + \delta'_{n+1}, \\
d((\Sigma_{\frak p},\vec{\frak w}_{\frak p}),\frak{forget}_{k^1_j+1,\ell^1_j+\ell';0,\ell'}(\Sigma_{{\bf q}_j},
\vec z_{{\bf q}_j},
\vec{\frak z}_{{\bf q}_j}\cup\vec{\frak w}_{{\bf q}_j,{\frak p}})) &<  \delta'_{j+1}+\dots + \delta'_{n+1}, \\
d((\Sigma_{\frak p},\vec{\frak w}_{\frak p}),\frak{forget}_{k+1,\ell+\ell';0,\ell'}(\Sigma_{{\bf p}},
\vec z_{{\bf p}},\vec{\frak z}_{{\bf p}}\cup\vec{\frak w}_{{\bf p},{\frak p}})) &<  \delta'_{1}+\dots + \delta'_{n+1}, \\
d((\Sigma_{\frak p},\vec{\frak w}_{\frak p}),\frak{forget}_{k+1,\ell+\ell';0,\ell'}(\Sigma_{{\bf x}},
\vec z_{{\bf x}},\vec{\frak z}_{{\bf x}}\cup\vec{\frak w}_{{\bf x},{\frak p}})) &<  \delta'_{0}+\dots + \delta'_{n+1}.
\endaligned
\end{equation}
\end{lem}
\begin{proof}
Using the fact that
$$
\aligned
&\frak{forget}_{k^1_{j+1}+1,\ell^1_{j+1}+\ell';0,\ell'}(\Sigma_{{\bf q}_{j,\frak S_j}},
\vec z_{{\bf q}_{j,\frak S_j}},\vec{\frak z}_{{\bf q}_{j,\frak S_j}} \cup \vec {\frak w}_{{\bf q}_j;{\frak p}})
\\
&=
\frak{forget}_{k^1_j+1,\ell^1_j+\ell';0,\ell'}(\Sigma_{{\bf q}_j},\vec z_{{\bf q}_j},
\vec{\frak z}_{{\bf q}_j}\cup\vec{\frak w}_{{\bf q}_j,{\frak p}})
\endaligned
$$
and the continuity of the forgetful map,
we can prove the lemma by  Lemma \ref{lem61060} and triangle inequality.
\end{proof}
Let $\delta(\frak p)$ be a positive number depending only on $\frak p$ 
and $\Xi_{\frak p}$,
and $\delta'_j$ positive numbers such that 
\begin{equation}\label{newnew723}
\delta'_{0}+\dots + \delta'_{n+1} < \delta(\frak p).
\end{equation}
We then obtain $\delta_j$ by Lemma \ref{lem611}.
We then apply Lemma \ref{lem61060} to obtain 
$\epsilon_0, \epsilon_{1,j}, \epsilon_{2}$.
Suppose the assumption (and hence the conclusion) of 
Lemma 
\ref{lem61060} is satisfied.
By Lemma \ref{lem611} and the trivialization of the universal family,
we obtain smooth embeddings
\begin{equation}\label{form6363rev}
\aligned
\widehat{\Phi}_{{\bf q}_j,\vec{\bf q},{\frak p};\Xi_{\frak p}} &:
{\rm Supp}(E_{\frak p}) \to \Sigma_{{\bf q}_j}, \\
\widehat{\Phi}_{{\bf p},\vec{\bf q},{\frak p};\Xi_{\frak p}} &:
{\rm Supp}(E_{\frak p}) \to \Sigma_{\bf p}, \\
\widehat{\Phi}_{{\bf x},\vec{\bf q},{\frak p};\Xi_{\frak p}} &:
{\rm Supp}(E_{\frak p}) \to \Sigma_{{\bf x}},
\endaligned
\end{equation}
in the same way as we obtained (\ref{form6363}). (Strictly speaking,
construction of  $\widehat{\Phi}_{{\bf q}_j,\vec{\bf q},{\frak p};\Xi_{\frak p}}$
involves only the subset $\{{\bf q}_{j+1}, \ldots, {\bf q}_n\}$ of
$\vec {\bf q}$ but for the simplicity of notation, we suppress this in our
notation which should not confuse readers.)
Note here we choose  sufficiently small $\delta(\frak p)$ 
depending on $\frak p$.
\begin{lem}\label{lem6767rev}
We can choose the positive numbers (\ref{form611}) so that
under the assumption of Lemma \ref{lem61060} the maps (\ref{form6363rev})
have the following properties:
\begin{enumerate}
\item We have
\begin{equation}\label{form6464}
\aligned
d_{C^2,{\rm Supp}(E_{\frak p})}
\left(
u_{{\bf q}_j} \circ \widehat{\Phi}_{{\bf q}_j,\vec{\bf q},{\frak p};\Xi_{\frak p}},u_{{\frak p}}
\right)
&<  \delta'_{j+1}+\dots + \delta'_{n+1},
\\
d_{C^2,{\rm Supp}(E_{\frak p})}
\left(
u_{{\bf p}} \circ \widehat{\Phi}_{{\bf p},\vec{\bf q},{\frak p};\Xi_{\frak p}},u_{{\frak p}}
\right)
&<  \delta'_{1}+\dots + \delta'_{n+1},
\\
d_{C^2,{\rm Supp}(E_{\frak p})}
\left(
u_{{\bf x}} \circ \widehat{\Phi}_{{\bf x},\vec{\bf q},{\frak p};\Xi_{\frak p}},u_{{\frak p}}
\right)
&<  \delta'_{0}+\dots + \delta'_{n+1}.
\endaligned
\end{equation}
\item We have
\begin{equation}\label{form6565}
\aligned
\vert \partial \,\widehat{\Phi}_{{\bf q}_j,\vec{\bf q},{\frak p};\Xi_{\bf q}}(x)\vert
&>
10\vert \overline\partial \,\widehat{\Phi}_{{\bf q}_j,\vec{\bf q},{\frak p};\Xi_{\frak p}}(x)\vert,
\\
\vert \partial \,\widehat{\Phi}_{{\bf p},\vec{\bf q},{\frak p};\Xi_{\frak p}}(x)\vert
&>
10\vert \overline\partial \,\widehat{\Phi}_{{\bf p},\vec{\bf q},{\frak p};
\Xi_{\bf q}}(x)\vert,
\\
\vert \partial \,\widehat{\Phi}_{{\bf x},\vec{\bf q},{\frak p};
\Xi_{\bf q}}(x)\vert
&>
10\vert \overline\partial \,\widehat{\Phi}_{{\bf x},\vec{\bf q},{\frak p};
\Xi_{\bf q}}(x)\vert
\endaligned
\end{equation}
for each $x \in {\rm Supp}(E_{\frak p})$.
\end{enumerate}
\end{lem}
The proof is the same as the proof of Lemma \ref{lem6767}.

Now using $\widehat{\Phi}_{{\bf x},\vec{\bf q},{\frak p};\Xi_{\frak p}}$
and Lemma \ref{lem6767rev}, we obtain a map

\begin{equation}\label{form66666rev}
\aligned
\mathcal P_{{\bf x},{\bf p},\vec{\bf q},{\frak p};\Xi_{\frak p}} : &
\,\, C^2(\widehat{\Phi}_{{\bf x},\vec{\bf q},{\frak p};\Xi_{\frak p}}
(U({\rm Supp}(E_{\frak p})));u_{\bf x}^*TX \otimes \Lambda^{0,1}) \\
&\to C^2(U({\rm Supp}(E_{\frak p}));u^*_{\frak p}TX \otimes \Lambda^{0,1})
\endaligned
\end{equation}
in the same way as we obtained (\ref{form66666}).
Remark \ref{rem78} also applies here.
We define:
\begin{equation}\label{form6969rev}
E_{{\bf p},\vec{\bf q},{\frak p};\Xi_{\frak p}}({\bf x})
:=
(\mathcal P_{{\bf x},{\bf p},\vec{\bf q},{\frak p};\Xi_{\frak p}})^{-1}
(E_{\frak p}).
\end{equation}
\begin{lem}\label{lem613}
${\bf x} \mapsto E_{{\bf p},\vec{\bf q},{\frak p};\Xi_{\frak p}}({\bf x})$
is smooth in the sense of \cite[Definition 8.7]{diskconst1}.
\end{lem}
\begin{proof}
The proof is the same as \cite[Subsection 11.2]{diskconst1} and so is omitted.
\end{proof}
\begin{rem}\label{rem714}
The subspace $E_{{\bf p},\vec{\bf q},{\frak p};\Xi_{\frak p}}({\bf x})$
depends not only on ${\bf x}$, ${\bf p}$, ${\frak p}$, $\Xi_{\frak p}$
but also on
all of $\vec{\bf q}$, $\vec{\frak S}$.
In fact, we consider an element ${\bf x}$ close to ${\bf p}$ which degenerates to
${\bf q}$ as in Figure 7.
\begin{figure}[h]
\centering
\includegraphics[scale=0.4]{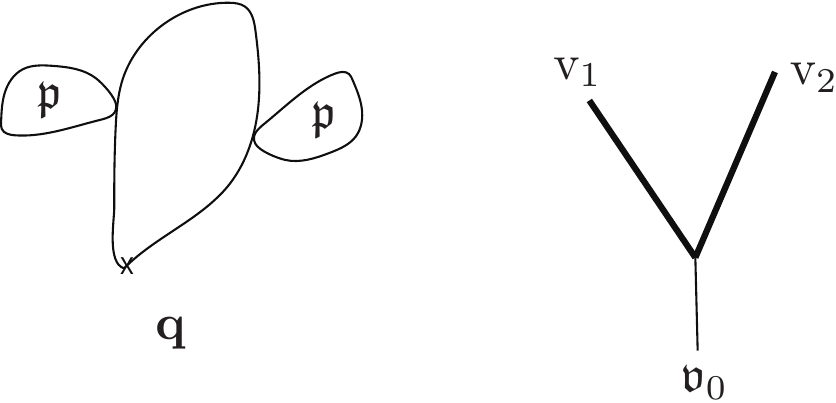}
\caption{Remark \ref{rem714}}
\label{zu5}
\end{figure}
The two (almost) bubbles appearing in the figure are supposed to be
close to ${\bf q}$. We have two choices
(${\bf p},{\bf q},{\rm v}_1$) and $({\bf p},{\bf q},{\rm v}_2)$.
Then for these choices the resulting  $E_{{\bf p},\vec{\bf q},{\frak p};\Xi_{\frak p}}({\bf x})$
are supported in the different part of $\Sigma_{\bf x}$ and so are linearly independent in the
obstruction bundle datum $E_{\bf p}({\bf x})$ to be defined later.
\end{rem}
On the other hand, $E_{{\bf p},\vec{\bf q},{\frak p};\Xi_{\frak p}}({\bf x})$
is independent of small perturbation of ${\bf p}$.
\begin{lem}\label{lem715}
Suppose ${\bf p}$, $\vec{\bf q}$, ${\frak p}$, $\Xi_{\frak p}$
satisfy the assumptions of Lemma \ref{lem61060}.
Then there exists $\frak o$ such that if ${\bf p}'
\in \mathcal M_{k+1,\ell}(X,L,J;\beta)$,
$d({\bf p},{\bf p}') \le \frak o$ then
$$
E_{{\bf p},\vec{\bf q},{\frak p};\Xi_{\frak p}}({\bf x})
=
E_{{\bf p}',\vec{\bf q},{\frak p};\Xi_{\frak p}}({\bf x})
$$
when both  sides are defined.
\end{lem}
\begin{proof}
By assumption, the added marked points $\vec{\frak w}_{{\bf x},{\frak p}}$, which are
defined by two different choices ${\bf p},\vec{\bf q},{\frak p},\Xi_{\frak p}$
and  ${\bf p}',\vec{\bf q},{\frak p},\Xi_{\frak p}$,
are close to each other. Then by applying the Implicit Function Theorem
using Lemma \ref{lem61060} (2), we derive they coincide. The lemma is a consequence of this fact
and the way  the space $E_{{\bf p},\vec{\bf q},{\frak p};\Xi_{\frak p}}({\bf x})$
is defined in \eqref{form6969}.
\end{proof}

\section{Existence of a disk-component-wise system of obstruction bundle data 2}
\label{subsec;exi2}

In this section we will choose an appropriate set of equivalence classes
of the choices of $\vec{\bf q},{\frak p}, \Xi_{\frak p}$ for each
${\bf p}\in \mathcal M_{k+1,\ell}(X,L,J;\beta)$,
and then the obstruction bundle data $E_{\bf p}({\bf x})$ will be a direct
sum of $E_{{\bf p},\vec{\bf q},{\frak p};\Xi_{\frak p}}({\bf x})$ for such choices.
Such an equivalence class will be called a {\it quasi-component}.
Finding a good choice of a set of quasi-components for each ${\bf p}$ is
the main task to carry out.
\begin{defn}
Let $\mathscr{TC}$ be the set of triples $(k,\ell,\beta)$
such that
$k,\ell \in \Z_{\ge 0}$,
$\beta \in H_2(X,L;\Z)$ and $\mathcal M_{k+1,\ell}(X,L,J;\beta)
\ne \emptyset$.
We define a partial order $<$ on $\mathscr{TC}$ as follows.
\par
Let $(k_i,\ell_i,\beta_i) \in \mathscr{TC}$ ($i=1,2$). We say
$(k_1,\ell_1,\beta_1) \,>' \, (k_2,\ell_2,\beta_2)$ if there 
exists ${\bf p} \in \mathcal M_{k_1+1,\ell_1}(X,L,J;\beta_1)(\frak T)$ and 
${\bf q} \in \mathcal M_{k_2+1,\ell_2}(X,L,J;\beta_2)$ such that 
${\bf q} = {\bf p}_{\frak S}$ for a certain subtree $\frak S$ of $\frak T$.
\par
We say $(k,\ell,\beta) > (k',\ell',\beta')$ if there exist
$(k_j,\ell_j,\beta_j)$ such that $(k,\ell,\beta) =(k_1,\ell_1,\beta_1)$, 
$(k',\ell',\beta') =(k_n,\ell_n,\beta_n)$ and 
$(k_j,\ell_j,\beta_j) \,>' \, (k_{j+1},\ell_{j+1},\beta_{j+1})$.
\end{defn}
We note that the following is an immediate consequence of Gromov compactness.
\begin{lem}\label{lem72}
For any $(k,\ell,\beta) \in \mathscr{TC}$,
there exists only a finite number of $(k',\ell',\beta') \in \mathscr{TC}$ such that
$(k',\ell',\beta') < (k,\ell,\beta)$.
\end{lem}

We will associate various objects to $\mathcal M_{k+1,\ell}(X,L,J;\beta)$ inductively on this 
partial order $<$.
The objects we will construct are as follows.
\begin{enumerate}
\item[(ob1)]
A finite subset $\frak P(k,\ell,\beta) \subset \mathcal M^{\circ}_{k+1,\ell}(X,L,J;\beta)$.
\item[(ob2)]
For each $\frak p \in \frak P(k,\ell,\beta)$, we take its open neighborhood $K_{\circ}(\frak p)$
in $\mathcal M^{\circ}_{k+1,\ell}(X,L,J;\beta)$ so that its closure $K(\frak p)$ in $\mathcal M^{\circ}_{k+1,\ell}(X,L,J;\beta)$
is a compact subset contained in $\mathcal M^{\circ}_{k+1,\ell}(X,L,J;\beta)$.
\item[(ob3)] We take type I strong stabilization data together with an obstruction space
$E_{\frak p}$ at $\frak p$, which we denote by $\Xi_{\frak p}$.
\end{enumerate}
We will construct other objects $\mathscr F_{k+1,\ell,\beta}$, $\mathscr F^{\circ}_{k+1,\ell,\beta}$ 
in addition which will be described later.
(See (ob4) right above Condition \ref{cond717}.)
\par
We start with describing the conditions we require for them.
(Existence of the objects satisfying those conditions will be
proved in Proposition \ref{prop722} later.)
Let
$$\frak p \in \frak P(k,\ell,\beta),
\qquad {\bf q} \in K(\frak p).
$$
We require $K(\frak p)$ to be sufficiently small so that Condition \ref{cond73}
below holds.
Let $\vec{\frak w}_{\frak p}$ be the additional interior marked points
which are parts of $\Xi_{\frak p}$.
If $K(\frak p)$ is sufficiently small, there exists
$\vec{w}_{{\bf q},\frak p}$ such that
${\bf q} \cup \vec{w}_{{\bf q},\frak p}$ is $\epsilon$-close to
${\frak p} \cup \vec{\frak w}_{\frak p}$.
In particular,
\begin{equation}\label{form71}
(\Sigma_{\bf q},\vec{w}_{{\bf q},\frak p})
\in
\Phi
\left(
\mathcal V_{{\frak p}}^{\rm d} \times \prod_{\frak v}
\mathcal V_{{\frak p},{\frak v}}^{\rm s}
\times (D^2(c))^{m_{\rm s}}
\right)
\end{equation}
where $\Phi$ is the map (\ref{fprmmapF}) induced by $\Xi_{\frak p}$.
Therefore in the same way as in Section \ref{subsec;exi1}, we obtain
an embedding
\begin{equation}\label{form72}
\hat{\Phi}_{{\bf q},{\frak p};\Xi_{\frak p}} : \rm{Supp}(E_{\frak p}) \to \Sigma_{\bf q}
\end{equation}
and then a map
\begin{equation}\label{form73}
\aligned
\mathcal P_{{\bf q},{\frak p};\Xi_{\frak p}} :\,\,\,
&C^2(\hat{\Phi}_{{\bf q},{\frak p};\Xi_{\frak p}}(U({\rm Supp}(E_{\frak p}))),u_{\bf q}^*TX
\otimes \Lambda^{0,1}) \\
&\to
C^2(U({\rm Supp}(E_{\frak p})),u_{\frak p}^*TX
\otimes \Lambda^{0,1})
\endaligned
\end{equation}
where $U({\rm Supp}(E_{\frak p}))$ is a sufficiently small open neighborhood of
${\rm Supp}(E_{\frak p})$.
Then we define
\begin{equation}\label{form6969revrev}
E_{{\bf q},{\frak p};\Xi_{\frak p}}({\bf q})
=
(\mathcal P_{{\bf q},{\frak p};\Xi_{\frak p}})^{-1}(E_{\frak p}).
\end{equation}

\begin{conds}\label{cond73}
We require that $K(\frak p)$ is so small that
(\ref{form71}) holds, (\ref{form72}),(\ref{form73}) are well defined, and
Definition \ref{defn631} (2)(3) hold with $E_{\frak p}$ replaced by
$E_{{\bf q},{\frak p};\Xi_{\frak p}}({\bf q})$.
\end{conds}
We remark that if ${\frak p} = {\bf q}$ then Definition \ref{defn631} (2)(3) hold
by assumption. Therefore Condition \ref{cond73} holds if $K(\frak p)$
is sufficiently small.
\par
We denote by $\frak T_0  \in \mathcal G(k+1,\ell,\beta)$ the unique element
that has only one interior vertex. We call it the {\it trivial element}.
\par
Now we will discuss the relationship between the data given above and the construction
of Subsection \ref{subsec:itera2}.
Suppose we are given a finite set $\frak P(k,\ell,\beta)$ as in (ob1)
for each $(k,\ell,\beta)$ and $\Xi_{\frak p}$ as in (ob3) for $\frak p \in \frak P(k,\ell,\beta)$.

\begin{shitu}\label{situ69rev}
We consider
${\bf p}$, ${\bf q}_j$, ${\frak p}$,
$\frak T_j$, $\frak S_j$ ($j=1,\dots,n$ for some $n\ge 1$)
as in Situation \ref{situ69},
except we require
\begin{enumerate}
\item[(3)']
$d({\bf p},{\bf q}_1) < \epsilon(k,\ell,\beta) = \epsilon(k_1,\ell_1,\beta_1)$.
\item[(5)']
For $j \le n-1$
we require ${\bf q}_{j,\frak S_j} \in \mathcal M_{k_{j+1}+1,\ell_{j+1}}(X,L,J;\beta_{j+1})$
and
\begin{equation}
d({\bf q}_{j,\frak S_j}, {\bf q}_{j+1}) < \epsilon(k_{j+1},\ell_{j+1},\beta_{j+1}).
\end{equation}
\item[(6)']
${{\frak p}} \in \frak P(k_{n+1},\ell_{n+1},\beta_{n+1})$, 
${\bf q}_{n,{\frak S}_n} \in \mathcal M_{k_{n+1}+1,\ell_{n+1}}(X,L,J;\beta_{n+1})$,
and
$$
d({\bf q}_{n,\frak S_n}, \frak p) < \epsilon(\frak p)
$$
\end{enumerate}
instead of (3), (5), (6).
\par
Here $\epsilon(k,\ell,\beta)$ is a sufficiently small positive number
depending only on $k,\ell,\beta$,
and the set $\frak P(k',\ell',\beta')$ 
with $(k,\ell,\beta) > (k',\ell',\beta')$, 
and $\epsilon(\frak p)$ is a positive number depending only on $\frak p$, $\Xi_{\frak p}$.
They are to be determined later during 
the proof of Lemma \ref{lem710}.
\par
In case $n=0$ we assume 
\begin{equation}\label{nform86}
d({\bf p},\frak p) < \epsilon(\frak p).
\end{equation}
 We require that (\ref{nform86}) implies ${\bf p} \in K_0(\frak p)$.
$\diamond$

\end{shitu}
We consider ${\bf x} \in  B_{\epsilon_0}(\mathcal X,{\bf p})$.
(Here $\epsilon_0$ depends on ${\bf p}$ etc. and will be determined later.)
Note that the conditions on the distance appearing in (3)',(5)',(6)' are
similar to but slightly different from (3),(5),(6) in Situation \ref{situ69}.
The next lemma claims that we can use (3)',(5)',(6)' in place of (3),(5),(6).
\begin{lem}\label{lem710}
Suppose $n\ge 1$ in Situation \ref{situ69rev}.
We may choose $\epsilon(k,\ell,\beta)$ so that
if (1),(2),(3)',(4),(5)',(6)',(7) above hold then
the conclusions of Lemmas \ref{lem611},\ref{lem6767rev} hold with the 
right hand sides of (\ref{newform22}),(\ref{form6464}) replaced by $\delta(\frak p)$ in 
(\ref{newnew723}).
\par
In case $n=0$ the same conclusion holds under the assumption 
(\ref{nform86}).
\end{lem}
Namely,
Lemma \ref{lem710} claims  uniformity of the
constants $\epsilon_{1,j}(\delta;{\bf q}_{j},\dots,{\bf q}_n,{\frak p},\Xi_{\frak p})$ and $\epsilon_{2}(\delta;{\frak p},\Xi_{\frak p})$.
In (\ref{form611}) they depend on
${\bf p},{\bf q}_{1},\dots,{\bf q}_n,\frak p$.
However, Lemma \ref{lem710} asserts that we can choose them so that they depend only on $k_j$, $\ell_j$, $\beta_j$
and that
the conclusions of Lemmas \ref{lem61060}, 
\ref{lem611} and \ref{lem6767rev} hold.
We prove this technical Lemma \ref{lem710} at the end of this section
using compactness of various spaces.
\par
Now we apply Lemma \ref{lem61060} to obtain
$\vec{\frak w}_{{\bf x};{\frak p}}$, $\vec{\frak w}_{{\bf p};{\frak p}}$
$\vec{\frak w}_{{\bf q}_j,{\frak p}}$ ($j=1,\dots,n$).
Then we use Lemma \ref{lem611} to obtain smooth embeddings
\begin{equation}\label{form6363revsuper}
\widehat{\Phi}_{{\bf x},{\bf p},\vec{\bf q},{\frak p};\Xi_{\frak p}} :
{\rm Supp}(E_{\frak p}) \to \Sigma_{{\bf x}}
\end{equation}
as in \eqref{form6363}.
Note $\vec{\frak w}_{{\bf p};{\frak p}}$ also depends on $\vec{\bf q}$
so we write $\vec{\frak w}_{{\bf p},\vec{\bf q};{\frak p}}$.
Then Lemma \ref{lem710} enables us to define the following notion.
\begin{defn}
We call $({\bf p},\vec{\bf q},\frak p;\vec{\frak T},\vec{\frak S})$ as in Situation \ref{situ69rev}
a {\it quasi-splitting sequence}.
Suppose $({\bf p},\vec{\bf q}_{(c)},\frak p;\vec{\frak T}_{(c)},\vec{\frak S}_{(c)})$
for $c=1,2$ are quasi-splitting sequences (with the same ${\bf p},\frak p$).
We say that they are {\it equivalent} if
$$
\vec{\frak w}_{{\bf p},\vec{\bf q}_{(1)},{\frak p}}
=
\vec{\frak w}_{{\bf p},\vec{\bf q}_{(2)},{\frak p}}.
$$
An equivalence class of quasi-splitting sequences is called a {\it quasi-component}.
\par
Let ${\mathscr{QC}}_{k+1,\ell}(X,L,J;\beta)$ be the set of all quasi-components. There is a map
\begin{equation}\label{map87}
\Pi=\Pi_{\beta} : {\mathscr{QC}}_{k+1,\ell}(X,L,J;\beta) \to \mathcal M_{k+1,\ell}(X,L,J;\beta)
\end{equation}
which assigns ${\bf p}$ to an equivalence class of $({\bf p},\vec{\bf q},\frak p;\vec{\frak T},\vec{\frak S})$.
We say an element $\xi$ of ${\mathscr{QC}}_{k+1,\ell}(X,L,J;\beta)$ is a
{\it quasi-component of ${\bf p}$}
if $\Pi(\xi) = {\bf p}$.
We write an element of
${\mathscr{QC}}_{k+1,\ell}(X,L,J;\beta)$
as
$$
({\bf p},\xi,\frak p)
$$
where ${\bf p}$ and ${\frak p}$ are
the first and the last element of the sequence.
We put
$$
\vec{\frak w}_{{\bf p},\xi,\frak p}
= \vec{\frak w}_{{\bf p},\vec{\bf q},{\frak p}}
$$
if $\xi$ is the equivalence class of $({\bf p},\vec{\bf q},{\frak p};\vec{\frak T},\vec{\frak S})$.
\end{defn}
We note that
$$
\widehat{\Phi}_{{\bf x},{\bf p},\vec{\bf q}_{(1)},{\frak p};\Xi_{\frak p}}
=
\widehat{\Phi}_{{\bf x},{\bf p},\vec{\bf q}_{(2)},{\frak p};\Xi_{\frak p}},
$$
if $({\bf p},\vec{\bf q}_{(1)},{\frak p};\vec{\frak T}_{(1)},\vec{\frak S}_{(1)})$ is
equivalent to
$({\bf p},\vec{\bf q}_{(2)},{\frak p};\vec{\frak T}_{(2)},\vec{\frak S}_{(2)})$.
Therefore for each quasi-component $({\bf p},\xi,\frak p)$ we can associate an embedding
\begin{equation}\label{form88}
\widehat{\Phi}_{({\bf x},{\bf p},\xi,\frak p)} : {\rm Supp}(E_{\frak p}) \to \Sigma_{{\bf x}}.
\end{equation}
It then induces a map
\begin{equation}\label{form66666revrev}
\aligned
\mathcal P_{({\bf x},{\bf p},\xi,\frak p)} : &C^2(\widehat{\Phi}_{({\bf x},{\bf p},\xi,\frak p)}(U({\rm Supp}(E_{\frak p})));u_{\bf x}^*TX \otimes \Lambda^{0,1}) \\
&\to C^2(U({\rm Supp}(E_{\frak p}));u^*_{\frak p}TX \otimes \Lambda^{0,1})
\endaligned
\end{equation}
in the same way as (\ref{form66666rev}). We define
\begin{equation}\label{form6969revrev}
E_{{\bf p},\xi,{\frak p}}({\bf x})
=
(\mathcal P_{({\bf x},{\bf p},\xi,\frak p)})^{-1}(E_{\frak p}).
\end{equation}
Compare this with (\ref{form6969rev}). Here $E_{\frak p}$ is a part of the data $\Xi_{\frak p}$.
\par
The obstruction bundle data $\{E_{{\bf p}}({\bf x})\}$ we will construct are the direct sum
of $E_{{\bf p},\xi,{\frak p}}({\bf x})$ for an appropriate set of
quasi-components of ${\bf p}$. We need
a careful choice of the set of the quasi-components so that
it satisfies the required properties.
The discussion of the process of choosing such a set of quasi-components
will follow.
We first observe:
\begin{lem}\label{lem87}
For each ${\bf p}  \in \mathcal M_{k+1,\ell}(X,L,J;\beta)$, the set $\Pi^{-1}({\bf p})$ of quasi-components of ${\bf p}$
is a finite set.
\end{lem}
\begin{proof}
It follows from Lemma \ref{lem65} and the transversality imposed on $\mathcal N_{\frak p,i}$ in Definition \ref{defn62}
that each $\frak w_{\frak p,i}$ carries its sufficiently small connected neighborhood $U_i \subset \Sigma_{\frak p}$
such that $U_i \cap \vec{\frak w}_{\frak p}$ is a single point, i.e.,
$$
U_i \cap \vec{\frak w}_{\frak p} = \{\frak w_{\frak p,i}\}
$$
and  $u_{\bf p} \circ \widehat \Phi_{({\bf p},\xi,\frak p)}: U_i \to X$ and
$u_{\frak p}: U_i \to X$ are $C^2$-close embeddings for \emph{all} possible such choices $({\bf p},\xi,\frak p)$.
Here and hereafter we denote $\widehat \Phi_{({\bf p},\xi,\frak p)} = \widehat \Phi_{({\bf p},{\bf p},\xi,\frak p)}$
and the right hand side is the case ${\bf p} = {\bf x}$ of (\ref{form88}).
Recall that, for given ${\bf p}$,
the number of possible objects $\frak p$
which appear in the quasi-component $({\bf p},\xi,\frak p)$
of ${\bf p}$ is finite by Lemma
\ref{lem72} and finiteness of $\frak P(k,\ell,\beta)$.

For any $({\bf p},\xi,\frak p)$, we put $U'_{i,\xi} = \widehat \Phi_{({\bf p},\xi,\frak p)}(U_i)$.
Then by the above mentioned $C^2$-closeness we have
\begin{equation}\label{areabound}
\int_{U'_{i,\xi}}u_{\bf p}^* \omega_X
> \frac{1}{2} \int_{U_{i}}u_{\frak p}^* \omega_X > c
\end{equation}
for some positive number $c$ independent of $i,\xi$.
\par
Now consider two different quasi-components $({\bf p},\xi,\frak p)$ and $({\bf p},\xi',\frak p)$. We put
$U'_{i,\xi'} = \widehat \Phi_{({\bf p},\xi',\frak p)}(U_i)$ similarly as
$U'_{i,\xi}$.
Then there exists $i$ such that
\begin{equation}\label{eq:2frakwi}
{\frak w}_{{\bf p},\xi,\frak p,i} \neq {\frak w}_{{\bf p},\xi',\frak p,i}, \quad
{\frak w}_{{\bf p},\xi,\frak p,i} \in U'_{i,\xi}, \,  {\frak w}_{{\bf p},\xi',\frak p,i} \in U'_{i,\xi'}
\end{equation}
by the definition of the equivalence class $\xi$.
It follows from the $C^2$-closeness of $u_{\bf p} \circ \widehat \Phi_{({\bf p},\xi,\frak p)}$ and $u_{\bf p} \circ \widehat \Phi_{({\bf p},\xi',\frak p)}$, the transversality
imposed on $\mathcal N_{\frak p,i}$ and the embedding properties of $\widehat \Phi_{({\bf p},\xi,\frak p)}$,
$\widehat \Phi_{({\bf p},\xi,\frak p)}$ that we have a covering map 
$$
(\widehat \Phi_{({\bf p},\xi,\frak p)})^{-1}  \cup (\widehat \Phi_{({\bf p},\xi',\frak p)})^{-1} : 
U'_{i,\xi} \cup U'_{i,\xi'} \to U_i
$$
which is a homeomorphism on each of $U'_{i,\xi}$ and on $U'_{i,\xi'}$ respectively.
Therefore \eqref{eq:2frakwi} implies $U'_{i,\xi} \cap U'_{i,\xi'} = \emptyset$.
This clearly implies that the number of such $({\bf p},\xi,\frak p)$ must be finite. In fact
otherwise we would have
$$
\infty> \int_{\Sigma_{\bf p}} u_{\bf p}^* \omega_X
\geq \sum_{({\bf p},\xi,\frak p)} \int_{U'_{i,\xi}}u_{\bf p}^* \omega_X =
\infty,
$$
a contradiction. Here we use (\ref{areabound}).
This finishes the proof.
\end{proof}
Let $\mathscr F$ be a subset of ${\mathscr{QC}}_{k+1,\ell}(X,L,J;\beta)$.
For a given point ${\bf p} \in \mathcal M_{k+1,\ell}(X,L,J;\beta)$ we put
$$
\mathscr F({\bf p}) = \Pi^{-1}({\bf p}) \cap \mathscr F,
$$
which is a finite set. In this way we regard the assignment
$\mathscr F: {\bf p} \mapsto \mathscr F({\bf p})$
as a map which assigns
a finite set of quasi-components of ${\bf p}$ to an element ${\bf p}$ of ${\mathcal M}_{k+1,\ell}(X,L,J;\beta)$.
We call $\mathscr F$ a {\it quasi-component choice map}.
\par
We next define a topology on ${\mathscr{QC}}_{k+1,\ell}(X,L,J;\beta)$
using the next lemma.
\par
Let ${\bf p} \in \mathcal M_{k+1,\ell}(X,L,J;\beta)$.
We fix a stabilization and trivialization data at ${\bf p}$
in the sense of \cite[Definition 4.9]{diskconst1}.
Then if ${\bf p}' \in \mathcal M_{k+1,\ell}(X,L,J;\beta)$ is close to ${\bf p}$,
we can define a map
$\hat{\Phi}_{{\bf p}'{\bf p}} : \Sigma_{\bf p}(\vec{\epsilon}) \to \Sigma_{{\bf p}'}$.
(See \cite[Lemma 3.9]{diskconst1}.)
\begin{lem}\label{lem712}
There are neighborhoods $U({\bf p})$ of ${\bf p}$
in $\mathcal M_{k+1,\ell}(X,L,J;\beta)$
and $U({\frak w}_{{\bf p},\xi,{\frak p},i})$ of ${\frak w}_{{\bf p},\xi,{\frak p},i}$
in $\Sigma_{\bf p}$
such that if ${\bf p}' \in U({\bf p})$ and $({\bf p},\xi,\frak p)$ is an
quasi-component of ${\bf p}$, and there exists  a unique
quasi-component $({\bf p}',\xi',\frak p)$ of ${\bf p}'$ with the same
${\frak p}$ such that
${\frak w}_{{\bf p}',\xi',{\frak p},i}\in
\hat{\Phi}_{{\bf p}'{\bf p}}(U({\frak w}_{{\bf p},\xi,{\frak p},i}))$.
\end{lem}
\begin{proof}
Let $({\bf p},\vec{\bf q},\frak p)$ be a quasi-splitting sequence
representing $({\bf p},\xi,\frak p)$.
It follows from definition that
$({\bf p}',\vec{\bf q},\frak p)$ is also a quasi-splitting sequence
if ${\bf p}'$ is sufficiently close to ${\bf p}$.
We thus obtain $({\bf p}',\xi',\frak p)$.
\par
The quasi-component $\xi'$ is independent of the choice of a representative of $\xi$
because the point
${\frak w}_{{\bf p}',\xi',{\frak p},i}$ which is close to
${\frak w}_{{\bf p},\xi,{\frak p},i}$ and
$u_{{\bf p}'}({\frak w}_{{\bf p}',\xi',{\frak p},i})
\in \mathcal N_{{\frak p},i}$ is unique. The
uniqueness in the statement of Lemma \ref{lem712} follows
from the same fact.
\end{proof}
\begin{defn}\label{defn89}
We define a topology on ${\mathscr{QC}}_{k+1,\ell}(X,L,J;\beta)$ as follows.
\par
Let $({\bf p},\xi,\frak p) \in {\mathscr{QC}}_{k+1,\ell}(X,L,J;\beta)$.
By Lemma \ref{lem712} we obtain
a neighborhood $U({\bf p})$ and an injective map
$U({\bf p}) \to {\mathscr{QC}}_{k+1,\ell}(X,L,J;\beta)$
which sends ${\bf p}'$ to $({\bf p}',\xi,\frak p)$.
We define a neighborhood system of  $({\bf p},\xi,\frak p)$ by sending one of ${\bf p}$ by this map.
\end{defn}
\begin{lem} This topology is Hausdorff.
The map $\Pi : {\mathscr{QC}}_{k+1,\ell}(X,L,J;\beta)
\to \mathcal M_{k+1,\ell}(X,L,J;\beta)$ is a local homeomorphism.
Namely for each point ${\frak x} \in {\mathscr{QC}}_{k+1,\ell}(X,L,J;\beta)$ there
exist neighborhoods of $\frak x$ and of $\Pi(\frak x)$ so that $\Pi$ induces
a homeomorphism between them.
\end{lem}
The proof is easy and is omitted.
\begin{defn}
A quasi-component choice map $\mathscr F$ is said to be {\it open} (resp. {\it closed})
if it is open (resp. closed) as a subset of ${\mathscr{QC}}_{k+1,\ell}(X,L,J;\beta)$.
\par
The {\it closure} of a quasi-component choice map is defined to be the closure as
a subset of ${\mathscr{QC}}_{k+1,\ell}(X,L,J;\beta)$.
\par
$\mathscr F$ is said to be {\it proper} if the restriction of $\Pi$ to $\mathscr F$
is a proper map (to $\mathcal M_{k+1,\ell}(X,L,J;\beta)$).
\end{defn}
Other objects mentioned right after Lemma \ref{lem72} we will construct are:
\begin{enumerate}
\item[(ob4)]
Quasi-component choice maps $\mathscr F^{\circ}_{k+1,\ell,\beta}$ and
$\mathscr F_{k+1,\ell,\beta}$ for each
$(k,\ell,\beta) \in \mathscr{TC}$.
\end{enumerate}
We describe the conditions we require for $\mathscr F^{\circ}_{k+1,\ell,\beta}$ and
$\mathscr F_{k+1,\ell,\beta}$ below.
\begin{conds}\label{cond717}
$\mathscr F^{\circ}_{k+1,\ell,\beta}$ is open.
$\mathscr F_{k+1,\ell,\beta}$ contains its closure and is proper.
$\mathscr F^{\circ}_{k+1,\ell,\beta}$ and $\mathscr F_{k+1,\ell,\beta}$ are invariant under the action of
extended automorphism group in an obvious sense.
\end{conds}
We require three more conditions.
The first one is
Condition \ref{cond718} which describes $\mathscr F^{\circ}_{k+1,\ell,\beta}$, $\mathscr F_{k+1,\ell,\beta}$ at the boundary points
of $\mathcal M_{k+1,\ell}(X,L,J;\beta)$.
Let ${\bf p} \in \mathcal M_{k+1,\ell}^{\circ}(X,L,J;\beta)(\frak T)$
where $\frak T=({\mathcal T}, \beta(\cdot), l(\cdot))$ is nontrivial, i.e.,
$\mathcal T$ has at least two interior vertices.
(See Definition \ref{defn274} and \eqref{stratumemb2} for this notation.)
For an interior vertex ${\rm v}$ of $\mathcal T$,
we obtain ${\bf p}_{\rm v} \in  
\mathcal M_{k_{\rm v}+1,\ell_{\rm v}}^{\circ}(X,L,J;\beta({\rm v}))$.
Let us define a map
\begin{equation}
\mathscr I_{{\bf p},{\rm v}} : \Pi_{\beta({\rm v})}^{-1}({\bf p}_{\rm v}) \to 
\Pi^{-1}_{\beta}({\bf p}).
\end{equation}
Note that $\Pi_{\beta}$, $\Pi_{\beta({\rm v})}$  are the maps (\ref{map87}) and $\Pi_{\beta}^{-1}({\bf p}_{\rm v})$ is the
set of all quasi-components of ${\bf p}_{\rm v}$.
Consider a quasi-component
$({\bf p}_{\rm v},\xi,\frak p;\vec{\frak T},\vec{\frak S})$ of ${\bf p}_{\rm v}$.
We put $\frak S'_1 = \{{\rm v}\}$
and ${\bf q}'_1 = {\bf p}_{\rm v}$.
If $\xi$ is represented by a sequence ${\bf q}_i$, $\frak S_i$,
we put ${\bf q}'_{i+1} = {\bf q}_{i}$ and shift the index of $
\vec{\frak S}$ by $1$
to obtain $\vec{\frak S}'$. $\vec{\frak T}'$ is obtained from  ${\bf q}'_i$
automatically.
We thus obtain a quasi-component $({\bf p},\xi',\frak p;\vec{\frak T}',\vec{\frak S}')$ of ${\bf p}$.
We define
$$
\mathscr I_{{\bf p},{\rm v}}({\bf p}_{\rm v},\xi,\frak p;\vec{\frak T},\vec{\frak S})
:=
({\bf p},\xi',\frak p;\vec{\frak T}',\vec{\frak S}').
$$

\begin{conds}\label{cond718}
({\bf Boundary stratifications})
For ${\bf p} \in \mathcal M_{k+1,\ell}^{\circ}(X,L,J;\beta)(\frak T)$
with nontrivial $\frak T=({\mathcal T}, \beta(\cdot), l(\cdot))$,
$\mathscr F^{\circ}_{k+1,\ell,\beta}({\bf p})$
is the set of all equivalence classes of quasi-components $({\bf p},\xi',\frak p)$
as above,
where we take all possible choices of ${\rm v}$ and
$({\bf p}_{\rm v},\xi,\frak p) \in
\mathscr F^{\circ}_{k_{\rm v},\ell_{\rm v},\beta({\rm v})}({\bf p}_{\rm v})$.
In other words, we have
\begin{equation}\label{form812}
\mathscr F^{\circ}_{k+1,\ell,\beta}({\bf p})
=
\bigcup_{{\rm v} \in C_0^{\rm int}(\mathcal T)}
\mathscr I_{{\bf p},{\rm v}}(\mathscr F^{\circ}_{k_{\rm v}+1,\ell_{\rm v},\beta_{\rm v}}({\bf p}_{\rm v})).
\end{equation}
The same holds if we replace $\mathscr F^{\circ}$ by $\mathscr F$.
Namely, we have
\begin{equation}\label{form81322}
\mathscr F_{k+1,\ell,\beta}({\bf p})
=
\bigcup_{{\rm v} \in C_0^{\rm int}(\mathcal T)}
\mathscr I_{{\bf p},{\rm v}}(\mathscr F_{k_{\rm v}+1,\ell_{\rm v},\beta_{\rm v}}({\bf p}_{\rm v})).
\end{equation}
\end{conds}
It is easy to see that both of the right hand sides of (\ref{form812}) and (\ref{form81322}) are disjoint 
unions.
\par
\medskip
The remaining two conditions we require are related to transversality.
For ${\bf p} \in \mathcal M_{k+1,\ell}(X,L,J;\beta)$ we consider the sum:
\begin{equation}\label{form712}
\sum_{({\bf p},\xi,\frak p) \in \mathscr F_{k+1,\ell,\beta}({\bf p})}E_{{\bf p},\xi,{\frak p}}({\bf p})
\subset C^{\infty}(\Sigma_{\bf p};u_{\bf p}^*TX \otimes \Lambda^{0,1}).
\end{equation}
Here we note that we have
$$
\widehat \Phi_{({\bf p},{\bf p},\xi,{\frak p})}(U(\text{\rm Supp}(E_{{\frak p}}))) \subset \Sigma_{\bf p}
$$
in \eqref{form88} with ${\bf x} ={\bf p}$.

\begin{conds}\label{cond719}
({\bf Direct sum})
The sum (\ref{form712}) is a direct sum which we denote by:
\begin{equation}\label{form712rev}
E_{{\bf p};\mathscr F}({\bf p}) =
\bigoplus_{({\bf p},\xi,\frak p) \in \mathscr F_{k+1,\ell,\beta}({\bf p})}E_{{\bf p},\xi,{\frak p}}({\bf p})
\subset C^{\infty}(\Sigma_{\bf p};u_{\bf p}TX \otimes \Lambda^{0,1}).
\end{equation}
\end{conds}
\par\smallskip
For a given point ${\bf p} \in \mathcal M_{k+1,\ell}(X,L,J;\beta)$,
we may take $\epsilon_0({\bf p})>0$ so small  that if
${\bf x} \in B_{\epsilon_0({\bf p})}(\mathcal  X,{\bf x})$, the subspace $E_{{\bf p},\xi,{\frak p}}({\bf x})$ as in
(\ref{form6969revrev}) is defined and we have the sums
$$\aligned
E_{{\bf p};\mathscr F}({\bf x}) &=
\bigoplus_{({\bf p},\xi,\frak p) \in \mathscr F_{k+1,\ell,\beta}({\bf p})}E_{{\bf p},\xi,{\frak p}}({\bf x})
\subset C^{2}(\Sigma_{\bf x};u_{\bf x}TX \otimes \Lambda^{0,1}),
\\
E_{{\bf p};\mathscr F}^{\circ}({\bf x}) &=
\bigoplus_{({\bf p},\xi,\frak p) \in \mathscr F^{\circ}_{k+1,\ell,\beta}({\bf p})}E_{{\bf p},\xi,{\frak p}}({\bf x})
\subset C^{2}(\Sigma_{\bf x};u_{\bf x}TX \otimes \Lambda^{0,1}).
\endaligned
$$
Note that the sums in the right hand sides are direct sums by Condition \ref{cond719}.
\begin{conds}\label{cond720}
({\bf Transversality})
We consider the operator
$$
D_{u_{\bf p}}\overline{\partial} :
W^2_{m+1}((\Sigma_{\bf p},\partial\Sigma_{\bf p});u_{\bf p}^*TX,u_{\bf p}^*TL)
\to L^2_m(\Sigma_{\bf p},u_{\bf p}^*TX\otimes \Lambda^{0,1})
$$
as in \cite[(5.1)]{diskconst1}.  Then we require
$$
{\rm Im} D_{u_{\bf p}}\overline{\partial} + E^{\circ}_{{\bf p};\mathscr F}({\bf p})
= L^2_m(\Sigma_{\bf },u_{\bf p}^*TX\otimes \Lambda^{0,1}).
$$
Moreover for the evaluation map
$$
{\mathcal{EV}} : W^2_{m+1}((\Sigma_{\bf p},\partial\Sigma_{\bf p});u_{\bf p}^*TX,u_{\bf p}^*TL)
\to T_{u_{\bf p}(z_0)}L
$$
at $z_0$, the restriction
$$
{\mathcal{EV}} : (D_{u_{\bf p}}\overline{\partial})^{-1}(E^{\circ}_{{\bf p};\mathscr F}({\bf p})) \to T_{u_{\bf p}(z_0)}L
$$
is surjective, and
the action of ${\rm Aut}({\bf p})$ on $(D_{u_{\bf p}}\overline{\partial})^{-1}(E^{\circ}_{{\bf p};\mathscr F}({\bf p}))/ {\frak {aut}} (\Sigma_{\bf p}, \vec z_{\bf p},
\vec {\frak z}_{\bf p})$ is effective.
\end{conds}
\begin{rem}
Note that in case $D_{u_{\bf p}}\overline{\partial}$ is 
not surjective, Condition \ref{cond720} 
implies that $\mathscr F^{\circ}_{k+1,\ell,\beta}({\bf p})$ 
is non-empty.
We may take $E^{\circ}_{{\mathbf p}, {\mathcal F}} = 0$, 
when $D_{u_{\mathbf p}}\overline{\partial}$ is surjective and the restriction of $\mathcal{EV}$ 
to $\rm{Ker} \, D_{u_{\mathbf p}}\overline{\partial}$ is surjective.  
\end{rem}
Now we have the following two results.
\begin{prop}\label{prop721}
Suppose we have the objects as in (ob1)-(ob4) and the constants $\epsilon(k,\ell,\beta) >0$,
$\epsilon(\frak p) > 0$
such that Conditions \ref{cond73}, \ref{cond717}, \ref{cond718}, \ref{cond719}, \ref{cond720} are satisfied and 
that $\epsilon(k,\ell,\beta)$, $\epsilon(\frak p) > 0$ are as in Lemma \ref{lem710}.
Then we can choose $\epsilon_0({\bf p})>0$ for each ${\bf p}$ such that
$\mathscr U_{\bf p} =  B_{\epsilon_0({\bf p})}(\mathcal X,{\bf p})$ and
$\{E_{{\bf p};\mathscr F}({\bf x})\}$ is
a system of obstruction bundle data that is disk-component-wise.
\end{prop}

\begin{prop}\label{prop722}
There exist objects as in (ob1)-(ob4) and constants $\epsilon(k,\ell,\beta)>0$,
$\epsilon(\frak p) > 0$
such that Conditions \ref{cond73}, \ref{cond717}, \ref{cond718}, \ref{cond719}, \ref{cond720}  are satisfied
and  $\epsilon(k,\ell,\beta)$, $\epsilon(\frak p) > 0$ are as in  Lemma \ref{lem710}.
\end{prop}
It is obvious that Propositions \ref{prop721} and \ref{prop722} imply
Theorem \ref{thm43}.
Thus to prove Theorem \ref{thm43} it remains to prove Propositions \ref{prop721} and \ref{prop722}
and Lemma \ref{lem710}.

\begin{proof}[Proof of Proposition \ref{prop721}]
We first check that $\{E_{{\bf p};\mathscr F}({\bf x})\}$ is an obstruction bundle data.
Definition \ref{defn51} (1) is obvious from construction.
Definition \ref{defn51} (2) (smoothness) is a consequence of Lemma \ref{lem613}.
\par
Definition \ref{defn51} (3) (transversality) is a consequence of
Condition \ref{cond720} for ${\bf x} = {\bf p}$.
By taking $\epsilon_0({\bf p})>0$ small enough, we can prove the same property for
${\bf x} \in B_{\epsilon_0({\bf p})}(\mathcal X,{\bf p})$.
\par
Definition \ref{defn51} (4) (semi-continuity) is a consequence of
Lemma \ref{lem715} and the
properness of $\mathscr F_{k+1,\ell,\beta}$, which is a part of
Condition \ref{cond717}.
Definition \ref{defn51} (5) (invariance under the extended automorphisms)
is a consequence of the invariance of $\mathscr F_{k+1,\ell,\beta}$,
which is a part of Condition \ref{cond717}.
\par
We then observe that disk-component-wise-ness (\ref{form410})
is an immediate consequence of Condition \ref{cond718} and the definition.
\end{proof}

\begin{proof}[Proof of Proposition \ref{prop722}]
The proof is by induction on $(k,\ell,\beta)$ with respect to the
partial order $<$.
We first consider the case when $(k,\ell,\beta)$ is minimal.
In this case $\mathcal G(k+1,\ell,\beta)$ consists of one element,
the trivial element $\frak T_0$.
We can construct $\frak P(k,\ell,\beta)$, $\Xi_{\frak p}$ and $K_{\circ}({\frak p})$
for its element $\frak p$ in the same way as in \cite[Section 11]{diskconst1}.
In fact, the set $\frak P(k,\ell,\beta)$ is the set
$\{{\bf p}_1,\dots,{\bf p}_{\mathscr P}\}$ appearing right above
\cite[(11.7)]{diskconst1}.
Here $K_{\circ}(\frak p)$ is the same as that of \cite[Section 11]{diskconst1}.
Condition \ref{cond73} is obviously satisfied from construction.
\par
In this case, ${\mathscr{QC}}_{k+1,\ell}(X,L,J;\beta)$ consists of the pair
$({\bf p},\frak p)$ such that ${\bf p} \in K_{\circ}(\frak p)$.
We define  ${\mathscr{F}}_{k+1,\ell,\beta}$ as follows.
We take compact subsets $K_{-}(\frak p) \subset K_{\circ}(\frak p)$
such that
\begin{equation}\label{form714}
\bigcup_{\frak p} {\rm Int}K_{-}(\frak p) = \mathcal M_{k+1,\ell}(X,L,J;\beta).
\end{equation}
We then put
$$
\aligned
{\mathscr{F}}^{\circ}_{k+1,\ell,\beta}({\bf p})
&=
\{({\bf p},\frak p) \mid {\bf p} \in {\rm Int}K_{-}(\frak p)\},
\\
{\mathscr{F}}_{k+1,\ell,\beta}({\bf p})
&=
\{({\bf p},\frak p) \mid {\bf p} \in K_{-}(\frak p)\}.
\endaligned
$$
Condition \ref{cond717} is immediate. Condition \ref{cond718} is void in this case.
Note that $({\bf p},\frak p)$ is the case $n=0$ of Situation \ref{situ69}.

In the same way as \cite[Lemma 11.7]{diskconst1} we can perturb $E_{\frak p}$ by an
arbitrary small amount so that Condition \ref{cond719} holds.
Condition \ref{cond720} is a consequence of (\ref{form714}).
We have thus completed the proof for the minimal $(k,\ell,\beta)$, that is
the first step of the induction.
\par
Next, we assume that we have already obtained
$\mathscr F^{\circ}_{k'+1,\ell',\beta'}$ and
$\mathscr F_{k'+1,\ell',\beta'}$ satisfying the required conditions
for all $(k',\ell',\beta')$ with $(k',\ell',\beta') < (k,\ell,\beta)$.
We will prove the same conclusion for the case of
$\mathscr F^{\circ}_{k+1,\ell,\beta}$ and
$\mathscr F_{k+1,\ell,\beta}$.
\par
We will first define an open subset $\mathscr F^{\circ \prime}_{k+1,\ell,\beta}$ and 
a compact subset
$\mathscr F^{\prime}_{k+1,\ell,\beta}$ of ${\mathscr{QC}}_{k+1,\ell}(X,L,J;\beta)$.
After that, we will modify them to obtain the desired
$\mathscr F^{\circ}_{k+1,\ell,\beta}$ and
$\mathscr F_{k+1,\ell,\beta}$.
\par
First we consider the case ${\bf p} \in \partial \mathcal M_{k+1,\ell}(X,L,J;\beta)$ and
define
$\mathscr F^{\circ \prime}_{k+1,\ell,\beta}({\bf p})$ and $\mathscr F^{\prime}_{k+1,\ell,\beta}({\bf p})$
to be the right hand sides of \eqref{form812}, \eqref{form81322} respectively.
Then we can show the following.
\begin{lem}\label{lem818}
$$
\bigcup_{{\bf p} \in \partial \mathcal M_{k+1,\ell}(X,L,J;\beta)}\mathscr F^{\circ \prime}_{k+1,\ell,\beta}({\bf p})
$$
is an open subset of
$\Pi^{-1}(\partial \mathcal M_{k+1,\ell}(X,L,J;\beta))
\subset {\mathscr{QC}}_{k+1,\ell}(X,L,J;\beta)$.
\end{lem}
\begin{proof}
Let ${\bf p} \in \partial \mathcal M_{k+1,\ell}(X,L,J;\beta)$
and ${\bf p}_j \in \partial \mathcal M_{k+1,\ell}(X,L,J;\beta)$.
Suppose $({\bf p}_j,\xi_j,\frak p_j)$ is a sequence of
quasi-components converging to  $({\bf p},\xi,\frak p)
\in \mathscr F^{\circ \prime}_{k+1,\ell,\beta}({\bf p})$.
It suffices to show that
$({\bf p}_j,\xi_j,\frak p_j) \in \mathscr F^{\circ \prime}_{k+1,\ell,\beta}({\bf p}_j)$
for sufficiently large $j$.
It is easy to see that $\frak p_j = \frak p$ for sufficiently large $j$.
\par
We take marked decorated rooted ribbon trees $\frak T_{(0)}$ and $\frak T_{(1)}$
such that
\begin{equation}\label{form813}
\aligned
{\bf p} \in &\mathcal M^{\circ}_{k+1,\ell}(X,L,J;\beta)(\frak T_{(0)}),  \\
{\bf p}_j \in &\mathcal M^{\circ}_{k+1,\ell}(X,L,J;\beta)(\frak T_{(1)}).
\endaligned
\end{equation}
(We may take $\frak T_{(1)}$ to be independent of $j$ by taking a subsequence
if necessary.)
\par
By \eqref{form812}, \eqref{form81322}, there exists an interior vertex ${\rm v}_0$ of $\frak T_{(0)}$ such that
$$
({\bf p},\xi,\frak p)
= \mathscr I_{{\bf p},{\rm v}_0}({\bf p}_{{\rm v}_0},\xi'_0,\frak p)
$$
for  $({\bf p}_{{\rm v}_0},\xi'_0,\frak p) \in
\mathscr F^{\circ \prime}_{k''+1,\ell'',\beta''}({\bf p}_{{\rm v}_0})$. More specifically
$\xi = ({\bf p}_{{\rm v}_0}, \xi_0')$.
\par
Note $\frak T_{(0)} \le \frak T_{(1)}$. Therefore
there exists a surjective map $\mathcal T_{(0)} \to \mathcal T_{(1)}$.
Let ${\rm v}$ be the image of ${\rm v}_0$ under this map.
\par
Using the fact that
$({\bf p}_j,\xi_j,\frak p)$ converges to $({\bf p},\xi,\frak p)$,
we can easily show that
there exists a sequence of
quasi-components $(({\bf p}_j)_{\rm v},\xi'_j,\frak p)$,
which determines $({\bf p}_j,\xi_j,\frak p)$ in the same way as above.
\par
Let $\frak S$ be the inverse image of the vertex ${\rm v}$
under the map $\mathcal T_{(0)} \to \mathcal T_{(1)}$.
Then again by \eqref{form812} 
there exists $({\bf p}_{\frak S},\xi_{\infty},\frak p)
\in \mathscr F^{\circ \prime}_{k'+1,\ell',\beta'}({\bf p}_{\frak S})$
which is determined by $({\bf p}_{{\rm v}_0},\xi'_0,\frak p)$.
Moreover we can show that $(({\bf p}_j)_{\rm v},\xi'_j,\frak p)$
converges to $({\bf p}_{\frak S},\xi_{\infty},\frak p)$.
\par
Since $\frak T_{(1)}$ is nontrivial,
$(k',\ell',\beta') < (k,\ell,\beta)$.
Therefore by the induction hypothesis we have $(({\bf p}_j)_{\rm v},\xi'_j,\frak p)
\in \mathscr F^{\circ \prime}_{k'+1,\ell',\beta'}(({\bf p}_j)_{\rm v})$
for all sufficiently large $j$.
Therefore again by \eqref{form812}, 
we have $({\bf p}_j,\xi_j,\frak p_j) \in \mathscr F^{\circ \prime}_{k+1,\ell,\beta}({\bf p}_j)$
for sufficiently large $j$.
\end{proof}
\begin{lem}\label{lem819}
The restriction of $\Pi$ to the subset
$$
\bigcup_{{\bf p} \in \partial \mathcal M_{k+1,\ell}(X,L,J;\beta)}\mathscr F^{\prime}_{k+1,\ell,\beta}({\bf p})
$$
is a proper map to $\partial \mathcal M_{k+1,\ell}(X,L,J;\beta)$.
\end{lem}
\begin{proof}
Let ${\bf p}_j,{\bf p} \in \partial \mathcal M_{k+1,\ell}(X,L,J;\beta)$
and suppose that ${\bf p}_j$ converges to ${\bf p}$.
Let $({\bf p}_j,\xi_j,\frak p) \in \mathscr F^{\prime}_{k+1,\ell,\beta}({\bf p}_j)$.
It suffices to show that $({\bf p}_j,\xi_j,\frak p)$
has a subsequence converging to an element $({\bf p},\xi,\frak p)
\in \mathscr F^{\prime}_{k+1,\ell,\beta}({\bf p})$.
\par
We take marked decorated rooted ribbon trees
$\frak T_{(0)}$ and $\frak T_{(1)}$ such that (\ref{form813}) holds.
(We take a subsequence of $\{{\bf p}_j\}$ if necessary.)
We may assume that there exists a sequence of interior vertices $\{{\rm v}_j\}$ contained in $\frak T_{(1)}$
such that
$$
\mathscr I_{{\bf p}_j,{\rm v}_j}({({\bf p}_j)}_{{\rm v}_j},\xi_j,\frak p) = ({\bf p}_j,\xi_j,\frak p).
$$
Then we may assume ${\rm v}_j = {\rm v}$ is independent of $j$.
Let $\frak S$ be the subgraph which is the inverse image of ${\rm v}$ in $\frak T_{(0)}$.
Using Lemma \ref{lem393939}, we find that $({\bf p}_j)_{\rm v}$ converges
to ${\bf p}_{\frak S}$.
The non-triviality of $\frak T_{(1)}$
and the induction hypothesis show that
we have a subsequence such that $({({\bf p}_j)}_{{\rm v}},\xi_j,\frak p)$
converges to $({\bf p}_{\frak S},\xi',\frak p)
\in \mathscr F^{\prime}_{k'+1,\ell',\beta'}({\bf p}_{\frak S})$.
Using \eqref{form81322}
twice, $({\bf p}_{\frak S},\xi',\frak p)$ determines an element $({\bf p},\xi,\frak p)
\in \mathscr F^{\prime}_{k+1,\ell,\beta}({\bf p})$,
to which $({\bf p}_j,\xi_j,\frak p)$ converges.
\end{proof}
We have thus defined $\mathscr F^{\prime}_{k+1,\ell,\beta}({\bf p})$,
$\mathscr F^{\circ \prime}_{k+1,\ell,\beta}({\bf p})$
for ${\bf p} \in \partial \mathcal M_{k+1,\ell}(X,L,J;\beta)$.
We next extend their definitions to a neighborhood of the boundary.
We take a sufficiently small $\rho > 0$,
with the following properties.
Let $({\bf p},\xi,\frak p)$ be an element
 of $\mathscr F^{\prime}_{k+1,\ell,\beta}({\bf p})$
 with
${\bf p} \in \partial \mathcal M_{k+1,\ell}(X,L,J;\beta)$
and ${\bf p}' \in \mathcal M_{k+1,\ell}(X,L,J;\beta)$
with $d({\bf p},{\bf p}') < \rho$.
Then there exists a representative
$({\bf p},\vec{\bf q},\frak p)$ of $({\bf p},\xi,\frak p)$
such that $({\bf p}',\vec{\bf q},\frak p)$ is a quasi-splitting sequence.
Existence of such $\rho$ is a consequence of Lemma \ref{lem818}.
We may take $\rho >0$ so small that $({\bf p}',\vec{\bf q},\frak p)$
does not depend on the representative.
We denote it by $({\bf p}',\xi,\frak p)$.
\par
Now for ${\bf p}' \in \mathcal M_{k+1,\ell}(X,L,J;\beta)$
we define:
\begin{defn}\label{defn820}
For $\rho > 0$ as above,
$\mathscr F^{\prime}_{k+1,\ell,\beta}({\bf p}')$ is defined to be
the set of all $({\bf p}',\xi,\frak p)$ such that
\begin{enumerate}
\item There exists ${\bf p} \in \partial \mathcal M_{k+1,\ell}(X,L,J;\beta)$.
\item There exists $({\bf p},\xi,\frak p) \in \mathscr F^{\prime}_{k+1,\ell,\beta}({\bf p})$.
\item
$$
d({\bf p}',{\bf p}) \le 2 d({\bf p}',\partial \mathcal M_{k+1,\ell}(X,L,J;\beta))
\le \rho/10.
$$
\end{enumerate}
We define $\mathscr F^{\circ\prime}_{k+1,\ell,\beta}({\bf p}')$
to be the set of all $({\bf p}',\xi,\frak p)$ such that
\begin{enumerate}
\item There exists ${\bf p} \in \partial \mathcal M_{k+1,\ell}(X,L,J;\beta)$.
\item There exists $({\bf p},\xi,\frak p) \in \mathscr F^{\circ\prime}_{k+1,\ell,\beta}({\bf p})$.
\item
 ${\bf p} = {\bf p}'$ or
$$
d({\bf p}',{\bf p}) < 2 d({\bf p}',\partial \mathcal M_{k+1,\ell}(X,L,J;\beta))
< \rho/10.
$$
\end{enumerate}
\end{defn}
Lemma \ref{lem818} implies that $\mathscr F^{\circ\prime}_{k+1,\ell,\beta}$
is open and Lemma \ref{lem819} implies that
$\mathscr F^{\prime}_{k+1,\ell,\beta}$ is proper.
By Item (3), Definition \ref{defn820} coincides with the
previously defined $\mathscr F^{\circ\prime}_{k+1,\ell,\beta}$,
$\mathscr F^{\prime}_{k+1,\ell,\beta}$ on the boundary.
Therefore $\mathscr F^{\circ\prime}_{k+1,\ell,\beta}$  satisfies
\eqref{form812}, \eqref{form81322}.
\par
We claim that we can choose $\rho$ so small that
$\mathscr F^{\circ\prime}_{k+1,\ell,\beta}$ and
$\mathscr F^{\prime}_{k+1,\ell,\beta}$ satisfy
Conditions \ref{cond719} and \ref{cond720} in a
small neighborhood of $\partial\mathcal M_{k+1,\ell}(X,L,J;\beta)$.
Indeed, this is an immediate consequence of \eqref{form812}, \eqref{form81322}
and the induction hypothesis on $\partial\mathcal M_{k+1,\ell}(X,L,J;\beta)$.
Then it holds on its small neighborhood.
\par
Now we choose (ob1), (ob2), (ob3) for $(k,\ell,\beta)$.
Then including them and quasi-splitting sequence of the form $({\bf p},\frak p)$
with $\frak p \in \frak P(k,\ell,\beta)$ we define
$\mathscr F_{k+1,\ell,\beta}$,
$\mathscr F^{\circ}_{k+1,\ell,\beta}$ 
and  $K_0({\frak p})$. 
This step is mostly the same as the first step of induction. The only difference is
we require that
$$
\bigcup_{\frak p \in \frak P(k,\ell,\beta)} {\rm Int}K_{-}(\frak p)
$$
contains  the complement of a small neighborhood of $\partial\mathcal M_{k+1,\ell}(X,L,J;\beta)$ 
in $\mathcal M_{k+1,\ell}(X,L,J;\beta)$,
instead of (\ref{form714})
and $K(\frak p) \subset \mathcal M_{k+1,\ell}^{\circ}(X,L,J;\beta)$.
Here the small neighborhood above is taken so that
$\mathscr F^{\circ\prime}_{k+1,\ell,\beta}$ and
$\mathscr F^{\prime}_{k+1,\ell,\beta}$ satisfy
Condition \ref{cond720} there.
\par
The proof of Proposition \ref{prop722} is now complete.
\end{proof}

\begin{rem}
Note that the number $\epsilon(k,\ell,\beta)$ depends 
on the set $\frak P(k',\ell',\beta')$
and that we use Lemma \ref{lem710} during the proof
of Proposition \ref{prop722}.
However the above proof is not circular.
This is because during the construction of the set
$\frak P(k,\ell,\beta)$ we use 
only $\epsilon(k',\ell',\beta')$ with 
$(k,\ell,\beta) \ge (k',\ell',\beta')$,
and such $\epsilon(k',\ell',\beta')$ 
depends only on $\frak P(k'',\ell'',\beta'')$
with $(k',\ell',\beta') > (k'',\ell'',\beta'')$.
\end{rem}

\begin{proof}[Proof of Lemma \ref{lem710}]
Note that in Lemma \ref{lem710} we are given 
finite sets $\frak P(k,\ell,\beta)$. 
We fix $\delta(\frak p)$ for each $\frak p \in \frak P(k,\ell,\beta)$ 
so that the map (\ref{form6363rev}) exists 
for this choice of  $\delta(\frak p)$.

We next take $\delta'(k,\ell,\beta) > 0$ for each $(k,\ell,\beta)$ so that the following 
holds.
Let $(k,\ell,\beta) = (k_1,\ell_1,\beta_1) > \dots > (k_{n},\ell_{n},\beta_{n})
\ge (k_{n+1},\ell_{n+1},\beta_{n+1})$ 
and $\frak p \in \frak  P(k_{n+1},\ell_{n+1},\beta_{n+1})$.
Then
\begin{equation}\label{new820}
\sum_{j=1}^{n+1} \delta'(k_j,\ell_j,\beta_j)
< \delta(\frak p).
\end{equation}
Note that (\ref{new820}) implies (\ref{newnew723}) when 
$\delta'_0 = \delta'({\bf p})$, $\delta'_j =  \delta'(k_j,\ell_j,\beta_j)$.
Here $\delta'({\bf p})$ is a sufficiently small positive number 
which may depend on ${\bf p}$.
We can find such $\delta'(k,\ell,\beta)$ by taking them to decay
sufficiently rapidly as $(k,\ell,\beta)$ increases with respect to the partial order $<$.
\par
We next claim the following.
There exists $\delta(k,\ell,\beta) > 0$ for each $(k,\ell,\beta)$
with the following properties.
Suppose the conclusions (1), (2), (3) of Lemma \ref{lem61060}
hold with $\delta_0 = \delta({\bf p})$, 
$\delta_j = \delta(k_j,\ell_j,\beta_j)$,
where $\delta_0({\bf p})$ is a small constant depending on ${\bf p}$.
Then (\ref{newform22}) holds with 
$\delta'_0 = \delta'({\bf p})$, $\delta'_j =  \delta'(k_j,\ell_j,\beta_j)$.
We can prove the existence of such $\delta(k,\ell,\beta) > 0$ in the same 
way as the proof of Lemma \ref{lem611}.
\par
Now we apply Lemma \ref{lem61060}.
Let  
$\delta_j = \delta(k_j,\ell_j,\beta_j)$ for $j=1,\dots,n+1$,
and $\delta_0$ a small constant depending on ${\bf p}$.
Then there exist constants as in (\ref{form611}) 
so that Lemma \ref{lem61060} holds.
\par
We remark that
the constants 
$
\epsilon_{1,j}(\delta_j;{\bf q}_j,\dots,{\bf q}_n,\frak p,\Xi_{\frak p})$ 
(appearing in (\ref{form611})) at this stage still depend on ${\bf q}_j$, $\frak p$.
Lemma \ref{lem710} which we are proving claims it depends only 
on $(k_j,\ell_j,\beta_j)$.
\par
For this purpose we prove the 
next sublemma by induction on $(k(0),\ell(0),\beta(0))$.
\begin{sublem}\label{sublemuni}
For any $(k(0),\ell(0),\beta(0)) \in \mathscr{TC}$, there 
exists $\epsilon(k,\ell,\beta)$ for $(k,\ell,\beta) \le (k(0),\ell(0),\beta(0))$ 
and $\epsilon(\frak p)$ for $\frak p \in \frak P(k,\ell,\beta)$  
such that the following holds.
\par
Suppose ${\bf p}$, ${\bf q}_j$, ${\frak p}$,
$\frak T_j$, $\frak S_j$ ($j=1,\dots,n$) $n \ge 1$ are
as in Situation \ref{situ69} (1)(2)(4)(7) and
\begin{enumerate}
\item[(3)'']
$d({\bf p},{\bf q}_1) \le \epsilon(k_1,\ell_1,\beta_1)$,
\item[(5)'']
For $j \le n-1$
we require ${\bf q}_{j,\frak S_j} \in \mathcal M_{k_{j+1}+1,\ell_{j+1}}(X,L,J;\beta_{j+1})$
and
$$
d({\bf q}_{j,\frak S_j}, {\bf q}_{j+1}) \le \epsilon(k_{j+1},\ell_{j+1},\beta_{j+1}),
$$
\item[(6)'']
We require ${\frak p}, {\bf q}_{n,\frak S_{n}} \in \mathcal M^{\circ}_{k_{n+1}+1,\ell_{n+1}}(X,L,J;\beta_{n+1})$ and
\begin{equation}
d({\bf q}_{n,\frak S_{n}},{\frak p}) \le \epsilon(\frak p),
\end{equation}
\end{enumerate}
with ${\bf p} \in \mathcal M_{k+1,\ell}(X,L,J;\beta)$
and $(k,\ell,\beta) \le (k(0),\ell(0),\beta(0))$.
Then the conclusions of Lemmas \ref{lem611} and \ref{lem6767rev} hold with the 
right hand sides of (\ref{newform22}) and (\ref{form6464}) replaced by $\delta(\frak p)$.
\par
In case $n=0$, $(k,\ell,\beta) \le (k(0),\ell(0),\beta(0))$ and
${\bf p},\frak p \in \mathcal M_{k+1,\ell}(X,L,J;\beta)$, the same holds under the assumption 
(\ref{nform86}).
\end{sublem}
\begin{proof}
We prove the sublemma by an upward induction on $(k(0),\ell(0),\beta(0))$.
\par
Suppose the sublemma is proved for $(k'(0),\ell'(0),\beta'(0)) < (k(0),\ell(0),\beta(0))$.
We prove the case of $(k(0),\ell(0),\beta(0))$. The case $n=0$ is easy.
\par
Let  ${\bf p}$, ${\bf q}_j$, ${\frak p}$,
$\frak T_j$, $\frak S_j$ ($j=1,\dots,n$), $n \ge 1$ be as in the 
assumption of the sublemma. Let ${\bf p} \in \mathcal M_{k+1,\ell}(X,L,J;\beta)$.
If   $(k,\ell,\beta) < (k(0),\ell(0),\beta(0))$ then the conclusion holds
by induction hypothesis. 
Let ${\bf p} \in \mathcal M_{k(0)+1,\ell(0)}(X,L,J;\beta(0))$.
\par
We first prove the part of the statement where ${\bf x}$ does not appear.
We apply the induction hypothesis to  the sequence
${\bf q}_{1,\frak S_1}$, ${\bf q}_{2}$,\dots, ${\bf q}_{n}$, $\frak p$.
Namely, ${\bf q}_{1,\frak S_1}$ plays the role of ${\bf p}$,
${\bf q}_{2}$  plays the role of ${\bf q}_{1}$ etc.
Then we obtain 
$\vec{\frak w}_{{\bf q}_j,{\frak p}} \enskip (j=1,\dots,n)$ such that
$$
d((\Sigma_{{\bf q}_{j+1}},\vec z_{{\bf q}_{j+1}},\vec{\frak z}_{{\bf q}_{j+1}}\cup \vec{\frak w}_{{\bf q}_{j+1},{\frak p}}),
(\Sigma_{{\bf q}_{j,\frak S_j}},\vec z_{{\bf q}_{j,\frak S_j}},\vec{\frak z}_{{\bf q}_{j,\frak S_j}}\cup
\vec{\frak w}_{{\bf q}_j,{\frak p}})) < \delta_{{j+1}}.
$$
Here $\delta_{{j+1}}= \delta(k_{j+1},\ell_{j+1},\beta_{j+1})$.
\par
Also we apply Lemma \ref{lem61060} to obtain the following.
There exists $\epsilon_1(\delta_1;{\bf q}_1,\dots,{\bf q}_n,\frak p,\Xi_{\frak p})$
with the following properties.
\par
If $d({\bf p},{\bf q}_1) < \epsilon_1(\delta_1;{\bf q}_1,\dots,{\bf q}_n,\frak p,\Xi_{\frak p})$
then there exists $\vec{\frak w}_{{\bf p};{\frak p }}$ such that 
$$
d((\Sigma_{{\bf q}_1},\vec z_{{\bf q}_1},\vec{\frak z}_{{\bf q}_1}\cup\vec {\frak w}_{{\bf q}_1;{\frak p}}),(\Sigma_{\bf p},
\vec z_{{\bf p}},\vec{\frak z}_{{\bf p}}\cup\vec{\frak w}_{{\bf p};{\frak p}})) < \delta_{{1}}. 
$$
We claim that we may take $\epsilon_1(\delta_1;{\bf q}_1,\dots,{\bf q}_n,\frak p,\Xi_{\frak p})$
which is independent of the choices of 
${\bf q}_1,\dots,{\bf q}_n,\frak p,\Xi_{\frak p}$.
\par
This follows from the following two observations.
The set of the sequences ${\bf q}_1$,\dots,${\bf q}_n$,$\frak p$ such that
$$
\aligned
d({\bf q}_{j,\frak S_j}, {\bf q}_{j+1}) &\le \epsilon(k_{j+1},\ell_{j+1},\beta_{j+1}) \\
d({\bf q}_{n,\frak S_n}, \frak p) &\le \epsilon(\frak p)
\endaligned
$$
is compact.  (We note that in (3)'', (5)'', (6)'' we replace 
the strict inequalities $<$, which are used in (3), (5), (6) in Situation \ref{situ69}, by the inequalities 
$\le$.)
\par
Moreover if ${\bf q}'_j$ is in a small neighborhood of  ${\bf q}_j$ we may take
$$
\epsilon_1(\delta_1;{\bf q}'_1,\dots,{\bf q}'_n,\frak p,\Xi_{\frak p})
=
\epsilon_1(\delta_1;{\bf q}_1,\dots,{\bf q}_n,\frak p,\Xi_{\frak p}).
$$
\par
Thus we completed the induction step except the 
statement related to ${\bf x}$. 
But actually it follows in the same way. 
In fact,  the constant 
$\epsilon_{0}(\delta_{0};{\bf p},{\bf q}_{1},\dots,{\bf q}_n,{\frak p},\Xi_{\frak p})$
which estimates $d({\bf x},{\bf p})$ {\it may} depend on 
${\bf p},{\bf q}_{1},\dots,{\bf q}_n,{\frak p},\Xi_{\frak p}$.
(In other words we do not claim its uniformity in Lemma \ref{lem710}.)
The proof of Sublemma \ref{sublemuni} is complete.
\end{proof}
\par
The proof of Lemma \ref{lem710} is complete.
\end{proof}
Therefore the proof of Theorem \ref{thm43} is now complete.
\qed

\section{Uniqueness of the Kuranishi structure up to pseudo isotopy}
\label{subsec;unique1}

\subsection{The case of a single $K$-space}
\label{subsec:unique1}

We first recall the notion of KK-embedding of Kuranishi structures.
\begin{defn}(\cite[Definition 3.20]{foootech2}, 
\cite[Definition 3.19]{Springer}).\label{def81}
Let $\mathcal M$ be a compact metrizable space and
$\widehat{\mathcal U^{(i)}} =(\{\mathcal U^{(i)}_{{\bf p}}\},\{\Phi^{(i)}_{{\bf p}{\bf q}}\})$
a Kuranishi structure on it, for $i=1,2$.
A {\it strict KK-embedding} $\widehat{\mathcal U^{(1)}} \to \widehat{\mathcal U^{(2)}}$
assigns an embedding of Kuranishi charts $\Phi_{\bf p} : \mathcal U^{(1)}_{\bf p}
\to \mathcal U^{(2)}_{\bf p}$ to each ${\bf p} \in \mathcal M$ such that
\begin{equation}\label{form8181}
\Phi_{\bf p} \circ \Phi^{(1)}_{{\bf p}{\bf q}}\vert_{U^{(1)}_{{\bf p}{\bf q}} \cap \varphi_{\bf q}^{-1}(U^{(2)}_
{{\bf p}{\bf q}})}
=
\Phi^{(2)}_{{\bf p}{\bf q}}\circ \Phi_{\bf q}\vert_{U^{(1)}_{{\bf p}{\bf q}} \cap \varphi_{\bf q}^{-1}(U^{(2)}_{{\bf p}{\bf q}})}.
\end{equation}
\par
A {\it KK embedding} of germs of Kuranishi structures is by definition the germ of a strict KK embedding between the representatives.
We can compose two KK-embeddings in an obvious way.
(\cite[Definition 5.16]{foootech2}, \cite[Definition 5.14]{Springer}.)
\end{defn}
An explanation of the notations in Definition \ref{def81} is in order.
A {\it Kuranishi chart} $\mathcal U^{(i)}_{{\bf p}}$ is given by $(U^{(i)}_{\bf p},E^{(i)}_{\bf p},s^{(i)}_{\bf p},\psi^{(i)}_{\bf p})$
where $U^{(i)}_{\bf p}$ is a Kuranishi neighborhood (an orbifold), $E^{(i)}_{\bf p}$ is an obstruction
bundle (a vector bundle on it), $s^{(i)}_{\bf p}$ is a Kuranishi map (a section of the obstruction bundle)
and $\psi^{(i)}_{\bf p} : (s^{(i)}_{\bf p})^{-1}(0) \to \mathcal M$ is a parametrization map
(a homeomorphism onto its image, which is open).
\par
An {\it embedding of Kuranishi charts} $\Phi_{{\bf p}{\bf q}} : (U_{\bf q},E_{\bf q},s_{\bf q},\psi_{\bf q}) \to (U_{\bf p},E_{\bf p},s_{\bf p},\psi_{\bf p})$
is a triple $(U_{{\bf p}{\bf q}},\varphi_{{\bf p}{\bf q}},\hat\varphi_{{\bf p}{\bf q}})$,
where $U_{{\bf p}{\bf q}} \subset U_{\bf q}$ is an open subset,
$\varphi_{{\bf p}{\bf q}} : U_{{\bf p}{\bf q}} \to U_{{\bf p}}$ is an embedding of orbifolds,
and $\hat\varphi_{{\bf p}{\bf q}} : E_{\bf q}\vert_{U_{{\bf p}{\bf q}}} \to E_{{\bf p}}$
is an embedding of vector bundles\footnote{orbibundles} which covers $\varphi_{{\bf p}{\bf q}}$.
We require the embedding $\Phi_{{\bf p}{\bf q}}$ to satisfy
certain compatibility with Kuranishi map and parametrization map. See \cite[Definition 3.2]{foootech2}, \cite[Definition 3.2]{Springer}.
\par
For a system $\widehat{\mathcal U^{(i)}} =(\{\mathcal U^{(i)}_{{\bf p}}\},\{\Phi^{(i)}_{{\bf p}{\bf q}}\})$
of Kuranishi charts $\mathcal U^{(i)}_{{\bf p}}$
and embeddings $\Phi^{(i)}_{{\bf p}{\bf q}} : \mathcal U^{(i)}_{{\bf q}} \to \mathcal U^{(i)}_{{\bf p}}$
to form a Kuranishi structure, we require appropriate compatibility
conditions. See \cite[Definition 3.8]{foootech2}, \cite[Definition 3.9]{Springer}.

\begin{defn}\label{defn8282}
$ $
\begin{enumerate}
\item
Let $\mathscr E^{(i)} = \{E^{(i)}_{\bf p}(\bf x)\}$ be two obstruction bundle data of
${\mathcal M}_{k+1,\ell}(X,L,J;\beta)$ for $i=1,2$.
We say $\mathscr E^{(1)}$ is {\it contained in} $\mathscr E^{(2)}$
and write $\mathscr E^{(1)} \subseteq \mathscr E^{(2)}$
if for each ${\bf p} \in {\mathcal M}_{k+1,\ell}(X,L,J;\beta)$ there exists its neighborhood $\mathscr U_{\bf p}$
in ${\mathcal X}_{k+1,\ell}(X,L;\beta)$ such that $E^{(1)}_{\bf p}({\bf x})
\subseteq E^{(2)}_{\bf p}({\bf x})$ for ${\bf x} \in \mathscr U_{\bf p}$.
\item
We say two obstruction bundle data $\mathscr E$,  $\mathscr E'$ are {\it equivalent}
if there exist obstruction bundle data $\mathscr E^{i}$, $i=0,\dots,2m$ such that:
\begin{enumerate}
\item $\mathscr E^{0} = \mathscr E$, $\mathscr E^{2m} = \mathscr E'$.
\item
$\mathscr E^{2j-1} \supseteq \mathscr E^{2j} \subseteq \mathscr E^{2j+1}$.
\end{enumerate}
\end{enumerate}
\end{defn}
\begin{prop}\label{prop83}
Let $\mathscr E^{(i)} = \{E^{(i)}_{\bf p}(\bf x)\}$ be  obstruction bundle data of
${\mathcal M}_{k+1,\ell}(X,L,J;\beta)$
and $\widehat{\mathcal U^{(i)}}$ a Kuranishi structure on ${\mathcal M}_{k+1,\ell}(X,L,J;\beta)$
associated to $\mathscr E^{(i)}$ by \cite[Theorem 7.1]{diskconst1}
for $i=1,2,3$ respectively.
If $\mathscr E^{(1)} \subseteq \mathscr E^{(2)}$, then there exists a KK-embedding
$\widehat{\mathcal U^{(1)}} \to \widehat{\mathcal U^{(2)}}$.
Furthermore
suppose $\mathscr E^{(1)} \subseteq \mathscr E^{(2)} \subseteq \mathscr E^{(3)}$.
Let $\Phi_{ij} : \widehat{\mathcal U^{(j)}} \to \widehat{\mathcal U^{(i)}}$ be the above
KK embeddings for $i>j$. Then we have
$$
\Phi_{32} \circ \Phi_{21} = \Phi_{31}.
$$
\end{prop}
\begin{proof}
We recall
$$
U^{(i)}_{\bf p} = \{{\bf x} = [(\Sigma_{{\bf x}},\vec z_{{\bf x}},\vec{\frak z}_{{\bf x}}),
u_{{\bf x}}]  \in \mathscr U_{{\bf p}} \mid
\overline\partial u_{{\bf x}} \in E^{(i)}_{\bf p}({\bf x})\}.
$$
Therefore $E^{(1)}_{\bf p}({\bf x})
\subseteq E^{(2)}_{\bf p}({\bf x})$ implies
$U^{(1)}_{\bf p} \subseteq U^{(2)}_{\bf p}$ set theoretically.\footnote{We remark
that $U^{(i)}_{\bf p}$ is a subset of ${\mathcal X}_{k+1,\ell}(X,L;\beta)$ for $i=1,2$.}
The fact that the inclusion map is induced by a smooth embedding of orbifolds
can be proved in the same way as the smoothness of the coordinate change
given in \cite[Subsection 12.2]{diskconst1}\footnote{There was a minor typographical error in the statement of \cite[Lemma 10.11]{diskconst1}
which is corrected in the recent arXiv version.}
using \cite[Theorem 6.4]{foooexp}.
Since $E^{(i)}_{\bf p}({\bf x})$ is the fiber of the obstruction bundle, 
the embedding $U^{(1)}_{\bf p} \to U^{(2)}_{\bf p}$ is covered by an
embedding of obstruction bundles.
Compatibility with Kuranishi map, parametrization map, and coordinate change
can be proved in the same way as the corresponding statement for the
coordinate change (\cite[Subsections 7.3 and 7.4]{diskconst1}).
The second half is obvious from definition.
\end{proof}
\begin{defn}\label{defn94}
Let $\widehat{\mathcal U^{(i)}}$ be germs of oriented Kuranishi structures on $\mathcal M$
without boundary for
$i=1,2$. We say  $\widehat{\mathcal U^{(1)}}$ is {\it isotopic}
to  $\widehat{\mathcal U^{(2)}}$ if there exists a Kuranishi structure
with boundary $\widehat{\mathcal U}$ on $[1,2] \times \mathcal M$
with the following properties.
\begin{enumerate}
\item
$$
\partial ([1,2] \times \mathcal M,\widehat{\mathcal U})
= -(\mathcal M,\widehat{\mathcal U^{(1)}})
\sqcup (\mathcal M,\widehat{\mathcal U^{(2)}}).
$$
Here the underlying topological space of
$(\mathcal M,\widehat{\mathcal U^{(i)}})$ is identified with
$\{i\} \times \mathcal M$, for $i=1,2$.
\item
There exists $\epsilon > 0$ such that there exist
isomorphisms of germs of Kuranishi structures:
$$
\aligned
&([1,1+\epsilon] \times \mathcal M,\widehat{\mathcal U}\vert_{[1,1+\epsilon] \times \mathcal M})
\cong  [1,1+\epsilon] \times (\mathcal M,\widehat{\mathcal U^{(1)}})   \\
&([2-\epsilon,2] \times \mathcal M,\widehat{\mathcal U}\vert_{[2-\epsilon,2] \times \mathcal M})
\cong  [2-\epsilon,2]  \times (\mathcal M,\widehat{\mathcal U^{(2)}})
\endaligned
$$
See \cite[Subsection 4.1]{foootech2}, 
\cite[Section 4.1]{Springer} for the definition of product of Kuranishi structures.
\end{enumerate}
\end{defn}
\begin{rem}
The definition of pseudo-isotopy in \cite[Definition 21.15]{foootech21},
\cite[Definition 21.15]{Springer}
is similar to Definition \ref{defn94} but we did not assume (2).
The reason why we assume (2) here is because  it is then obvious that
`isotopic' becomes an equivalence relation.
\end{rem}
\begin{lem}\label{lem8686}
Let $\widehat{\mathcal U^{(i)}}$ be germs of oriented Kuranishi structures on $\mathcal M$
without boundary for
$i=1,2$.
Suppose there exists an orientation preserving  KK-embedding of Kuranishi structures
$\widehat{\mathcal U^{(1)}} \to \widehat{\mathcal U^{(2)}}$.
Then $\widehat{\mathcal U^{(1)}}$ is isotopic
to $\widehat{\mathcal U^{(2)}}$.
\end{lem}
\begin{proof}
We use the notation of Definition \ref{def81} and the explanation thereafter
and will construct Kuranishi structure $\widehat{\mathcal U}$ on $[1,2] \times \mathcal M$.
\par
Let $(t,{\bf p}) \in  [1,2] \times \mathcal M$.

Suppose $t < 3/2$. We take $\delta < \min\{ 3/2-t, t-1\}$
and
$$
\mathcal U_{(t,{\bf p})}
= (t-\delta,t+\delta) \times \mathcal U^{(1)}_{{\bf p}}.
$$
Suppose $t \ge 3/2$. We take $\delta < 2-t$ and
$$
\mathcal U_{(t,{\bf p})}
= (t-\delta,t+\delta) \times \mathcal U^{(2)}_{{\bf p}}.
$$
We next define a coordinate change between them.
Let $(t,{\bf p}), (s,{\bf q}) \in [1,2] \times  \mathcal M$
such that
$$
(s,{\bf q}) \in \rm{Im} \psi_{(t,{\bf p})}.
$$
There are three cases.
\par\smallskip
\noindent(Case 1): $t,s < 3/2$.
In this case we put
$$
\aligned
& U_{(t,{\bf p})(s,{\bf q})} =
((t-\delta,t+\delta)\cap (s-\delta',s+\delta')) \times
U^{(1)}_{{\bf p}{\bf q}} \\
& \varphi_{(t,{\bf p})(s,{\bf q})} = {\rm id} \times \varphi^{(1)}_{{\bf p}{\bf q}}, \quad
\hat\varphi_{(t,{\bf p})(s,{\bf q})} = {\rm id} \times \hat\varphi^{(1)}_{{\bf p}{\bf q}}.
\endaligned
$$
Here $\delta$ and $\delta'$ are chosen for $(t,{\bf p})$ and $(s,{\bf q})$ as above, respectively.
\par
\noindent(Case 2): $t,s \ge 3/2$.
In this case we put:
$$
\aligned
& U_{(t,{\bf p})(s,{\bf q})} = ((t-\delta,t+\delta)\cap (s-\delta',s+\delta'))\times  U^{(2)}_{{\bf p}{\bf q}} \\
& \varphi_{(t,{\bf p})(s,{\bf q})} = {\rm id} \times \varphi^{(2)}_{{\bf p}{\bf q}}, \quad
\hat\varphi_{(t,{\bf p})(s,{\bf q})} = {\rm id} \times \hat\varphi^{(2)}_{{\bf p}{\bf q}}.
\endaligned
$$
\noindent(Case 3): $t \ge 3/2 > s$.
In this case we put:
$$
\aligned
& U_{(t,{\bf p})(s,{\bf q})} = ((t-\delta,t+\delta)\cap (s-\delta',s+\delta')) \times
(\varphi_{\bf q}^{-1}(U^{(2)}_{{\bf p}{\bf q}}) \cap U_{{\bf p}{\bf q}}^{(1)})\\
& \varphi_{(t,{\bf p})(s,{\bf q})} = {\rm id} \times \varphi^{(2)}_{{\bf p}{\bf q}}\circ \varphi_{\bf q}, \quad
\hat\varphi_{(t,{\bf p})(s,{\bf q})} = {\rm id}
 \times \hat\varphi^{(2)}_{{\bf p}{\bf q}}\circ \hat\varphi_{\bf q}.
\endaligned
$$
The compatibility with Kuranishi map and parametrization map of them
follows easily from the fact that $\Phi^{(i)}_{{\bf p}{\bf q}}$ is a coordinate change
and $\Phi_{{\bf q}}$ is an embedding of Kuranishi charts.
\par
Using (\ref{form8181}) and the compatibilities of coordinate changes for $\widehat{\mathcal U^{(i)}}$,
we can show the compatibility for the above system to be a Kuranishi structure.
The properties (1),(2) above are immediate from construction.
\end{proof}
Proposition \ref{prop83} and Lemma \ref{lem8686} imply that if ${\mathcal M}_{k+1,\ell}(X,L,J;\beta)$
has no boundary then the Kuranishi structure we obtain from \cite[Theorem 7.1]{diskconst1}
depends only on the equivalence class of obstruction bundle data up to isotopy.
Moreover we can show the next result.

\begin{thm}\label{them8787}
Any two obstruction bundle data of ${\mathcal M}_{k+1,\ell}(X,L,J;\beta)$
are equivalent in the sense of Definition \ref{defn8282} (2).
\end{thm}
\begin{proof}
Let $\mathscr E = \{E_{\bf p}({\bf x})\}$ and
$\mathscr E' = \{E'_{\bf p}({\bf x})\}$ be two obstruction bundle data of ${\mathcal M}_{k+1,\ell}(X,L,J;\beta)$.
We will show that $\mathscr E$ is equivalent to $\mathscr E'$.
\begin{lem}\label{lem8888}
For any ${\bf p} \in {\mathcal M}_{k+1,\ell}(X,L,J;\beta)$
there exist obstruction bundle data $E^0_{\bf p}({\bf x})$ at ${\bf p}$
(in the sense of \cite[Definition 5.1]{diskconst1}) and a neighborhood $\mathscr U_{\bf p}$ of ${\bf p}$
in ${\mathcal X}_{k+1,\ell}(X,L;\beta)$ such that the following holds in addition.
\begin{equation}\label{formula82}
E_{\bf p}({\bf x}) \cap E^0_{\bf p}({\bf x}) = \{0\},
\qquad
E'_{\bf p}({\bf x}) \cap E^0_{\bf p}({\bf x}) = \{0\}
\end{equation}
if ${\bf x} \in \mathscr U_{\bf p}$.
\end{lem}
\begin{proof}
The proof is mostly the same as that of \cite[Proposition 11.4]{diskconst1}.
We first take a finite dimensional subspace $E^0_{\bf p}({\bf p})$
satisfying \cite[Lemma 11.2 (1)(2)(3)]{diskconst1}.
We may take the subspace $E^0_{\bf p}({\bf p})$ so that (\ref{formula82}) holds at ${\bf x} = {\bf p}$.
We then choose additional marked points
 to stabilize the 
domain of ${\bf p}$ and also codimension 2 submanifolds 
which intersect transversally 
with the image of $u_{\bf p}$ at those 
marked points.
They determine the corresponding marked points 
of the domain of ${\bf x}$, that is nothing but the points sent to the codimension 2 submanifolds 
by $u_{\bf x}$.
Therefore the domains of ${\bf p}$ and ${\bf x}$ are stabilized.
We now use the local trivialization of the 
universal family of the domains and the parallel transport
to define $E^0_{\bf p}({\bf x})$.
See \cite[(11.1)]{diskconst1} for detail.
Since (\ref{formula82}) is an open condition we can take a
neighborhood $\mathscr U_{\bf p}$ of ${\bf p}$
small so that (\ref{formula82}) holds for ${\bf x} \in \mathscr U_{\bf p}$. The fact that
${\bf x} \mapsto E^0_{\bf p}({\bf x})$ is obstruction bundle data at ${\bf p}$
is proved in \cite[Subsections 11.1 and 11.2]{diskconst1}.
\end{proof}
Let $\mathscr U_{\bf p}$ be as in Lemma \ref{lem8888}. We put:
\begin{equation}
\mathfrak U_{\bf p} = \mathscr U_{\bf p} \cap {\mathcal M}_{k+1,\ell}(X,L,J;\beta).
\end{equation}
This is an open neighborhood of ${\bf p}$ in
${\mathcal M}_{k+1,\ell}(X,L,J;\beta)$.
For ${\bf q} \in \mathfrak U_{\bf p}$ we take its neighborhood $\mathscr U_{{\bf q}:{\bf p}}$
in ${\mathcal X}_{k+1,\ell}(X,L;\beta)$ such that
$\mathscr U_{{\bf q}:{\bf p}} \subset \mathscr U_{{\bf p}}$.
For ${\bf x} \in \mathscr U_{{\bf q}:{\bf p}}$ we define
$$
E^0_{{\bf q}:{\bf p}}({\bf x}) = E^0_{{\bf p}}({\bf x}).
$$
By \cite[Proposition 11.4]{diskconst1}, ${\bf x} \mapsto E^0_{{\bf q}:{\bf p}}({\bf x})$
is obstruction bundle data at ${\bf q}$.
\par
For ${\bf p} \in {\mathcal M}_{k+1,\ell}(X,L,J;\beta)$ we take its open neighborhood
$K_{0}({\bf p})$  in ${\mathcal M}_{k+1,\ell}(X,L,J;\beta)$
such that its closure $K({\bf p})$ is contained in $\mathfrak U_{\bf p}$.
Since $\mathcal M_{k+1,\ell}(X,L,J;\beta)$ is compact, we can find a
finite subset $\{{\bf p}_1,\dots,{\bf p}_{\mathscr P}\}$ of $\mathcal M_{k+1,\ell}(X,L,J;\beta)$
such that
\begin{equation}\label{2199}
\bigcup_{i=1}^{\mathscr P} K_{0}({\bf p}_i) = \mathcal M_{k+1,\ell}(X,L,J;\beta).
\end{equation}
For ${\bf q} \in \mathcal M_{k+1,\ell}(X,L,J;\beta)$ we put
$
I({\bf q}) = \{ i \in \{1,\dots,\mathscr P\} \mid {\bf q} \in K({\bf p}_i) \}.
$\index{00I1qq@$I({\bf q})$}
\begin{lem}\label{lem1116}
We may perturb $E^0_{{\bf p}_i}({\bf p}_i)$ by an arbitrary small amount in $C^{2}$ norm
so that the following holds.
For each ${\bf q} \in \mathcal M_{k+1,\ell}(X,L,J;\beta)$
the sum $\sum_{i \in I({\bf q})} E^{0}_{{\bf q},{\bf p}_i}({\bf q})$
of vector subspaces in $C^{\infty}(\Sigma_{\bf q}(\vec{\epsilon'});u_{{\bf q}}^*TX\otimes \Lambda^{0,1})$
is a direct sum
\begin{equation}\label{formula85}
\bigoplus_{i \in I({\bf q})} E^{0}_{{\bf q},{\bf p}_i}({\bf q}).
\end{equation}
Moreover
\begin{equation}\label{formula86}
\bigoplus_{i \in I({\bf q})} E^{0}_{{\bf q},{\bf p}_i}({\bf q})
\cap E'_{\bf q}({\bf q})
= \{0\},
\quad
\bigoplus_{i \in I({\bf q})} E^{0}_{{\bf q},{\bf p}_i}({\bf q})
\cap E_{\bf q}({\bf q})
= \{0\}.
\end{equation}
\end{lem}
\begin{proof}
The proof is the same as that of
\cite[Lemma 11.7]{diskconst1}.
Actually (\ref{formula85}) is proved there.
We can perturb $E^0_{{\bf p}_i}({\bf p}_i)$ more so that the other condition (\ref{formula86}) is satisfied
by the same argument.
\end{proof}
Now we define $\mathscr E^{(j)} = \{E^{(j)}_{\bf q}({\bf x})\}$,
$j=0,1,2,3,4$ as follows:
$\mathscr E^{(0)} = \mathscr E$, $\mathscr E^{(4)} = \mathscr E'$
$$
\aligned
E^{(1)}_{\bf q}({\bf x}) &=
\bigoplus_{i \in I({\bf q})} E^{0}_{{\bf q},{\bf p}_i}({\bf x})
\oplus E_{\bf q}({\bf x}), \\
E^{(2)}_{\bf q}({\bf x}) &=
\bigoplus_{i \in I({\bf q})} E^{0}_{{\bf q},{\bf p}_i}({\bf x}),
\\
E^{(3)}_{\bf q}({\bf x}) &=
\bigoplus_{i \in I({\bf q})} E^{0}_{{\bf q},{\bf p}_i}({\bf x})
\oplus E'_{\bf q}({\bf x}).
\endaligned
$$
It is easy to check that $\{\mathscr E^{(j)} \mid j=0,\dots,4\}$
satisfies Definition \ref{defn8282} (2), (a) and (b).
The proof of Theorem \ref{them8787} is complete.
\end{proof}

\subsection{The case of a system of $K$-spaces}
\label{subsec:unique2}

So far in this section we have studied a single moduli space
${\mathcal M}_{k+1,\ell}(X,L,J;\beta)$.
In the rest of this section we provide its tree-like K-system version.
The way to do so is rather straightforward so we will be sketchy sometimes.
\begin{defn}\label{81010}
Let $\mathbb U^{(i)} = \{\widehat{\mathcal U^{(i)}_{k+1,\ell,\beta}}\} =\{(\{\mathcal U^{(i)}_{{\bf p},k+1,\ell,\beta}\},
\{\Phi^{(i)}_{{\bf p}{\bf q},k+1,\ell,\beta}\})\}$ $(i=1,2)$
be two systems of Kuranishi structures on $\{\mathcal M_{k+1,\ell}(X,L,J;\beta)\}$
which consist of tree-like K-systems with interior marked points
(Definition \ref{defn2167}).
\par
An {\it embedding from the system $\mathbb U^{(1)}$ to $\mathbb U^{(2)}$} is a system of
embeddings
from $\widehat{\mathcal U^{(1)}_{k+1,\ell,\beta}}$ to $\widehat{\mathcal U^{(2)}_{k+1,\ell,\beta}}$
such that they commute with the evaluation maps in Theorem \ref{therem274} Item (III),
preserve orientations in Item (VII) and the corner compatibility isomorphisms
in Items (IX) (X).
\end{defn}
\begin{defn}\label{defn8111}
Let $\mathscr E^{(i)} = \{E^{(i)}_{\bf p}({\bf x})\}$ $(i=1,2)$ be two disk-component-wise systems of obstruction bundle data.
We say {\it $\mathscr E^{(1)}$ is contained in $\mathscr E^{(2)}$} if
$E^{(1)}_{\bf p}({\bf x}) \subseteq E^{(2)}_{\bf p}({\bf x})$ when both hand sides are defined.
\par
We define an equivalence of two disk-component-wise systems of obstruction bundle data
in the same way as Definition \ref{defn8282} (2).
\end{defn}
\begin{lem}
Let $\mathscr E^{(i)}$ $(i=1,2)$ be two disk-component-wise systems of obstruction bundle data.
For each $i=1,2$, let $\mathbb U^{(i)}$ be the system of Kuranishi structures on $\{\mathcal M_{k+1,\ell}(X,L,J;\beta)\}$
consisting of a  tree-like K-system with interior marked points, which is produced by
Theorem \ref{therem274}.
\par
If $\mathscr E^{(1)}$ is contained in $\mathscr E^{(2)}$, then
there exists an embedding from the system $\mathbb U^{(1)}$ to $\mathbb U^{(2)}$.
The same statement as the second half of Proposition \ref{prop83} holds.
\end{lem}
\begin{proof}
The embedding of Kuranishi structures on each  $\mathcal M_{k+1,\ell}(X,L,J;\beta)$
can be constructed by Proposition \ref{prop83}.
The compatibility of this embedding required in Definition \ref{81010}
is obvious from construction.
\end{proof}
The notion of pseudo-isotopy between two tree-like $K$-systems (= $A_{\infty}$
correspondences) is given in \cite[Definition 21.15]{foootech21}, 
\cite[Definition 21.15]{Springer} and Definition \ref{def218}. 
It gives a system of K-spaces $\mathcal M_{k+1}(\beta;[1,2])$
with some compatibility conditions.   
\begin{defn}\label{defn913}
We consider the case when the underlying topological spaces of
two tree-like  $K$-systems
are the same
and are $\mathcal M_{k+1,\ell}(X,L,J;\beta)$.
(But their Kuranishi structures may be different.)
\par
We say a pseudo-isotopy (in the sense of Definition \ref{def218}) between them is
an {\it isotopy} if:
\begin{enumerate}
\item $\mathcal M_{k+1,\ell}(X,L,J;\beta;[1,2])$ is homeomorphic to $
[1,2] \times \mathcal M_{k+1,\ell}(X,L,J;\beta)$.
\item  The restriction of Kuranishi structure on $\mathcal M_{k+1,\ell}(X,L,J;\beta;[1,2])$
to the subspaces $\mathcal M_{k+1,\ell}(X,L,J;\beta;[1,1+\epsilon])$ and $\mathcal M_{k+1,\ell}(X,L,J;\beta;[2-\epsilon,2])$
are isomorphic to the direct product of the Kuranishi structure on $\mathcal M_{k+1,\ell}(X,L,J;\beta)$
and a trivial Kuranishi structure on the subset of the interval.
\item The isomorphism in (2) commutes with various isomorphisms describing compatibility
at the corner etc..
\end{enumerate}
\end{defn}
\begin{lem}
Let $\mathbb U^{(i)}$ ($i=1,2$) be two objects as in Definition \ref{81010}.
Suppose there exists an embedding $\mathbb U^{(1)} \to \mathbb U^{(2)}$.
Then those systems are isotopic.
\end{lem}
\begin{proof}
The Kuranishi structure on $\mathcal M_{k+1,\ell}(X,L,J;\beta;[1,2])$
for each $k,\ell,\beta$
is constructed during the proof of Lemma \ref{lem8686}.
Various compatibility conditions required to show that they are isotopies
are obvious from construction.
\end{proof}
\begin{prop}
Any two disk-component-wise systems of obstruction bundle data
are equivalent in the sense of Definition \ref{defn8111}.
\end{prop}
\begin{proof}
Let
$\mathscr E^{(i)} = \{E^{(i)}_{\bf p}({\bf x})\}$ be two disk-component-wise systems of obstruction bundle data
for $i=1,2$.
In the same way as the proof of Theorem \ref{thm43} we can find another
disk-component-wise system of obstruction bundle data $\{E^{(0)}_{\bf p}({\bf x})\}$
such that
$$
E^{(0)}_{\bf p}({\bf x}) \cap E^{(i)}_{\bf p}({\bf x}) = \{0\}
$$
for $i=1,2$ if ${\bf x}$ is in a small neighborhood of ${\bf p}$.
The rest of the proof is the same as the last part of the proof of
Theorem \ref{them8787}.
\end{proof}
Combining the above results we obtain:
\begin{thm}\label{thm96}
The system of Kuranishi structures obtained in Theorem \ref{thm42}
depends only on $(X,\omega,J)$ and $(L,\sigma)$ up to isotopy, where
$\sigma$ is a relative spin structure on $L$.
In particular, it
is independent of the choice of obstruction bundle data up to isotopy.
\end{thm}
It implies independence of Kuranishi structures obtained in  Theorem \ref{therem274}
up to isotopy.

\subsection{Independence of almost complex structure up to pseudo-isotopy}
\label{subsec;unique2}

In this subsection we prove Theorem \ref{therem274}
in the case $P$ is not necessarily a point
and
Theorems \ref{pisotopyexixsts}, \ref{thm219}.
Suppose we are in Situation \ref{situ213}.
We define a stable map topology on the parametrized moduli space
$\mathcal M_{k+1,\ell}(X,L,\mathcal J;\beta)$ in the same way as the case we fix one almost complex structure.
We define its ambient set as follows.
$$
\mathcal X_{k+1,\ell}(X,L;\beta;P)
=
P \times \mathcal X_{k+1,\ell}(X,L;\beta).
$$
Note $\mathcal X_{k+1,\ell}(X,L;\beta)$ is independent of
the almost complex structure.
Obviously
$$
\mathcal M_{k+1,\ell}(X,L,\mathcal J;\beta)
\subset
\mathcal X_{k+1,\ell}(X,L;\beta;P).
$$
\begin{lem}
There is a partial topology
of the pair
$$
(\mathcal X_{k+1,\ell}(X,L;\beta;P),
\mathcal M_{k+1,\ell}(X,L,\mathcal J;\beta)).
$$
\end{lem}
\begin{proof}
We fix a metric on $P$.
Let $(t,{\bf p}) \in \mathcal M_{k+1,\ell}(X,L,\mathcal J;\beta)$.
We put
$$
B_{\epsilon}(\mathcal X,(t,{\bf p}))
=
B_{\epsilon}(t,P)  \times  B_{\epsilon}(\mathcal X,{\bf p}).
$$
Here $B_{\epsilon}(\mathcal X,{\bf p})$ is the partial topology of
$(\mathcal X_{k+1,\ell}(X,L;\beta),
\mathcal M_{k+1,\ell}(X,L,J_t;\beta))$ and
$B_{\epsilon}(t,P)$ is the metric $\epsilon$ ball 
in $P$
centered at $t$.
(See \cite[Proposition 4.3]{diskconst1}.)
In the same way as \cite[Section 4]{diskconst1} we can prove that this satisfies the required
properties.
\end{proof}
\begin{defn}
{\it Obstruction bundle data} for
$\mathcal M_{k+1,\ell}(X,L,\mathcal J;\beta)$ 
assign $E_{(t,{\bf p})}(\frak t,{\bf x})
\subset C^2(\Sigma_{\bf x};u_{\bf x}^*TX \otimes \Lambda^{0,1})$
to each $(t,{\bf p}) \in \mathcal M_{k+1,\ell}(X,L,\mathcal J;\beta)$
and $(\frak t,{\bf x}) \in B_{\epsilon}(\mathcal X,(t,{\bf p};\beta))$
such that:
\begin{enumerate}
\item
$E_{(t,{\bf p})}(\frak t,{\bf x})$
is a finite dimensional linear subspace.
 The supports of its elements are away from nodal or marked points and the boundary.
\item {\bf (Smoothness)}
$E_{(t,{\bf p})}(\frak t,{\bf x})$ depends smoothly on $(\frak t,{\bf x})$
in the same sense as
\cite[Definition 8.7]{diskconst1}.
\item {\bf (Transversality)}
$E_{(t,{\bf p})}(\frak t,{\bf x})$ satisfies the transversality condition in the
same sense as \cite[Definition 5.5]{diskconst1}.
Note this condition concerns only the case $(t,{\bf p}) = (\frak t,{\bf x})$.
We use the almost complex structure $J_{t}$ to define the linearized
Cauchy-Riemann equation which appears in \cite[Definition 5.5]{diskconst1}.
\item {\bf (Semi-continuity)}
$E_{(t,{\bf p})}(\frak t,{\bf x})$ is semi-continuous on  $(t,{\bf p})$ in
the same sense as \cite[Definition 5.2]{diskconst1}.
\item {\bf (Invariance under extended automorphisms)}
$E_{(t,{\bf p})}(\frak t,{\bf x})$ is invariant under the extended automorphism group of
${\bf x}$ in the same sense as in \cite[Condition 5.6]{diskconst1}.
\item[(6)]{\bf (Effectivity)}
The action of ${\rm Aut}({\bf p})$ on $(D_{u_{\bf p}}\overline{\partial})^{-1}(E_{\bf p})/ {\frak {aut}} (\Sigma_{\bf p}, \vec z_{\bf p}, \vec {\frak z}_{\bf p})$ is effective.
\end{enumerate}
\end{defn}
\begin{lem}\label{lem919}
Obstruction bundle data of
$\mathcal M_{k+1,\ell}(X,L,\mathcal J;\beta)$
induce a Kuranishi structure on it,
which is independent of the choices
in the sense of germs.
\end{lem}
\begin{proof}
We define a Kuranishi neighborhood of $(t,{\bf p})$
(as a set) by
$$
U_{(t,{\bf p})} = \{(\frak t,{\bf x})
\in B_{\epsilon}(\mathcal X,(t,{\bf p}))
\mid \overline{\partial}_{J_{\frak t}} u_{\bf x} \in E_{(t,{\bf p})}(\frak t,{\bf x})\}.
$$
Here we use the almost complex structure $J_{\frak t}$ to define
$\overline{\partial}_{J_{\frak t}}$.
The rest of the proof is the same as the proof of \cite[Theorem 7.1]{diskconst1}.
\end{proof}
\begin{lem}
We consider the case $P = [1,2]$.
Let $t_0 = 1$ or $2$. The restriction of obstruction bundle data of
$\mathcal M_{k+1,\ell}(X,L,\mathcal J;\beta)$ to $t=t_0$
defines obstruction bundle data of
$\mathcal M_{k+1,\ell}(X,L,J_{t_0};\beta)$.
By \cite[Theorem 7.1]{diskconst1}, the Kuranishi structure induced from the restriction
 coincides with the
restriction of the Kuranishi structure of Lemma \ref{lem919}
to the corresponding component of the normalized boundary.
\end{lem}
The proof is obvious from construction.
The moduli space
$\mathcal M_{k+1,\ell}(X,L,\mathcal J;\beta)$ comes
with evaluation maps:
$$
{\rm ev} = ({\rm ev}_0,\dots,{\rm ev}_{k},{\rm ev}^{\rm int}_1,\dots,
{\rm ev}^{\rm int}_{\ell}):
\mathcal M_{k+1,\ell}(X,L,\mathcal J;\beta)
\to L^{k+1} \times X^{\ell}.
$$
There is also an evaluation map which assigns the $P$ factor.
We denote it by
$$
{\rm ev}_{P} : \mathcal M_{k+1,\ell}(X,L,\mathcal J;\beta)  \to P.
$$
\begin{defn}\label{defcollared}
We consider the case $P = [1,2]$.
Suppose that there exists $\epsilon > 0$ such that
$J_{t} = J_1$ for $t \in [1,1+\epsilon]$
and $J_{t} = J_2$ for $t \in [2-\epsilon,2]$.
(We say $\mathcal J$ is collared if this condition is satisfied.) 
We say the obstruction bundle data are {\it collared} if:
\begin{enumerate}
\item
$E_{(t,{\bf p})}(\frak t,{\bf x}) = E_{(1,{\bf p})}(1,{\bf x})$
for $t,\frak t \in [1,1+\epsilon]$.
\item
$E_{(t,{\bf p})}(\frak t,{\bf x}) = E_{(2,{\bf p})}(2,{\bf x})$
for $t,\frak t \in [2-\epsilon,2]$.
\end{enumerate}
\end{defn}
\begin{lem}\label{lem922}
We consider the case $P = [1,2]$.
If the obstruction bundle data are collared, then the
next isomorphisms hold as isomorphisms of K-spaces.
$$
\aligned
&(\mathcal M_{k+1,\ell}(X,L,\mathcal J;\beta;[1,2])\vert_{{\rm ev}_P^{-1}([1,1+\epsilon])}
\cong [1,1+\epsilon] \times \mathcal M_{k+1,\ell}(X,L,J_1;\beta) \\
&(\mathcal M_{k+1,\ell}(X,L,\mathcal J;\beta;[1,2])\vert_{{\rm ev}_P^{-1}([2-\epsilon,2])}
\cong [2-\epsilon,2] \times \mathcal M_{k+1,\ell}(X,L,J_2;\beta)
\endaligned
$$
\end{lem}
The proof is obvious from construction.
\par
We next discuss the family version of disk-component-wiseness.
We use the evaluation maps $({\rm ev}_{P},{\rm ev}_i)$
and $({\rm ev}_{P},{\rm ev}_0)$ to $P \times L$ to define the
next fiber product.
$$
\mathcal M_{k_1+1,\ell_1}(X,L,\mathcal J;\beta_1) {}_{({\rm ev}_{P},{\rm ev}_i)} \times_{({\rm ev}_{P},{\rm ev}_0)} \mathcal M_{k_2+1,\ell_2}(X,L,\mathcal J;\beta_2).
$$
We call it the parameter-wise fiber product.
Let ${\frak T} \in \mathcal G(k+1,\ell,\beta)$. Using 
the parameter-wise fiber product
we can modify Definition \ref{defn3131} in an obvious way to define:
$$
\prod^{\rm{pw}}_{(\mathcal T,\beta(\cdot),l(\cdot))} {\mathcal M}_{k_{\rm v}+1,
\#l({\rm v})}(X,L,\mathcal J;\beta({\rm v})).
$$
We also denote this space  by
\begin{equation}\label{form9898}
\mathcal M_{k+1,\ell}(X,L,\mathcal J;\beta)(\frak T).
\end{equation}
In the same way as in Lemma \ref{lem3232} this space is
embedded into $\mathcal M_{k+1,\ell}(X,L,\mathcal J;\beta)$
as one of the components of normalized corners.
A component of normalized corners of $\mathcal M_{k+1,\ell}(X,L,\mathcal J;\beta;[1,2])$
is of the form
$$
\mathcal M_{k+1,\ell}(X,L,\mathcal J\vert_{\widehat S_{m'}P};\beta)(\frak T).
$$
Here $\mathcal J\vert_{\widehat S_{m'}P}$ is the restriction of $\mathcal J$ to
$\widehat S_{m'}P$ in an
obvious sense.
We can define a stratification
$$
\widehat S_{m'}P \times
\mathcal X_{k+1,\ell}(X,L;\beta)(\frak T)
$$
of the ambient set in the same way.
We can show the parametrized version of Lemma \ref{lem393939}
in the same way.
Now the parametrized version of Definition \ref{defn5151} is as follows.
\begin{defn}\label{defn5151para}
Suppose we are given obstruction bundle data $\{E_{(t,{\bf p})}(\frak t,{\bf x})\}$
for each $\mathcal M_{k+1,\ell}(X,L,\mathcal J;\beta)$.
We say that they form a {\it disk-component-wise system of obstruction bundle data}
of $\{\mathcal M_{k+1,\ell}(X,L,\mathcal J;\beta) ~\vert~ k,\ell,\beta\}$
if the following holds.
\par
Let
$(t,{\bf p}) \in \mathcal M_{k+1,\ell}(X,L,\mathcal J;\beta)$ and
$(t,{\bf p}) = (t,{\bf p}_{{\rm v}})_{{\rm v} \in C_{0}^{{\rm int}}(\mathcal T)} \in \mathcal M_{k+1,\ell}(X,L,J;\beta)(\frak T)$.
Then for sufficiently small neighborhoods
$$
\mathscr U_{(t,{\bf p})}
\subset {\mathcal X}_{k+1,\ell}(X,L;\beta;[1,2]),
\quad
\mathscr U_{(t,{\bf p}_{\rm v})}
\subset {\mathcal X}_{k_{\rm v}+1,l({\rm v})}(X,L;\beta_{\rm v};P)
$$
with
$$
\prod^{\rm{pw}}_{(\mathcal T,\beta(\cdot),l(\cdot))} \mathscr U_{(t,{\bf p}_{\rm v})}
\subseteq
\mathscr U_{(t,{\bf p})},
$$
the equality
\begin{equation}\label{form41}
E_{(t,{\bf p})}(\frak t,{\bf x}) = \bigoplus_{{\rm v} \in C_{0}^{{\rm int}}(\mathcal T)}
E_{(t,{\bf p}_{\rm v})}(\frak t,{\bf x}_{\rm v})
\end{equation}
holds, where $(\frak t,{\bf x}) = (\frak t,{\bf x}_{{\rm v}})_{{\rm v}
\in C_{0}^{{\rm int}}(\mathcal T)}$
is an arbitrary element of $\prod^{\rm{pw}}_{(\mathcal T,\beta(\cdot),l(\cdot))} \mathscr U_{(t,{\bf p}_{\rm v})}$.
\end{defn}
\begin{lem}
Suppose $\{E_{(t,{\bf p})}(\frak t,{\bf x})\}$
is a disk-component-wise system of obstruction bundle data.
It induces a Kuranishi structure on ${\mathcal M}_{k+1,\ell}(X,L,\mathcal J;\beta)$
for each $k,\ell,\beta$ by Lemma \ref{lem919}.
\begin{enumerate}
\item
The conclusion of Theorem \ref{therem274} holds.
\item
If $P = [1,2]$, they define a pseudo-isotopy of tree-like K-systems with
interior marked points.
\item
If $P = [1,2]$ and the obstruction bundle data are collared,
then the isomorphisms in Lemma \ref{lem922} commute with
various isomorphisms appearing in the definition of pseudo-isotopy.
\end{enumerate}
\end{lem}
The proof is the same as the proof of Theorem \ref{thm42}.
\begin{prop}\label{prop925}
There exists a disk-component-wise system of obstruction bundle data of
$\{{\mathcal M}_{k+1,\ell}(X,L,\mathcal J;\beta) \mid k,\ell,\beta\}$.
\par
If $P = [1,2]$ and $\mathcal J$ is collared, we may choose
the pseudo-isotopy to be collared.
\end{prop}
\begin{proof}
The proof is mostly the same as that of Theorem \ref{thm43}.
Recalling Situations \ref{situ69} and \ref{situ69rev} together with
Lemma \ref{lem710},
let $\frak T_1,\dots, \frak T_n, \frak S_1,\dots, \frak S_n$
be sequences of marked decorated rooted ribbon trees
such that
$$\frak T_i \supseteq \frak S_i, \quad
\frak S_i > \frak T_{i+1},
$$
and
$$
\aligned
\hat{{\bf q}}_i & = (t_i,{\bf q}_i) \in \mathcal M_{k_i+1,\ell_i}(X,L,\mathcal J;\beta_i)
(\frak T_i),\\
\hat{\bf p} & = (t_0,{\bf p}) \in \mathcal M_{k+1,\ell}(X,L,\mathcal J;\beta),\\
\hat{\frak p} & = (t_{\frak p},{\frak p})
\in \mathcal M^{\circ}_{k'+1,\ell'}(X,L,\mathcal J;\beta').
\endaligned
$$
We will fix a finite subset $\hat{\frak P}(k',\ell',\beta')$ of $\mathcal M^{\circ}_{k'+1,\ell'}(X,L,\mathcal J;\beta')$.
We fix type I strong stabilization and trivialization data, and an obstruction
space at ${\frak p}$
for each $\hat{\frak p} = (t_{\frak p},{\frak p}) \in \hat{\frak P}(k',\ell',\beta')$
and denote it by $E_{\hat{\frak p}}$.
We assume
$$
\aligned
d((t_0,{\bf p}),(t_1,{\bf q}_1)) &< \epsilon(k,\ell,\beta), \\
d((t_j,{\bf q}_{j,\frak S_j}), (t_{j+1},{\bf q}_{j+1})) &< \epsilon(k_{j+1},\ell_{j+1},\beta_{j+1}), \\
d((t_n,{\bf q}_{n,\frak S_n}), \hat{\frak p}) &< \epsilon(k',\ell',\beta'), \\
(\frak t,{\bf x}) & \in B_{\epsilon_0}(\mathcal X,(t_0,{\bf p})).
\endaligned
$$
We may choose $\epsilon(k_{j+1},\ell_{j+1},\beta_{j+1})>0$
small so that this condition induces smooth embeddings:
$$
\widehat{\Phi}_{(\hat{\bf x},\hat{\bf p},\vec{\hat{\bf q}},\hat{\frak p})} : {\rm Supp}(E_{\hat{\frak p}}) \to \Sigma_{{\bf x}},
\quad
\widehat{\Phi}_{(\hat{\bf p},\vec{\hat{\bf q}},,\hat{\frak p})} : {\rm Supp}(E_{\hat{\frak p}}) \to \Sigma_{{\bf p}}.
$$
Then we use them to obtain
$
E_{(\hat{\bf p},\vec{\hat{\bf q}},\hat{\frak p})}(\hat{\bf x})
$
as in (\ref{form6969rev}).
\par
We call such system $(\hat{{\bf p}},\vec{\hat{{\bf q}}},\hat{\frak p};\vec{\frak T},\vec{\frak S})$
a {\it quasi-splitting sequence}.
We fix type I strong stabilization data at each $\frak p$.
Then we obtain additional interior marked points
$
\vec{\frak w}_{\hat{\bf p},\vec{\hat{\bf q}},\hat{\frak p}}
$
on $\Sigma_{\bf p}$ in the same way as in Lemma \ref{lem61060}.
We say two quasi-splitting sequences are {\it equivalent}
if
$
\vec{\frak w}_{\hat{\bf p},\vec{\hat{\bf q}},\hat{\frak p}}
$
coincides.
An equivalence class of quasi-splitting sequences is called a
{\it quasi-component} and is written as $(\hat{\bf p},\hat\xi,\hat{\frak p})$.
$E_{(\hat{\bf p},\vec{\hat{\bf q}},\hat{\frak p})}(\hat{\bf x})$ depends
only on the equivalence class of
$(\hat{{\bf p}},\vec{\hat{{\bf q}}},\hat{\frak p};\vec{\frak T},\vec{\frak S})$
and we write it as
$
E_{(\hat{\bf p},\vec{\xi},\hat{\frak p})}(\hat{\bf x}).
$
\par
Now let $\widehat{\mathscr{QC}}_{k+1,\ell}(X,L,{\mathcal J};\beta)$ be the set of all
quasi-components. We define a map
$$
\widehat\Pi : \widehat{\mathscr{QC}}_{k+1,\ell}(X,L,{\mathcal J};\beta)
\to {\mathcal M}_{k+1,\ell}(X,L,\mathcal J;\beta)
$$
by $\widehat\Pi(\hat{\bf p},\hat\xi,\hat{\frak p}) = \hat{\bf p}$.
In the same way as in Lemma \ref{lem87} we can show that
each fiber of $\widehat\Pi$ is a finite set.
In the same way as Definition \ref{defn89}
we can define a Hausdorff topology on $\widehat{\mathscr{QC}}_{k+1,\ell}(X,L,{\mathcal J};\beta)$
such that $\hat\Pi$ is a local homeomorphism.
\par
We define an open subset $\widehat{\mathscr F}^{\circ}_{k+1,\ell,\beta}$
of  $\widehat{\mathscr{QC}}_{k+1,\ell}(X,L,{\mathcal J};\beta)$
and a closed subset
$\widehat{\mathscr F}_{k+1,\ell,\beta}$ of
$\widehat{\mathscr{QC}}_{k+1,\ell}(X,L,{\mathcal J};\beta)$
so that it contains the closure of $\widehat{\mathscr F}^{\circ}_{k+1,\ell,\beta}$ and
the restriction of $\widehat\Pi$ to $\widehat{\mathscr F}_{k+1,\ell,\beta}$ is proper.
We require conditions for them similar to Conditions \ref{cond717}
and \ref{cond718}.
Actually, the construction is by induction on $(k,\ell,\beta)$
and is the same as the proof of Proposition \ref{prop722}.
Now we put
\begin{equation}\label{form712revrefv}
E_{\hat{\bf p};\widehat{\mathscr F}}({\bf x}) =
\bigoplus_{(\hat{\bf p},\hat\xi,\hat{\frak p}) \in \widehat{\mathscr F}_{k+1,\ell,\beta}({\bf p})}
E_{(\hat{\bf p},\hat\xi,\hat{\frak p})}({\bf p})
\subset C^{\infty}(\Sigma_{\bf x};u_{\bf x}TX \otimes \Lambda^{0,1}).
\end{equation}
In the same way as in Proposition \ref{prop721}
we can show that it defines a disk-component-wise system of obstruction bundle data of
$\{{\mathcal M}_{k+1,\ell}(X,L,\mathcal J;\beta)~\vert~ k,\ell,\beta\}$.
\par
In the case $P = [1,2]$ and $\mathcal J$ is collared, it is easy to see that we can choose
$\widehat{\mathscr F}_{k+1,\ell,\beta}$ so that (\ref{form712revrefv})
is collared in the sense of Definition \ref{defcollared}. The proof of Proposition \ref{prop925} is complete.
\end{proof}
Now we can prove Theorem  \ref{thm219}.
Suppose we are in the situation of Theorem \ref{pisotopyexixsts}.
Then by Proposition \ref{prop925} we can find disk-component-wise systems of obstruction bundle data
of $\{\mathcal M_{k+1,\ell}(X,L,J_1;\beta)\}$ and $\{\mathcal M_{k+1,\ell}(X,L,J_2;\beta)\}$
so that there exists a pseudo-isotopy
between two   tree-like K-systems obtained by them.
We then use Theorem \ref{thm96} to prove that
there exists a pseudo-isotopy
between two  tree-like K-systems obtained by
any choice of obstruction bundle data.
Since our pseudo isotopies are collared, we can glue them easily.
Thus Theorem \ref{thm219} is proved.
Theorem \ref{pisotopyexixsts} is its special case where $\ell = 0$.
\qed

\bibliographystyle{amsalpha}

\begin{thebibliography}{FOOOXX}

\bibitem[DF]{DF} A. Daemi and K. Fukaya,
{\em Monotone Lagrangian Floer theory in smooth divisor complement: III},
submitted, arXiv:2211.02095.
\bibitem[Fu]{FuFu6}
K. {Fukaya},
\emph{Unobstructed immersed Lagrangian correspondence and
filtered $A_{\infty}$ functor}, submitted, arXiv:1706.02131.


\bibitem[FOOO1]{fooobook} K. Fukaya, Y.-G. Oh, H. Ohta and K. Ono,
{\em Lagrangian intersection Floer theory-anomaly and
obstruction, Part I,} AMS/IP Studies in Advanced Math. vol. 46.1, International Press/
Amer. Math. Soc. (2009). MR2553465.

\bibitem[FOOO2]{fooobook2} K. Fukaya, Y.-G. Oh, H. Ohta and K. Ono,
{\em Lagrangian intersection Floer theory-anomaly and
obstruction, Part II,} AMS/IP Studies in Advanced Math. vol. 46.2, International Press/
Amer. Math. Soc. (2009). MR2548482.

\bibitem[FOOO3]{foootoric32}
K. Fukaya, Y.-G. Oh, H. Ohta and K. Ono,
{\em Floer theory and mirror symmetry on compact toric manifolds},  Ast\'erisque, No. 376, (2016),
arXiv:1009.1648v2.

\bibitem[FOOO4]{foootech}
K. Fukaya, Y.-G. Oh, H. Ohta and K. Ono,
{\em Technical details on Kuranishi structure and virtual fundamental chain},
arXiv:1209.4410.

\bibitem[FOOO5]{foootech2}
K. Fukaya, Y.-G. Oh, H. Ohta and K. Ono,
{\em Kuranishi structure,
Pseudo-holomorphic curve,
and Virtual fundamental chain; Part 1}, 
arXiv:1503.07631.

\bibitem[FOOO6]{foootech21}
K. Fukaya, Y.-G. Oh, H. Ohta and K. Ono,
{\em Kuranishi structure,
Pseudo-holomorphic curve,
and Virtual fundamental chain; Part 2}, 
arXiv:1704.01848.

\bibitem[FOOO7]{foooexp}
K. Fukaya, Y.-G. Oh, H. Ohta and K. Ono,
{\em Exponential decay estimates and smoothness of the moduli space of pseudoholomorphic curves},
to appear in Mem. Amer. Math. Soc.,
arXiv:1603.07026.

\bibitem[FOOO8]{diskconst1}
K. Fukaya, Y.-G. Oh, H. Ohta and K. Ono,
{\em Construction of Kuranishi structures
on the moduli spaces of pseudo-holomorphic disks: I},
Surveys in Differential Geom. 22 (2018) 133--190,
arXiv:1710.01459.

\bibitem[FOOO9]{Springer}
K. Fukaya, Y.-G. Oh, H. Ohta and K. Ono,
{\em Kuranishi structures and Virtual fundamental chains},
Springer Monograph in Math. (2020) Springer, Singapore.

\bibitem[FOn]{FO}
K. Fukaya and K. Ono,
{\em  Arnold conjecture and Gromov-Witten invariant},
Topology 38 (1999), no. 5, 933--1048.

\bibitem[S]{Si} J. C. Sikorav, {\em Some properties of holomorphic curves in almost complex manifolds}, 165--189
in
`Holomorphic Curves in Symplectic Geometry',
ed. M. Audin and J. Lafontaine, Birkh\"auser, 1994.


\end{thebibliography}

\end{document}